\documentclass[a4paper]{amsart}
\usepackage{amsmath,amssymb,amsfonts,amsthm,enumerate}
\usepackage{mathrsfs}
\newtheorem{theorem}{Theorem}[section]
\newtheorem{lemma}[theorem]{Lemma}
\newtheorem{proposition}[theorem]{Proposition}
\newtheorem{corollary}[theorem]{Corollary}
\theoremstyle{definition}
\newtheorem{definition}[theorem]{Definition}
\theoremstyle{remark}
\newtheorem{example}[theorem]{Example}
\numberwithin{equation}{section}

\newcommand{\mat}{}

\let\Oldepsilon\epsilon
\let\Oldvarepsilon\varepsilon
  \renewcommand{\varepsilon}{\Oldepsilon}
  \renewcommand{\epsilon}{\Oldvarepsilon}
\let\Oldphi\phi
\let\Oldvarphi\varphi
  \renewcommand{\varphi}{\Oldphi}
  \renewcommand{\phi}{\Oldvarphi}

\newcounter{lastenumi}

\newcommand{\sfrac}[2]{#1 / #2}
\newcommand\llangle{\langle\!\langle}
\newcommand\rrangle{\rangle\!\rangle}
\newcommand{\shiftleft}{\!\!\!\!\!\!}

\newcommand{\Ext}{\mathsf{\Lambda}}

\newcommand{\parto}{\dashrightarrow}
\newcommand{\qquant}{\qquad\forall\ }
\newcommand{\quant}{\quad\forall\ }
\newcommand{\leftopenint}{\left]}
\newcommand{\rightopenint}{\right[}
\newcommand{\leftclosedint}{\left[}
\newcommand{\rightclosedint}{\right]}
\newcommand{\Symbol}[1]{#1{^{\sigma}}}
\newcommand{\zeroplus}{0}
\newcommand{\tc}{\,:\,}
\newcommand{\dbar}{\overline{\partial}}
\newcommand{\afterline}[1]{#1\overline{\phantom{i}}}

\newcommand{\R}{\mathbb{R}}
\newcommand{\C}{\mathbb{C}}
\newcommand{\N}{\mathbb{N}}
\newcommand{\Z}{\mathbb{Z}}

\newcommand{\Q}{\mathbb{Q}}

\newcommand{\bundle}{\mathscr}

\newcommand{\Epsilon}{E}
\newcommand{\Alpha}{A}
\newcommand{\Beta}{B}
\newcommand{\vectau}{{\boldsymbol{\tau}}}
\newcommand{\veceps}{{\boldsymbol{\epsilon}}}
\newcommand{\veczeroplus}{\boldsymbol{0}}

\newcommand{\loc}{\mathrm{loc}}
\newcommand{\comp}{\mathrm{c}}
\newcommand{\op}{\mathrm{op}}

\DeclareMathOperator{\supp}{supp}
\DeclareMathOperator{\Hom}{Hom}

\DeclareMathOperator{\Div}{div}
\DeclareMathOperator{\sinc}{sinc}
\DeclareMathOperator{\codim}{codim}
\DeclareMathOperator{\Span}{span}
\DeclareMathOperator*{\esssup}{ess\,sup}

\newcommand{\dist}{\varrho}
\newcommand{\flow}{\mathrm{flow}}
\newcommand{\subunit}{\mathrm{unit}}
\newcommand{\Diff}{\mathfrak{D}}
\newcommand{\mult}{m}

\DeclareSymbolFont{t1letters}{T1}{cmr}{m}{it}
\DeclareMathSymbol{\DD}{0}{t1letters}{208}

\newcommand{\E}{\mathrm{e}}
\newcommand{\I}{\mathrm{i}}
\newcommand{\D}{\mathrm{d}}
\newcommand{\runinhead}{\subsubsection}

\begin{document}

\title[Geometry and Finite Propagation]{Sub-Finsler geometry and \\finite propagation speed}
\author{Michael G.~Cowling}
\address{Michael G.~Cowling \\ School of Mathematics and Statistics \\ University of New South Wales \\ UNSW Sydney 2052 \\ Australia}
\email{m.cowling@unsw.edu.au}

\author{Alessio Martini}
\address{Alessio Martini \\ School of Mathematics and Statistics \\ University of New South Wales \\ UNSW Sydney 2052 \\ Australia}
\curraddr{Mathematisches Seminar \\ Christian-Albrechts-Universit\"at zu Kiel \\ Ludewig-Meyn-Str.\ 4 \\ D-24098 Kiel \\ Germany}
\email{martini@math.uni-kiel.de}

\thanks{This research was supported under the Australian Research Council's Discovery Projects funding scheme (project number DP-110102488).}

\begin{abstract}
We prove a number of results on the geometry associated to the solutions of evolution equations given by first-order differential operators on manifolds.
In particular, we consider distance functions associated to a first-order operator, and discuss the associated geometry, which is sometimes surprisingly different to riemannian geometry.
\end{abstract}

\maketitle

\section{Introduction}

Suppose that $D$ is a first-order, formally self-adjoint differential operator on a manifold $M$.
Under what circumstances can we define a group of operators $\E^{\I tD}$ (where $t\in \R$) and when can we say that solutions to the corresponding differential equation $(\partial_t  - \I D)u = 0$ propagate with finite speed?
If $D$ is an operator between vector bundles, how do we measure the speed?
The aim of this paper is to answer these questions, under some assumptions on $D$, which are related to the H\"ormander condition for families of vector fields.
We do this precisely, and while many of the ideas here are in the literature, we have not seen them put together in a coherent way as we do here.

In particular, we establish when formally self-adjoint operators are essentially self-adjoint, and produce sharp estimates for the propagation of solutions, which involve a ``sub-Finsler'' distance when the operators act between vector bundles.
We also give a detailed description of the associated geometry.

Every first-order differential operator $D$ between vector bundles has a symbol $\sigma(D)$, which maps the cotangent space at a point $x$ to the space of linear operators from the fibre of one vector bundle to another.
The mapping that sends a cotangent vector $\xi$ to the operator norm of $\sigma(D)(\xi)$ is thus a seminorm $P_x$ on the cotangent space $T^*_xM$.
When $D$ is elliptic, the seminorm $P_x$ is a norm at each point $x$, but when $D$ is not elliptic, the seminorm may well have a nontrivial kernel, and the dimension of this kernel may vary from point to point.
Dual to the seminorm on the cotangent space, there is an extended norm $P_x^*$ on the tangent space $T_xM$ (by ``extended norm'', we mean that some vectors may have infinite norm).
The annihilator of the kernel of the seminorm $P_x$ in the tangent space is the space of tangent vectors of finite norm.
Thus in general, the geometry that we consider is similar to subriemannian geometry, but we must allow for the possibility that the dimension of the space of vectors of finite norm is not constant.
Further, when the bundles are one-dimensional, the seminorm is euclidean (once the kernel is factored out), but when the bundles are higher-dimensional, the norm is more general.
Thus we consider ``sub-Finsler'' geometry, an extension of subriemannian geometry.
We define various natural distance functions, and show that under various hypotheses they coincide; but surprisingly, they do not always do so, and we give a number of examples that show that  results that are obvious in more restricted circumstances may in fact be false in our more general context.
For example, we show that it may not be possible to measure the length of a smooth curve by considering a smooth parametrisation, and that the ``right'' distance to measure propagation may not be euclidean.
Because ``obvious'' results may be false, we feel that we are justified in giving fairly complete proofs of most results; expert readers may skip over proofs, in the knowledge that they are the proofs that may be expected, but we do suggest looking at the counterexamples later in the paper.

As we commented above, most of the ideas that we consider are not new, but have been considered in less general contexts.
For example, our technique for establishing finite propagation speed for first-order operators is well-known in the elliptic context, but less so in general; there are, for instance, a number of proofs in the subelliptic context that consider elliptic approximants to subelliptic operators rather than working directly with subelliptic operators.
Some of the analysis of distance functions that we carry out is familiar in the context of ``metric spaces'', but those who work in the context do not seem usually to consider vector bundles.

It is important for us to work in the generality of vector bundles, as we need to work with self-adjoint operators.
Given complex vector bundles $\bundle{E}$ and $\bundle{F}$, with hermitean fibre inner products (inner products on each fibre), and a differential operator $D: C^\infty(\bundle{E}) \to C^\infty(\bundle{F})$, we define a new differential operator $\DD:  C^\infty(\bundle{E} \oplus \bundle{F}) \to C^\infty(\bundle{E} \oplus \bundle{F})$ as the sum of $D$ and its formal adjoint $D^+$: more precisely, $\DD(f,g) = (D^+g, Df)$.
Then $\DD$ is formally self-adjoint, and $\DD$ induces the same distance function as $D$.
By studying the propagation of solutions to $(\partial_t  - \I \DD)u = 0$, we can say something about the wave equation $(\partial_t ^2 -  D^+D)v = 0$.
Vector bundles are also a natural context for considering systems of vector fields: to $\{X_1, \dots, X_r\}$, we associate the differential operator sending a function $f$ to the vector-valued function $(X_1f, \dots, X_rf)$, that is, from a section of a trivial bundle with fibre $\C$ to a section of a trivial bundle with fibre $\C^r$.

\subsection{Notation and Background}
Throughout, $M$ is an $n$-dimensional manifold, by which we mean a smooth $\sigma$-compact, and hence para\-compact, manifold without boundary.
Then $M$ admits a countable locally finite atlas $(\phi_\alpha)_{\alpha\in\Alpha}$; here each $U_\alpha \subseteq M$ and each $\phi_\alpha\colon  U_\alpha \to \R^n$ is a smooth bijection with smooth inverse.
By choosing a partition of unity $(\eta_\alpha)_{\alpha \in \Alpha}$ subordinate to the cover $(U_\alpha)_{\alpha \in \Alpha}$ and then rescaling the $\phi_\alpha$ so that $\phi_\alpha(\supp (\eta_\alpha)) \subseteq B_{\R^n}(0,1)$, where $B_{\R^n}(x,r)$ denotes the open ball in $\R^n$ with centre $x$ and radius $r$, we may suppose that $\bigcup_{\alpha \in \Alpha} V_\alpha = M$, where $V_\alpha = \phi_\alpha^{-1}(B_{\R^n}(0,1))$.
Then $\sum_{\alpha} \eta_\alpha = 1$ and the $\eta_\alpha$ are bump functions on $M$, by which we mean smooth compactly-supported functions taking values in $\leftclosedint 0,1 \rightclosedint$.
We write $\mathfrak{O}(M)$ and $\mathfrak{K}(M)$, or just $\mathfrak{O}$ and $\mathfrak{K}$, for the collections of all open subsets and all compact subsets of $M$.

We will endow $M$, and subsets thereof, with various extended distance functions $\dist:M \times M \to [0,\infty]$; by this, we mean that $\dist$ satisfies the usual conditions for a distance function, but may take the value $\infty$.
One way to do this is to choose a continuous ``fibre seminorm'' $P$ on $T^*M$, that is, $P_x$ is a seminorm on each fibre $T^*_xM$, and $P: T^*M \to [0,\infty)$ is continuous.
Dually, there is an extended fibre norm $P^*$ on the tangent space $TM$, given by
\[
P^*_x(v) = \sup_{\substack{\xi \in T^*_x M \\ P(\xi) \leq 1}} |\xi(v)| \,.
\]
We then say that a curve $\gamma:[a,b]\to M$ is subunit if is is absolutely continuous and $P^*(\gamma') \leq 1$ almost everywhere in $[a,b]$.
We define the (possibly infinite) distance $\dist_P(x,y)$ between points $x$ and $y$ in $M$ to be the infimum of the set of lengths of the intervals of definition of subunit curves starting at $x$ and ending at $y$.
We consider both subunit and smooth subunit curves in the text, and show, under suitable hypotheses, that it does not matter which are used, but in general there is a distinction.
It is easier to work with $P$ rather than $P^*$, as describing the continuity requirements on $P^*$ is more complex; further, when $P$ and $P^*$ arise in the analysis of a first-order differential operator, $P$ has a simple description in terms of the symbol of the operator.

In general, the topology induced by $\dist_P$ may not be equivalent to the original manifold topology of $M$.
It is easier to work with distance functions that do give rise to the original topology, and we give these a special name.

\begin{definition}\label{def:varietal}
An extended distance function is said to be \emph{varietal} if the topology that it induces coincides with the manifold topology. 
\end{definition}

Given a distance function $\dist$ on $M$, a point $x$ in $M$, and $\epsilon \in \R^+$, we write
\[
B_{\dist}(x,\epsilon) = \{y \in M \tc \dist(x,y) < \epsilon\}
\quad\text{and}\quad
\bar{B}_{\dist}(x,\epsilon) = \{y \in M \tc \dist(x,y) \leq \epsilon\};
\]
the latter set need not be closed in the manifold topology, and, given a subset $X$ of $M$, we write $\afterline{X}$ for the manifold closure of $X$.
As usual, $\dist(X,x) = \inf_{y \in X} \dist(y,x)$.
We define $B_{\dist}(X,\epsilon)$ and $\bar{B}_{\dist}(X,\epsilon)$ analogously.

We equip $M$ with a smooth measure that is equivalent to Lebesgue measure in all coordinate charts, and write ${\D}x$, ${\D}y$, \dots, for the measure elements.
Take a smooth complex finite-rank fibre-normed vector bundle $\bundle{E}$ on $M$.
We use ``function-like notation'' for spaces of sections of $\bundle{E}$; for instance, we write $L^p_\loc(\bundle{E})$ for the space of (equivalence classes of) sections $f$ of $\bundle{E}$ such that $|f|^p$ is locally integrable on $M$ if $p < \infty$, or $|f|$ is essentially bounded if $p = \infty$, and $L^p_\comp(\bundle{E})$ for the space of compactly-supported sections in $L^p_\loc(\bundle{E})$.
The former space is equipped with a Fr\'echet structure: $f_m \to f$ in $L^p_\loc(\bundle{E})$ if and only if
\[
\int_K |f_m(x) - f(x)|^p \,{\D}x \to 0 
\quad\text{as $m \to \infty$}
\]
for all $K \in \mathfrak{K}(M)$ (recall that, in general, a Fr\'{e}chet space structure involves a countable family of seminorms $Q_k$ such that $f = 0 $ if and only if $Q_k(f) = 0$ for all indices $k$); the latter is an inductive limit of Banach spaces.
We write $C(\bundle{E})$ for the space of  continuous sections of $\bundle{E}$; then convergence in $C(\bundle{E})$ means uniform convergence on compacta.
If $\bundle{E}$ has a hermitean fibre inner product $\langle\cdot,\cdot\rangle$, then, for all $f,g \in L^2(\bundle{E})$, we write $\langle f, g \rangle$ for their pointwise inner product, which is a function on $M$, and $\llangle f, g\rrangle$ for their inner product:
\[
\llangle f, g\rrangle = \int_M \langle f(x), g(x) \rangle \,{\D}x.
\]
We write $\bundle{T}$ and $\bundle{T}^r$ for the trivial bundles over $M$ with fibres $\C$ and $\C^r$, and $\bundle{T}_\R$ for the trivial bundle over $M$ with fibre $\R$.
Thus $C^\infty_\comp(\bundle{T})$ and $C^\infty_\comp(\bundle{T}_\R)$ denote the usual space of smooth compactly-supported complex-valued functions on $M$, and the subspace thereof of real-valued functions.

Suppose that $\phi_\alpha\colon  {U_\alpha} \to \R^n$ is a coordinate chart and $\bundle{E}$ is a vector bundle over $M$ with fibre $\C^r$.
On $\R^n$, as on any contractible manifold, all vector bundles are trivialisable \cite[Corollary 3.4.8]{husemoller_fibre_1994}.
Thus, when we consider the restriction $\bundle{E}|_{U_\alpha}$ of $\bundle{E}$ to ${U_\alpha}$, there are invertible linear maps $T_x$ from $\bundle{E}_x$, the fibre over $x$, to $\C^r$, which vary smoothly with $x$ in $M$, so the map $w\mapsto (\pi(w), T_{\pi(w)} w)$, where $\pi$ is the projection from $\bundle{E}$ to $M$, is a vector bundle isomorphism of $\bundle{E}|_{U_\alpha}$ with the bundle ${U_\alpha} \times \C^r$ over $U_\alpha$.
In fact, when $\bundle{E}$ has a hermitean structure, then the $T_x$ may be chosen to be isometries.
Furthermore, the map $\phi_\alpha \otimes I$ is a vector bundle isomorphism from the bundle ${U_\alpha} \times \C^r$ over $U_\alpha$ to the bundle $\R^n \times \C^r$ over $\R^n$.
This isomorphism in turn induces an identification $\tau_{\bundle{E},\alpha}$ of the sections of $\bundle{E}|_{U_\alpha}$ with the sections of the trivial bundle $\R^n \times \C^r$ over $\R^n$, which we identify with the $\C^r$-valued functions on $\R^n$.
For instance, $\tau_{\bundle{E},\alpha}: C^\infty_\comp(\bundle{E}|_{U_\alpha}) \to C^\infty_\comp(\R^n \times \C^r)$ is defined by $\tau_{\bundle{E},\alpha} f(x) = T_{\phi_\alpha^{-1}(x)}  f(\phi_\alpha^{-1}(x))$ for all $x$ in $\R^n$.
At the risk of confusion, we usually just write $\tau_\alpha$ rather than $\tau_{\bundle{E},\alpha}$.
We also use $\tau$ for the map of other spaces of sections, such as $L^1_\loc(\bundle{E})$.
When we write $\tau_\alpha^{-1}f$, where $f$ is a section over $\R^n$, we intend the section of $\bundle{E}$ that vanishes outside $U_\alpha$.

We use the letter $\kappa$ for constants; these may vary from one paragraph to the next. 
We often highlight the parameters on which these constants depend.

\section{Differential operators and symbols}

We denote by $\Diff_k(\bundle{E},\bundle{F})$ the space of smooth linear $k$th-order differential operators from $C^\infty(\bundle{E})$ to $C^\infty(\bundle{F})$, where $\bundle{E}$ and $\bundle{F}$ are smooth complex finite-rank vector bundles on $M$. 
In local coordinates and trivialisations of the bundles, as described above, each $D \in \Diff_k(\bundle{E},\bundle{F})$ may be written as
\begin{equation}\label{eq:diffopcoord}
\tau_\alpha (Df)(x) = \sum_{|J|\leq k} \mat{a}_J(x) \, \partial_J (\tau_\alpha f )(x)
\qquant x \in \R^n,
\end{equation}
where the $J$ are multi-indices and the coefficients $\mat{a}_J(x)$ are matrices that depend smoothly on $x$ in $\R^n$.
We also write
\[
\tau_\alpha D = \sum_{|J|\leq k} \mat{a}_J \, \partial_J .
\]
Note that $\Diff_{k_1}(\bundle{E},\bundle{F}) \subseteq \Diff_{k_2}(\bundle{E},\bundle{F})$ if $k_1 \leq k_2$.

Every $D \in \Diff_k(\bundle{E},\bundle{F})$ has an associated symbol $\sigma_{k}(D)$, which is a smooth section of $\Hom(S^k(\C T^* M), \Hom(\bundle{E},\bundle{F}))$;
in other words, the symbol $\sigma_k(D)$ at a point $x \in M$ is a $\Hom(\bundle{E}_x,\bundle{F}_x)$-valued symmetric $k$-linear form on $\C T^*_x M$.
In local coordinates and trivialisations, if $D$ is given by \eqref{eq:diffopcoord}, then
\begin{equation}\label{eq:symbolcoord}
\tau_\alpha(\sigma_k(D))(x)(\xi^{\odot k}) = \sum_{|J| = k} \xi^J \mat{a}_J(x)
\qquant x \in \R^n
\quant \xi \in \C^n
\end{equation}
where $\xi^{\odot k}$ denotes the symmetrised version of $\xi\otimes \dots \otimes \xi$
(with $k$ factors). 
The mapping $D \mapsto \sigma_k(D)$ is $\C$-linear, and its kernel is $\Diff_{k-1}(\bundle{E},\bundle{F})$;
moreover, if $D_1 \in \Diff_{k_1}(\bundle{E},\bundle{F})$ and $D_2 \in \Diff_{k_2}(\bundle{F},\bundle{G})$, where $\bundle{G}$ is another vector bundle on $M$, then $D_2 D_1 \in \Diff_{k_1+k_2}(\bundle{E},\bundle{G})$ and
\begin{equation}\label{eq:symbcomp}
  \sigma_{k_1+k_2}(D_2 D_1)(\xi^{\odot (k_1+k_2)})
= \sigma_{k_2}(D_2)(\xi^{\odot k_2}) \, \sigma_{k_1}(D_1)(\xi^{\odot k_1})
\end{equation}
for all $\xi \in \C T^* M$.

Recall that $M$ is endowed with a smooth measure that is equivalent to Lebesgue measure in all coordinate charts, and suppose that $\bundle{E}$ and $\bundle{F}$ are endowed with hermitean fibre inner products.
Then each $D \in \Diff_k(\bundle{E},\bundle{F})$ has a formal adjoint $D^+ \in \Diff_k(\bundle{F},\bundle{E})$, which is uniquely determined by the identity
\begin{equation}\label{eq:adjoint}
\llangle D f, g \rrangle = \llangle f, D^+ g \rrangle
\end{equation}
for all $f \in C^\infty_\comp(\bundle{E})$ and $g \in C^\infty_\comp(\bundle{F})$.
This identity extends to sections $f$ and $g$ such that $\supp f \cap\supp g$ is compact, since $\llangle Df, g \rrangle = \llangle D(\eta f), \eta g \rrangle$ for all bump functions $\eta$ equal to $1$ on $\supp f \cap \supp g$.
Clearly, the mapping $D \mapsto D^+$ is conjugate-linear and $(D_2 D_1)^+ = D_1^+ D_2^+$.
Moreover, for all $\theta \in S^k(\C T^* M)$, 
\begin{equation}\label{eq:symbadj}
\sigma_k(D^+)(\theta) = (-1)^k (\sigma_k(D)(\overline{\theta}))^*
\end{equation}
where the final ${}^*$ denotes the adjoint with respect to the hermitean inner products along the fibres of $\bundle{E}$ and $\bundle{F}$; note that the symbol of the formal adjoint does not depend on the choice of measure on $M$.

\subsection{Zeroth-order differential operators}\label{subsection:zerothorder}

Every $D \in \Diff_0(\bundle{E},\bundle{F})$ is a multiplication operator: it is given by multiplication by a smooth section of $\Hom(\bundle{E},\bundle{F})$, namely, the symbol $\sigma_0(D)$.
Formal adjunction of $D$ then corresponds to pointwise adjunction of the multiplier:
\begin{equation}\label{eqn:mult_adjoint}
\llangle h f, g \rrangle = \llangle f, h^* g \rrangle
\end{equation}
for all $h \in C^\infty(\Hom(\bundle{E},\bundle{F}))$, all $f \in C^\infty(\bundle{E})$, and all $g \in C^\infty(\bundle{F})$ such that $\supp f \cap \supp g \cap \supp h$ is compact.
Here are some special cases of \eqref{eqn:mult_adjoint}.

First, if $\bundle{E} = \bundle{F}$ and $h \in C^\infty(\bundle{T})$, then $h$ corresponds to a scalar section of $\Hom(\bundle{E},\bundle{E})$, whose pointwise adjoint corresponds to the pointwise conjugate $\overline{h}$, so
\[
\llangle h f, g \rrangle = \llangle f, \overline{h} g \rrangle.
\]

Next, if $h,g \in C^\infty(\bundle{E})$ and $f \in C^\infty(\bundle{T})$, then $h$ corresponds to a smooth section of $\Hom(\bundle{T}, \bundle{E})$, whose pointwise adjoint corresponds to the section $h^* = \langle \cdot, h \rangle$ of $\bundle{E}^*$, which we may identify with $\Hom(\bundle{E}, \bundle{T})$; now
\[
\llangle f h, g \rrangle = \llangle f, h^* g \rrangle = \llangle f, \langle g, h \rangle \rrangle.
\]

Finally, if $h \in C^\infty(\bundle{E})$, $f \in C^\infty(\Hom(\bundle{E},\bundle{F}))$ and $g \in C^\infty(\bundle{F})$, then $h$ corresponds to a smooth section of $\Hom(\Hom(\bundle{E},\bundle{F}),\bundle{F})$, whose pointwise adjoint, with respect to the Hilbert--Schmidt inner product on $\Hom(\bundle{E},\bundle{F})$, is a section of $\Hom(\bundle{F},\Hom(\bundle{E},\bundle{F}))$, given, modulo the identification of $\Hom(\bundle{E},\bundle{F})$ with $\bundle{E}^* \otimes \bundle{F}$, by the pointwise tensor product with $h^*$, and
\[
\llangle f h, g \rrangle = \llangle f, h^* \otimes g \rrangle.
\]

By the way, by using a partition of unity and local trivialisations, it is easily shown that each smooth compactly-supported section $h$ of $\Hom(\bundle{E},\bundle{F})$ may be written as a finite sum of sections of the form $f^* \otimes g$ for appropriate $f \in C^\infty_\comp(\bundle{E})$ and $g \in C^\infty_\comp(\bundle{F})$.

\subsection{First-order differential operators}\label{subsection:firstorder}

Suppose that $D \in \Diff_1(\bundle{E},\bundle{F})$.
Given any $h \in C^\infty(\bundle{T})$, denote by $\mult_{\bundle{E}}(h)$ and $\mult_{\bundle{F}}(h)$ the multiplication operators $f \mapsto h f$ on smooth sections of $\bundle{E}$ and $\bundle{F}$, and define
\begin{equation}\label{eq:symbolcommutator}
[D,\mult(h)] = D \mult_{\bundle{E}}(h) - \mult_{\bundle{F}}(h) D.
\end{equation}
In local coordinates and trivialisations, if $D$ is given by \eqref{eq:diffopcoord}, then
\[
\tau_\alpha([D,\mult(h)] f)(x) = \sum_{j = 1}^n \partial_j (\tau_\alpha h)(x) \, \mat{a}_j(x) \,\tau_\alpha f(x)
\qquant x \in \R^n;
\]
in other words, the commutator $[D,\mult(h)]\colon C^\infty(\bundle{E}) \to C^\infty(\bundle{F})$ acts by multiplication by $\sigma_1(D)({\D}h) \in C^\infty(\Hom(\bundle{E},\bundle{F}))$.
Observe that the correspondence $h \mapsto \sigma_1(D)({\D}h)$ is a differential operator $\Symbol{D} \in \Diff_1(\bundle{T},\Hom(\bundle{E},\bundle{F}))$, given in local coordinates by
\[
\tau_\alpha(\Symbol{D} h)(x) = \sum_{j = 1}^n \partial_j (\tau_\alpha h)(x) \, \mat{a}_j(x)
\qquant x \in \R^n.
\]
Clearly $\Symbol{D}$ is homogeneous, that is, $\Symbol{D} 1 = 0$, and the map $D \mapsto \Symbol{D}$ is linear.
Moreover \eqref{eq:symbolcommutator} may be rewritten as Leibniz' rule for $D$, that is,
\[
D(h f) = (\Symbol{D} h) f + h \, Df
\]
for all $f \in C^\infty(\bundle{E})$ and $h \in C^\infty(\bundle{T})$.
This identity, together with \eqref{eq:adjoint} and its zeroth-order instances discussed in \S~\ref{subsection:zerothorder}, easily implies that
\begin{equation}\label{eq:adjtau}
(\Symbol{D})^+ (f^* \otimes g) = \langle D^+ g, f \rangle - \langle g, Df \rangle
\end{equation}
for all $f \in C^\infty(\bundle{E})$ and $g \in C^\infty(\bundle{F})$, whereas from \eqref{eq:symbadj} it follows that
\[
\Symbol{(D^+)} h = - (\Symbol{D} \overline{h})^*
\]
for all $h \in C^\infty(\bundle{T})$.

For more on differential operators, see \cite[Section IV]{palais_seminar_1965}, \cite[Section 2.1]{berline_heat_1992}, and \cite[Section IV.2]{wells_differential_2008}; see also \cite[Section 10]{higson_analytic_2000} for the first-order case.

\section{Distributions and weak differentiability}\label{section:weakderivatives}

Recall that $C^\infty_\comp(\bundle{E})$ denotes the LF-space of compactly-supported smooth sections of $\bundle{E}$; its conjugate dual $C^\infty_\comp(\bundle{E})'$ is the space of $\bundle{E}$-valued distributions on $M$.
As usual, we identify a locally integrable section $f \in L^1_\loc(\bundle{E})$ with the distribution $\phi \mapsto \llangle f, \phi \rrangle$ and extend the inner product between sections of $\bundle{E}$ to denote the duality pairing between $C^\infty_\comp(\bundle{E})'$ and $C^\infty_\comp(\bundle{E})$.

Every differential operator $D \in \Diff_k(\bundle{E},\bundle{F})$ then extends to an operator on distributions: given any $u \in C^\infty_\comp(\bundle{E})'$, we define $D u \in C^\infty_\comp(\bundle{F})'$ by
\begin{equation}\label{eqn:defn-distn}
\llangle Du , \phi \rrangle = \llangle u, D^+ \phi \rrangle
\qquant \phi \in C^\infty_\comp(\bundle{F}).
\end{equation}
Since zeroth-order differential operators are multiplication operators, \eqref{eqn:defn-distn} includes the definition of the ``pointwise product'' of smooth sections and distributions, in all the variants discussed in \S~\ref{subsection:zerothorder}.
Moreover the identity
\[
\llangle u^* , \phi \rrangle = \afterline{\llangle u, \phi^* \rrangle}
\]
allows us to extend pointwise adjunction to $\Hom(\bundle{E},\bundle{F})$-valued distributions.

For a first-order operator $D$, with these definitions, we may extend the identities of \S~\ref{subsection:firstorder} to the realm of distributions.
For instance, to show that
\begin{equation}\label{eq:dist_tauadj}
\Symbol{(D^+)} h = - (\Symbol{D} \overline{h})^*
\end{equation}
for all $h \in C^\infty_\comp(\bundle{T})'$, we note that it suffices to test this distributional identity on sections of the form $f^* \otimes g$ where $f \in C^\infty_\comp(\bundle{E})$ and $g \in C^\infty_\comp(\bundle{F})$; to do this, we apply \eqref{eq:adjtau}.
Similarly it may be proved that
\begin{equation}\label{eq:dist_leibniz}
D(h f) = (\Symbol{D} h) f + h \, Df
\end{equation}
when $h \in C^\infty_\comp(\bundle{T})'$ and $f \in C^\infty(\bundle{E})$, or when $h \in C^\infty(\bundle{T})$ and $f \in C^\infty_\comp(\bundle{E})'$.

For $D \in \Diff_k(\bundle{E},\bundle{F})$, the definition of the $D$-derivative of an $\bundle{E}$-valued distribution is based on that of the formal adjoint $D^+$ and depends on the choice of measure on $M$ and on the hermitean structures on $\bundle{E}$ and $\bundle{F}$; the same holds for the definition of the embedding of $L^1_\loc(\bundle{E})$ in $C^\infty_\comp(\bundle{E})'$.
However, $L^1_\loc(\bundle{E})$ and $L^1_\loc(\bundle{F})$ do not depend on those structures: if we change the measure or inner products, then we get the same linear spaces, with equivalent families of seminorms and so equivalent Fr\'echet structures.
Moreover, if $f \in L^1_\loc(\bundle{E})$ and $Df \in L^1_\loc(\bundle{E})$, then the section in $L^1_\loc(\bundle{F})$ that corresponds to the distributional derivative $Df$ does not depend on these structures.

We say that $f \in L^1_\loc(\bundle{E})$ is weakly $D$-differentiable if $D f \in L^1_\loc(\bundle{F})$.
Given any $p \in [1 , \infty]$ and $D \in \Diff_k(\bundle{E},\bundle{F})$, we define the local Sobolev space $W^p_{D, \loc}(\bundle{E})$ by
\[
W^p_{D, \loc}(\bundle{E}) = \{f \in L^p_\loc(\bundle{E}) \tc Df \in L^p_\loc(\bundle{F})\},
\]
which is given a Fr\'echet structure by identifying it with a closed subspace of $L^p_\loc(\bundle{E}) \times L^p_\loc(\bundle{F})$ by the map $f \mapsto (f,Df)$.
Similarly, we define the Sobolev space $W^p_{D}(\bundle{E})$ by
\[
W^p_{D}(\bundle{E}) = \{f \in L^p(\bundle{E}) \tc Df \in L^p(\bundle{F})\}.
\]
The Banach space $W^p_{D}(\bundle{E})$ depends on the choice of measure on $M$ and on the hermitean structures on $\bundle{E}$ and $\bundle{F}$, while $W^p_{D,\loc}(\bundle{E})$ does not.
Finally, $W^p_{D,0}(\bundle{E})$ denotes the closure of $C^\infty_c(\bundle{E})$ in $W^p_D(\bundle{E})$.

\subsection{Mollifiers and smooth approximation}

Mollifiers, introduced by K.~O.\ Friedrichs \cite{friedrichs_identity_1944}, allow us to approximate distributions, and in particular, locally integrable functions, by smooth functions.
We now describe the application of this technique to sections of vector bundles on the manifold $M$.

For convenience, we first consider the case where $M$ is $\R^n$, equipped with Lebesgue measure and euclidean distance function, and $\bundle{T}$ is the trivial bundle $\R^n \times \C$ over $\R^n$.
Recall that all vector bundles on $\R^n$ are trivialisable, and sections of a trivial bundle over $\R^n$ with fibre $\C^r$ may be identified with functions from $\R^n$ to $\C^r$.
Hence it is easy to define mollifiers on $\R^n$ globally, and mollifiers on general manifolds and bundles may then be defined by local trivialisations and partitions of unity.

Choose a bump function $\phi \in C^\infty_\comp(\bundle{T})$ with unit mass and support in the unit ball; for all $\epsilon \in \leftopenint 0,1\rightclosedint$, define $\phi_\epsilon(x) = \epsilon^{-n} \phi(\epsilon^{-1}x)$ for all $x \in \R^n$.
For a distributional section $f \in C^\infty_\comp(\bundle{T}^r)'$ of a trivial bundle with fibre $\C^r$, we set
\begin{equation}\label{eqn:defn-of-J}
J_\epsilon f(x) = \phi_\epsilon * f(x) = \sum_{k=1}^r \llangle f, \phi_\epsilon(x-\cdot) e_k \rrangle e_k
\qquant x \in \R^n,
\end{equation}
where $\{e_1,\dots,e_r\}$ is the canonical basis of $\C^r$.

\begin{proposition}\label{prp:mollifier}
Suppose that $\bundle{E}$ is the trivial bundle $\R^n \times \C^r$ over $\R^n$, and that $1 \leq p \leq \infty$.
For all $f \in C^\infty_\comp(\bundle{E})'$, the formula \eqref{eqn:defn-of-J} defines smooth sections $J_\epsilon f$  of $\bundle{E}$ that converge to $f$ distributionally as $\epsilon \to \zeroplus$.
Moreover, the following hold.
\begin{enumerate}[(i)]
\item (Supports)
$\supp J_\epsilon f \subseteq \bar{B}_{\R^n}(\supp f ,\epsilon)$.
\item (Equicontinuity)
The operators $J_\epsilon$ are bounded on $L^p_\loc(\bundle{E})$, uniformly for $\epsilon$ in $\leftopenint 0,1\rightclosedint$.
\item (Approximation)
If $p < \infty$ and $f \in L^p_\loc(\bundle{E})$, then $J_\epsilon f \to f$ in $L^p_\loc(\bundle{F})$ as $\epsilon \to \zeroplus$; the same holds if $p = \infty$ and $f \in C(\bundle{E})$.
\item (Upper bound)
For all continuous fibre seminorms $P$ on $\bundle{E}$, all $K \in \mathfrak{K}(\R^n)$ and all $f \in L^\infty_\loc(\bundle{E})$,
\[
\limsup_{\epsilon \to \zeroplus} \sup_{x\in K} P(J_\epsilon f)(x) \leq \inf_{\substack{W \in \mathfrak{O} \\ W \supseteq K }} \esssup_{x \in W} P(f)(x).
\]
\end{enumerate}
\end{proposition}
\begin{proof}
These are well-known facts about convolution and approximate identities in $\R^n$, and we omit the proofs, except for part (iv).

For all $x \in \R^n$ and $\epsilon \in \R^+$,
\[
J_\epsilon f(x) = \int_{\R^n} \phi_\epsilon(y) f(x-y) \,{\D}y.
\]
Since the functions $\phi_\epsilon$ are nonnegative and have unit mass, while $P_x\colon \bundle{E}_x \to \R$ is convex, Jensen's inequality implies that
\[
P_x(J_\epsilon f(x)) \leq \int_{\R^n} \phi_\epsilon(y) \, P_x(f(x-y)) \,{\D}y,
\]
whence
\begin{equation}\label{eq:seminormmajorisation}
\begin{aligned}
P(J_\epsilon f)(x)
&\leq \int_{\R^n} \phi_\epsilon(y) \, P(f)(x-y) \,{\D}y \\
&\qquad + \int_{\R^n} \phi_\epsilon(y) \, (P_x(f(x-y)) - P_{x-y}(f(x-y))) \,{\D}y.
\end{aligned}
\end{equation}
Suppose now that $K \subseteq W$, where $K\in \mathfrak{K}(\R^n)$ and $W \in \mathfrak{O}(\R^n)$.
Take $\bar\epsilon \in \R^+$ such that $\bar{B}_{\R^n}(K,\bar\epsilon) \subseteq W$.
Since $f \in L^\infty_\loc(\bundle{E})$,
\[
\esssup_{x \in \bar{B}_{\R^n}(K ,\bar\epsilon)} |f(x)| < \infty;
\]
further, $P\colon \bundle{E} \to \R$ is continuous, so uniformly continuous when restricted to $\bar{B}_{\R^n}(K,\bar\epsilon) \times \{v \in \C^r \tc |v| \leq R\}$, for all $R \in \R^+$.
Thus the second integral on the right-hand side of \eqref{eq:seminormmajorisation} tends to $0$ as $\epsilon \to \zeroplus$, while the first integral is bounded by $\esssup_{x\in W} P(f)(x)$ when $\epsilon \leq \bar\epsilon$.
Part (iv) follows.
\end{proof}

The interaction of mollifiers and differentiation is more interesting: for a differential operator $D \in \Diff_k(\bundle{E},\bundle{F})$, it is reasonable to ask whether $D J_\epsilon f$ converges to $D f$ as $\epsilon \to \zeroplus$.
When we are working on $\R^n$ with trivial bundles $\bundle{E}$ and $\bundle{F}$,  we already know that $J_\epsilon Df$ approximates $Df$.
In this case, the problem reduces to the study of the commutator operators $[D,J_\epsilon]$, given by
\[
[D, J_\epsilon] f = D J_\epsilon f - J_\epsilon D f.
\]
If $D$ is translation-invariant, then $[D,J_\epsilon] = 0$.
For an arbitrary $D$, it is clear that $[D, J_\epsilon] f \to 0$ distributionally as $\epsilon \to \zeroplus$.
Stronger forms of convergence to $0$ may be proved easily for first-order operators $D$.

\begin{proposition}\label{prp:mollifiercommutator}
Suppose that $\bundle{E}$ and $\bundle{F}$ are the trivial bundles $\R^n \times \C^r$ and $\R^n \times \C^s$ over $\R^n$, that $D \in \Diff_1(\bundle{E},\bundle{F})$, and that $1 \leq p \leq \infty$.
Then the following hold.
\begin{enumerate}[(i)]
\item (Equicontinuity)
The operators $[D, J_\epsilon]$ are bounded from $L^p_\loc(\bundle{E})$ to $L^p_\loc(\bundle{F})$, uniformly  for $\epsilon$ in $\leftopenint 0, 1\rightclosedint$.
\item (Vanishing)
If $p < \infty$ and $f \in L^p_\loc(\bundle{E})$, then $[D, J_\epsilon] f \to 0$ in $L^p_\loc(\bundle{F})$  as $\epsilon \to \zeroplus$; the same holds if $p = \infty$ and $f \in C(\bundle{E})$.
\end{enumerate}
\end{proposition}
\begin{proof}
Compare with \cite[Appendix]{garofalo_isoperimetric_1996}.

We may suppose that $D$ has the form
\[
Df(x) = \mat{b}(x) f(x) + \sum_{j=1}^n \mat{a}_{j}(x) \partial_j f(x) = D_0f(x)  + D_1 f(x),
\]
say, where the matrix-valued functions $\mat{a}_{j}$ and $\mat{b}$ are smooth.

Since $D_0$ is a multiplication operator, it is bounded from $L^p_\loc(\bundle{E})$ to $L^p_\loc(\bundle{F})$ for all $p \in \leftclosedint 1,\infty\rightclosedint$.
Hence, by Proposition~\ref{prp:mollifier}, both $D_0 J_\epsilon$ and $J_\epsilon D_0$ are bounded from $L^p_\loc(\bundle{E})$ to $L^p_\loc(\bundle{F})$, uniformly for $\epsilon$ in $\leftopenint 0,1\rightclosedint$, and
\[
D_0 J_\epsilon f \to D_0 f
\quad\text{and}\quad
J_\epsilon D_0 f \to D_0 f
\]
in $L^p_\loc(\bundle{F})$ as $\epsilon \to \zeroplus$ if either $p < \infty$ and $f \in L^p_\loc$, or $p = \infty$ and $f \in C$.
Hence parts (i) and (ii) hold for $[D_0,J_\epsilon]$, and it suffices to consider $D_1$.

If $f \in C^\infty_\comp(\bundle{E})$, then
\[
J_\epsilon  f(x)  =  \int_{\R^n}  f (x-y) \phi_\epsilon(y) \,{\D}y,
\]
so
\[
\begin{aligned}
\lbrack D_1, J_{\epsilon}\rbrack f(x)
&=  D_1 J_\epsilon f(x)  - J_\epsilon D_1 f(x)  \\
&=   \int_{\R^n} \Bigl(  \sum_{j=1}^n \bigl[\mat{a}_{j}(x) - \mat{a}_{j}(x-y)\bigr] \partial_j  f(x-y) \Bigr)  \phi_\epsilon(y) \,{\D}y \\
&= C_\epsilon f(x),
\end{aligned}
\]
say.

Define
\[
F_{\epsilon}(x,y) =  \sum_{j=1}^n \Bigl( \phi_\epsilon(y) \partial_j \mat{a}_{j}(x-y) + \bigl[\mat{a}_{j}(x) - \mat{a}_{j}(x-y)\bigr] \partial_j\phi_\epsilon(y) \Bigr)  ,
\]
and observe that, when $f \in C^\infty_\comp(\bundle{E})$ and $\lambda \in \C$,
\[
\begin{aligned}
C_\epsilon f(x)
&=   \int_{\R^n} \Bigl( \sum_{j=1}^n \bigl[\mat{a}_{j}(x) - \mat{a}_{j}(y)\bigr] \partial_j  f(y) \Bigr)  \phi_\epsilon(x - y) \,{\D}y\\
&=   -\int_{\R^n}  \Bigl( \sum_{j=1}^n \frac{\partial}{\partial y_j} \Bigl[ \bigl[\mat{a}_{j}(x) - \mat{a}_{j}(y)\bigr] \phi_\epsilon(x - y)\Bigr] \Bigr)   f(y) \,{\D}y\\
&=   -\int_{\R^n}  \Bigl( \sum_{j=1}^n \frac{\partial}{\partial y_j} \Bigl[ \bigl[\mat{a}_{j}(x) - \mat{a}_{j}(y)\bigr] \phi_\epsilon(x - y)\Bigr] \Bigr)   [f(y)-\lambda f(x)] \,{\D}y \\
&=   \int_{\R^n}  F_\epsilon(x,y) [f(x- y) - \lambda f(x)] \,{\D}y.
\end{aligned}
\]
This formula extends to all $f \in C^\infty_\comp(\bundle{E})'$ by continuity, since $F_\epsilon$ is smooth and supported in the set $\{(x,y) \in \R^n \times \R^n \tc |y| \leq \epsilon \}$.

For $x \in \R^n$, define the quadrilinear form $A(x)$ on $\C^n \times \C^n \times \C^r \times \C^s$ by
\[
A(x)(u',u,v,w) = \sum_{j,k=1}^n u_j' \frac{\partial}{\partial x_j} u_k(a_k(x) v,w),
\]
where $(w',w)$ denotes $\sum_{l=1}^s w'_l w_l$, and write $|A(x)|$ for the maximum of the expressions $|A(x)(u',u,v,w)|$ as $u'$, $u$, $v$ and $w$ range over the unit spheres in $\C^n$, $\C^n$, $\C^r$ and $\C^s$. Then for $v \in \C^r$ and $w \in \C^s$,
\[
\begin{aligned}
\Bigl| \Bigl( \sum_{j=1}^n \bigl[ a_j(x) - a_j(x-y) \bigr] \partial_j \phi_\epsilon(y) v, w \Bigr) \Bigr|
&=
\Bigl| \int_0^1 A(x-y+ty)(y,\nabla(\phi_\epsilon)(y),v,w) \,{\D}t \Bigr|
\\&\leq
\sup_{z \in B_{\R^n}(x,|y|)} \left|A(z)\right| \left|y\right| \left|\nabla(\phi_\epsilon)(y)\right| \left|v\right| \left|w\right|
\\&\leq
\sup_{z \in B_{\R^n}(x,|y|)} \left|A(z)\right| |\nabla\phi|_\epsilon(y) \left|v\right| \left|w\right|,
\end{aligned}
\]
where $|\nabla\phi|_\epsilon(y) = \epsilon^{-n} |\nabla\phi|(\epsilon^{-1} y)$, and similarly
\[
\begin{aligned}
\Bigl| \sum_{j=1}^n \phi_\epsilon(y) \partial_j a_j(x-y) v \Bigr|
&= |\phi_\epsilon(y)| \Bigl| \sum_{j=1}^n A(x-y)(e_j,e_j,v,w) \Bigr|
\\& \leq n \left|\phi_\epsilon(y)\right| \sup_{z \in B_{\R^n}(x,|y|)} \left|A(z)\right| \left|v\right| \left|w\right|,
\end{aligned}
\]
where the $e_j$ are the standard basis vectors in $\R^n$.

Set $\psi = n\phi + |\nabla \phi|$ and $\psi_\epsilon(z) = \epsilon^{-n} \psi(\epsilon^{-1} z)$; then, taking operator norms,
\[
\begin{aligned}
|F_{\epsilon}(x,y)|
&\leq \Bigl| \sum_{j=1}^n \bigl[\partial_j \mat{a}_j(x-y)\bigr]  \phi_\epsilon(y) \Bigr|
   + \Bigl| \sum_{j=1}^n \Bigl[ \bigl[\mat{a}_{j}(x-y) - \mat{a}_{j}(x)\bigr] \partial_j \phi_\epsilon(y)\Bigr] \Bigr| \\
&\leq \sup_{z\in B_{\R^n}(x,\epsilon)}  |A(z)| \, \psi_\epsilon(y).
\end{aligned}
\]
Now $\psi$ is continuous and supported in the unit ball, hence bounded, so
\[
\begin{aligned}
| C_\epsilon f(x) |
&\leq \left| \int_{\R^n} F_{\epsilon} (x,y) (f(x-y) - \lambda f(x) )\, {\D}y \right| \\
&\leq \sup_{z \in B_{\R^n}(x,\epsilon)} |A(z) | \int_{B_{\R^n}(0,\epsilon)} \psi_{\epsilon} (y)  \left|  f(x-y) - \lambda f(x) \right |\, {\D}y ,
\end{aligned}
\]
whence, from Minkowski's inequality, for all $K\in \mathfrak{K}(\R^n)$ and $\epsilon \in \leftopenint 0, 1\rightclosedint$,
\begin{equation*}
\begin{aligned}
&\shiftleft\left(\int_K |C_\epsilon f(x)|^p \,{\D}x\right)^{1/p} \\
&\leq \sup_{z \in B_{\R^n}(K,\epsilon)} |A(z) | \int_{B_{\R^n}(0,\epsilon)} \psi_{\epsilon} (y)  \left(\int_{K} |f(x-y) - \lambda f(x) |^p \,{\D}x\right)^{1/p} \,{\D}y\\
&\leq \kappa_{n,K,D,\phi} \sup_{y \in B_{\R^n}(0,1)} \left(\int_{K} |f(x-y) - \lambda f(x) |^p \,{\D}x\right)^{1/p} .
\end{aligned}
\end{equation*}
On the one hand, when $\lambda=1$, the integral on the right-hand side tends to $0$ as $y \to 0$, and we obtain part (ii) in the case where $p < \infty$.
On the other hand, when $\lambda = 0$, it follows that
\[
\begin{aligned}
\biggl(\int_K |C_\epsilon f(x)|^p \,{\D}x\biggr)^{1/p}
&\leq \kappa_{n,K,D,\phi}  \biggl(\int_{B_{\R^n}(K ,1)} |f(x)|^p \,{\D}x\biggr)^{1/p} .
\end{aligned}
\]
which establishes part (i) in the case where $p < \infty$.
To prove part (i) and part (ii) when $p = \infty$, we replace the $L^p$ norm in the argument above by an essential supremum.
\end{proof}

We consider now the general case.
Recall that there is a countable locally finite atlas of smooth bijections $\phi_\alpha\colon U_\alpha \to \R^n$ on $M$; here $U_\alpha \subseteq M$ and $\bigcup_{\alpha \in \Alpha} V_\alpha = M$, where $V_\alpha = \phi_\alpha^{-1}(B_{\R^n}(0,1))$.
There is also a partition of unity $(\eta_\alpha)_{\alpha \in \Alpha}$ on $M$ for which $\supp(\eta_\alpha) \subseteq V_\alpha$.

For each $\alpha$, there is a trivialisation $\tau_\alpha$ taking sections of $\bundle{E}$ over $U_\alpha$ to sections of a trivial bundle $\bundle{T}^r$ on $\R^n$, and similarly for $\bundle{F}$.
Sections of $\bundle{E}$ with support contained in $V_\alpha$ are then identified with sections of $\bundle{T}^r$ with support contained in the open unit ball.

Denote by $\Epsilon$ the set of all sequences $(\epsilon_\alpha)_{\alpha \in \Alpha}$, where each $\epsilon_\alpha \in \leftopenint 0,1\rightclosedint$, that is, $\Epsilon = \leftopenint 0,1\rightclosedint^\Alpha$.
For $\veceps \in \Epsilon$ and $f \in L^1_\loc(\bundle{E})$, define  $J^{\bundle{E}, \vectau}_{\veceps}f$, which we usually write as $J^{\vectau}_{\veceps}f$, as follows:
\begin{equation}\label{eq:sumdefinition}
J^\vectau_\veceps f = \sum_\alpha \tau_\alpha^{-1} J_{\epsilon_\alpha} \tau_\alpha(\eta_\alpha f);
\end{equation}
since $\supp \tau_\alpha^{-1} J_{\epsilon_\alpha} \tau_\alpha(\eta_\alpha f) \subseteq U_\alpha$, this sum is locally finite.

Given $\veceps,\veceps' \in \Epsilon$, we write $\veceps \leq \veceps'$ when $\epsilon_\alpha \leq \epsilon_\alpha'$ for all $\alpha \in \Alpha$.
The ordering of $\Epsilon$ gives a meaning to limit-like expressions along $\Epsilon$, such as
\[
\limsup_{\veceps \to \veczeroplus} F(\veceps) = \inf_{\bar\veceps \in \Epsilon} \sup_{\substack{\veceps \in \Epsilon \\ \veceps \leq \bar\veceps}} F(\veceps)
\]
for a function $F\colon \Epsilon \to \leftclosedint-\infty,\infty\rightclosedint$.
We show now that $J^\vectau_\veceps f \to f$ as $\veceps \to \veczeroplus$.

\begin{theorem}\label{thm:generalmollifier}
Suppose that $\bundle{E}$ and $\bundle{F}$ are vector bundles on a manifold $M$, that $D \in \Diff_1(\bundle{E},\bundle{F})$, and that $1 \leq p \leq \infty$.
Then the linear operators $J^\vectau_\veceps$ on $L^1_\loc(\bundle{E})$ defined by \eqref{eq:sumdefinition} have the following properties.
\begin{enumerate}[(i)]
\item (Smoothing) If $f \in L^1_\loc(\bundle{E})$, then $J^\vectau_\veceps f \in C^\infty(\bundle{E})$.
\item (Supports) If $C$ is a closed subset of $M$ and $W$ is an open neighbourhood of $C$, then there exists $\bar\veceps \in \Epsilon$ such that $\supp J^\vectau_\veceps f \subseteq W$ when $\supp f \subseteq C$ and $\veceps \leq \bar\veceps$.
\item (Equicontinuity) If $\zeta \in L^\infty_\loc(\bundle{T})$, then there exists $\xi \in L^\infty_\loc(\bundle{T})$ such that
\[
\|\zeta J^\vectau_\veceps f\|_p \leq \|\xi f\|_p
\]
for all $f \in L^p_\loc(\bundle{E})$ and $\veceps \in \Epsilon$; if $\zeta \in L^\infty_\comp$, then we may choose $\xi \in L^\infty_\comp$.
\item (Approximation) If $p < \infty$ and $f \in L^p_\loc(\bundle{E})$, then
\[
\limsup_{\veceps \to \veczeroplus} \| \zeta (J^\vectau_\veceps f - f) \|_p = 0
\]
for all $\zeta \in L^\infty_\loc(\bundle{T})$; the same holds if $p = \infty$ and $f \in C(\bundle{E})$.
\item (Upper bound) If $f \in L^\infty_\loc(\bundle{E})$, then, for each continuous fibre seminorm $P$ on $\bundle{E}$ and closed subset $C$ of $M$,
\begin{equation}\label{eq:seminormcontrol}
\limsup_{\veceps \to \veczeroplus} \sup_{x\in C} P(J^\vectau_\veceps f)(x) \leq \inf_{\substack{W \in \mathfrak{O}\\W \supseteq C}} \esssup_{x\in W} P(f)(x).
\end{equation}
\setcounter{lastenumi}{\value{enumi}}
\end{enumerate}
Further, the ``commutators'' $[D,J^{\vectau}_\veceps]$, defined by
\[
[D,J^{\vectau}_\veceps] f = D J^{\bundle{E},\vectau}_\veceps f - J^{\bundle{F},\vectau}_\veceps D f,
\]
have the following properties.
\begin{enumerate}[(i)]
\setcounter{enumi}{\value{lastenumi}}
\item (Equicontinuity) If $\zeta \in L^\infty_\loc(\bundle{T})$, then there exists $\xi \in L^\infty_\loc(\bundle{T})$ such that
\[
\|\zeta [D,J^\vectau_\veceps] f\|_p \leq \|\xi f\|_p
\]
for all $f \in L^p_\loc(\bundle{E})$ and $\veceps \in \Epsilon$; if $\zeta \in L^\infty_\comp$, then we may choose $\xi \in L^\infty_\comp$.
\item (Vanishing) If $f \in L^p_\loc(\bundle{E})$ and $p<\infty$, then
\[
\limsup_{\veceps \to \veczeroplus} \| \zeta [D, J^\vectau_\veceps] f \|_p = 0
\]
for all $\zeta \in L^\infty_\loc(\bundle{T})$; the same holds if $p = \infty$ and $f \in C(\bundle{E})$.
\setcounter{lastenumi}{\value{enumi}}
\end{enumerate}
As operators on $W^p_{D, \loc}(\bundle{E})$, the $J^\vectau_\veceps$ have the following properties.
\begin{enumerate}[(i)]
\setcounter{enumi}{\value{lastenumi}}
\item (Equicontinuity) If $\zeta \in L^\infty_\loc(\bundle{T})$, then there exists $\xi \in L^\infty_\loc(\bundle{T})$ such that
\[
\|\zeta D J^\vectau_\veceps f\|_p \leq \|\xi f\|_p + \|\xi  Df \|_p
\]
for all $f \in W^p_{D, \loc}(\bundle{E})$ and $\veceps \in \Epsilon$; if $\zeta \in L^\infty_\comp$, then we may choose $\xi \in L^\infty_\comp$.
\item (Approximation) If $f \in W^p_{D, \loc}(\bundle{E})$ and $p<\infty$, then
\[
\limsup_{\veceps \to \veczeroplus} \| \zeta (D J^\vectau_\veceps f - D f) \|_p = 0
\]
for all $\zeta \in L^\infty_\loc(\bundle{T})$; the same holds if $p = \infty$, $f \in C(\bundle{E})$, and $Df \in C(\bundle{F})$.
\item (Upper bound) If $f\in C(\bundle{E})$ and $Df \in L^\infty_\loc(\bundle{F})$, then, for each continuous fibre seminorm $P$ on $\bundle{F}$ and closed subset $C$ of $M$,
\[
\limsup_{\veceps \to \veczeroplus} \sup_{x\in C} P(DJ^\vectau_\veceps f)(x) \leq \inf_{\substack{W \in \mathfrak{O} \\ W \supseteq C}} \esssup_{x\in W} P(Df)(x).
\]
\end{enumerate}
\end{theorem}

The alert reader will have already noticed that, in contrast to Propositions~\ref{prp:mollifier} and \ref{prp:mollifiercommutator}, equicontinuity and the limiting properties here refer to a topology on the spaces of sections $L^p_\loc$ (and $C$), defined by the ``extended seminorms'' $f \mapsto \| \zeta f \|_p$, where $\zeta$ ranges over $L^\infty_\loc(\bundle{T})$, which is finer than the usual Fr\'echet topology when $M$ is not compact.
Indeed, if $M$ is not compact, then the finer topology, known as the Whitney topology (at least in the case of $C$ \cite[Chapter~2]{hirsch_differential_1994}), is not metrisable, nor does it yield a topological vector space structure: the mapping $\lambda \mapsto \lambda f$ is not continuous unless the section $f$ is compactly-supported.
However, like the Fr\'echet topology, the Whitney topology is independent of the measure on $M$ and the hermitean structure of the bundle.

Propositions~\ref{prp:mollifier} and \ref{prp:mollifiercommutator} cannot be strengthened by just replacing the Fr\'echet topology with the Whitney topology; the stronger approximation result of Theorem~\ref{thm:generalmollifier} is due to the fact that the approximant $J^\vectau_\veceps f$ depends on the sequence $\veceps \in \Epsilon$ whose components $\epsilon_\alpha$ may be chosen independently.

A propos of limits along $\Epsilon$, the following remark will be useful in the course of the proof: if $\{\Alpha_\beta\}_{\beta \in \Beta}$ is a collection of subsets of $\Alpha$ which is locally finite, in the sense that $\{ \beta \in \Beta \tc \alpha \in \Alpha_\beta \}$ is finite for all $\alpha \in \Alpha$, then
\begin{equation}\label{eq:scambiosuplimsup}
\limsup_{\veceps \to \veczeroplus} \sup_{\beta \in \Beta} F_\beta(\veceps|_{\Alpha_\beta}) = \sup_{\beta \in \Beta} \limsup_{\veceps \to \veczeroplus} F_\beta(\veceps|_{\Alpha_\beta})
\end{equation}
for all functions $F_\beta\colon \leftopenint0,1\rightclosedint^{\Alpha_\beta} \to \leftclosedint-\infty,\infty\rightclosedint$.

\begin{proof}
First, the sum defining $J^\vectau_\veceps$ is a locally finite sum of smooth, compactly-supported sections of $\bundle{E}$, so part (i) clearly holds.

Next, for all closed subsets $C$ of $M$, open neighbourhoods $W$ of $C$, and $\alpha \in \Alpha$, we may find $\bar{\epsilon}_\alpha \in \leftopenint 0,1\rightclosedint$ such that 
\[
\bar{B}_{\R^n} (\phi_\alpha( \supp \eta_\alpha \cap C) , \bar{\epsilon}_\alpha) \subseteq \phi_\alpha(W),
\]
and part (ii) follows from Proposition~\ref{prp:mollifier} (i).

Further, if $\zeta \in L^\infty_\loc(\bundle{T})$, then
\[
|\zeta J^\vectau_\veceps f| \leq \sum_\alpha \kappa_\alpha |J_{\epsilon_\alpha} \tau_\alpha (\eta_\alpha f)|
\]
pointwise almost everywhere, where the constants $\kappa_\alpha$ are independent of $f$.
We deduce from Proposition~\ref{prp:mollifier} (ii) that
\[
\|J_{\epsilon_\alpha} \tau_\alpha (\eta_\alpha f)\|_p \leq \kappa'_\alpha \|\tau_\alpha (\eta_\alpha f) \|_p
\]
and hence
\[
\|\zeta J^\vectau_\veceps f\|_p \leq \sum_\alpha \kappa''_\alpha \|\eta_\alpha f\|_p \leq \sup_{\alpha\in\Alpha} \kappa'''_\alpha \|\eta_\alpha f\|_p
\]
for new constants $\kappa'_\alpha$, $\kappa''_\alpha$ and $\kappa'''_\alpha$, and we take $\xi$ to be $\sup_{\alpha\in\Alpha} \kappa'''_\alpha \eta_\alpha$ to prove part (iii).
Analogously, one shows that
\[
\|\zeta (J^\vectau_\veceps f - f) \|_p \leq \sup_{\alpha\in\Alpha} \kappa''''_\alpha \|J_{\epsilon_\alpha} \tau_\alpha(\eta_\alpha f) - \tau_\alpha(\eta_\alpha f)\|_p
\]
for suitable constants $\kappa''''_\alpha$, whence
\[
\limsup_{\veceps \to \veczeroplus} \|\zeta (J^\vectau_\veceps f - f) \|_p \leq \sup_{\alpha\in\Alpha} \kappa_\alpha \limsup_{t \to \zeroplus} \|J_t \tau_\alpha(\eta_\alpha f) - \tau_\alpha(\eta_\alpha f)\|_p
\]
by \eqref{eq:scambiosuplimsup}, and part (iv) follows from Proposition~\ref{prp:mollifier} (iii).

Suppose now that $f \in L^\infty_\loc(\bundle{E})$ and $P$ is a continuous fibre seminorm on $\bundle{E}$.
Write $P_\alpha$ for the corresponding seminorm on the fibres of the trivial bundle $\R^n \times \C^r$ over $\R^n$; in other words, for a section $f$ of $\bundle{E}$ with support in $U_\alpha$, $P_\alpha(\tau_\alpha f)(\phi(x)) = P(f)(x)$ for all $x \in U_\alpha$.
Given any $\alpha \in \Alpha$, write $K_\alpha$ for $\phi_\alpha^{-1}(\bar{B}_{\R^n}(0,2))$, so $\supp \tau_\alpha^{-1} J_t \tau_\alpha (\eta f) \subseteq K_\alpha$ when $t \leq 1$, and given any $\beta \in \Alpha$, denote by $\Alpha_\beta$ the finite set of indices $\alpha$ in $\Alpha$ such that $K_\alpha \cap K_\beta \neq \emptyset$.

Fix $\delta \in \R^+$.
Given any $\beta \in \Alpha$, we may find a finite decomposition of $K_\beta$ as $K_{\beta,1} \cup \dots \cup K_{\beta,k_\beta}$, in which each $K_{\beta,j}$ is compact and the oscillation of $\eta_\alpha$ on an open neighbourhood $W_{\beta,j}$ of $K_{\beta,j}$ is bounded by $\delta/|\Alpha_\beta|$.
Take $y_{\beta,j}$ in $K_{\beta,j}$.
Then, by Proposition~\ref{prp:mollifier} (iv), for all closed subsets $C$ of $M$ and open neigh\-bour\-hoods $W$ of $C$,
\begin{align*}
\limsup_{\veceps \to \veczeroplus} \sup_{x \in K_{\beta,j} \cap C} P(J^\vectau_\veceps f)(x)
& \leq \sum_{\alpha \in \Alpha_\beta} \limsup_{t \to \zeroplus} \sup_{x \in K_{\beta,j} \cap C} P_\alpha(J_{t} \tau_\alpha (\eta_\alpha f))(\phi_\alpha(x)) \\
& \leq \sum_{\alpha \in \Alpha_\beta} \esssup_{x \in W_{\beta,j} \cap W} P_\alpha(\tau_\alpha (\eta_\alpha f))(\phi_\alpha(x)) \\
& \leq \sum_{\alpha \in \Alpha_\beta} \sup_{x \in W_{\beta,j}} \eta_\alpha(x) \esssup_{z \in W} P(f)(z) \\
& \leq \sum_{\alpha \in \Alpha_\beta}  (\eta_\alpha(y_{\beta,j}) + \delta/|\Alpha_\beta|) \esssup_{z \in W} P(f)(z) \\
& \leq (1 + \delta) \esssup_{z \in W} P(f)(z)
\end{align*}
for each $K_{\beta,j}$.
Since the restriction of $P(J^\vectau_\veceps f)$ to $K_{\beta,j}$ depends only on $\veceps|_{\Alpha_\beta}$, and the set $\{(\beta,j) \tc \alpha \in \Alpha_\beta\} = \bigcup_{\beta \in \Alpha_\alpha} \{\beta\} \times \{1,\dots,k_\beta\}$ is finite for all $\alpha \in \Alpha$,
\[
\begin{aligned}
\limsup_{\veceps \to \veczeroplus} \sup_{x\in C} P(J^\vectau_\veceps f)(x)
&= \sup_{\substack{\beta \in \Alpha \\ j = 1,\dots,k_\beta}} \limsup_{\veceps \to \veczeroplus} \sup_{x\in K_{\beta,j} \cap C} P(J^\vectau_\veceps f)(x)  \\
&\leq (1 + \delta) \esssup_{z\in W} P(f)(z),
\end{aligned}
\]
by \eqref{eq:scambiosuplimsup}, and part (v) follows from the arbitrariness of $\delta$ and $W$.

We now write $D \in \Diff_1(\bundle{E},\bundle{F})$ in local coordinates, and decompose $[D, J_\veceps^\vectau]$ as $I^1_\veceps + I^2_\veceps$, where
\[
\begin{aligned}
I^1_\veceps f &= \sum_\alpha \tau_\alpha^{-1} J_{\epsilon_\alpha} \tau_\alpha ((\Symbol{D} \eta_\alpha) f) \\
I^2_\veceps f &= \sum_\alpha \tau_\alpha^{-1} [\tau_\alpha (D), J_{\epsilon_\alpha}] \tau_\alpha (\eta_\alpha f).
\end{aligned}
\]
The properties (vi) and (vii) of $[D,J^\vectau_\veceps]$ follow from the analogous properties of $I^1_\veceps$ and $I^2_\veceps$, which in turn are obtained from Propositions~\ref{prp:mollifier} and \ref{prp:mollifiercommutator}, by arguing as in the proofs of parts (iii) and (iv) of this theorem and observing that
\[
\sum_\alpha (\Symbol{D} \eta_\alpha) f = \Bigl( \Symbol{D} \sum_\alpha \eta_\alpha \Bigr) f = 0.
\]

Finally, the decomposition
\[
D J^{\bundle{E},\vectau}_\veceps f = [D,J_\veceps^\vectau] f + J^{\bundle{F},\vectau}_\veceps Df
\]
shows that part (viii) follows from parts (iii) and (vi), while part (ix) follows from parts (iv) and (vii).
Moreover, given any continuous fibre seminorm $P$ on $\bundle{F}$,
\[
|P(D J^{\bundle{E},\vectau}_\veceps f) - P(J^{\bundle{F},\vectau}_\veceps Df)| \leq P([D,J^\vectau_\veceps] f),
\]
and, by part (vii), the right-hand side tends to $0$ uniformly as $\veceps \to \veczeroplus$ whenever $f$ is continuous; therefore, under our assumptions,
\[
\limsup_{\veceps \to \veczeroplus} \sup_{x \in C} P(D J^{\bundle{E},\vectau}_\veceps f)(x) = \limsup_{\veceps \to \veczeroplus} \sup_{x \in C} P(J^{\bundle{F},\vectau}_\veceps Df)
\]
for all subsets $C$ of $M$, and part (x) follows from part (v).
\end{proof}

Not only is the Whitney topology finer than the Fr\'echet space topology on $L^p_\loc$, but also, when restricted to $L^p$, it is finer than the usual Banach space topology of $L^p$.
Hence the following density result is an immediate consequence of Theorem~\ref{thm:generalmollifier}.

\begin{corollary}\label{cor:cssmoothapprox}
Suppose that $D \in \Diff_1(\bundle{E},\bundle{F})$ and $1 \leq p < \infty$.
Then $C^\infty_\comp(\bundle{E})$ is dense in $W^p_{D, \loc}(\bundle{E})$, and $W^p_D \cap C^\infty(\bundle{E})$ is dense in $W^p_{D}(\bundle{E})$.
\end{corollary}

An analogous result when $p = \infty$ may be obtained by restricting to continuous sections with continuous $D$-derivatives.
The following weaker result, however, does not require the continuity of the $D$-derivatives.

\begin{corollary}\label{cor:bdsmoothapprox}
Suppose that $D \in \Diff_1(\bundle{E},\bundle{F})$, $f \in C(\bundle{E})$,  $D f \in L^\infty_\loc(\bundle{F})$, and $\|P(Df)\|_\infty \leq 1$ for some continuous fibre seminorm $P$ on $\bundle{F}$.
Then there exists a sequence of $C^\infty(\bundle{E})$-sections $f_m$ that converges to $f$ uniformly on compacta, such that $\|P(Df_m)\|_\infty \leq 1$ for all $m$ and $\supp f_m \subseteq W$ for all open neighbourhoods $W$ of $\supp f$ once $m$ is large enough.
Moreover, if $\bundle{E} = \bundle{T}$ and $f$ is real-valued, then the $f_m$ may be chosen to be real-valued.
\end{corollary}
\begin{proof}
By parts (iv) and (x) of Theorem~\ref{thm:generalmollifier}, $\limsup_{\veceps \to \veczeroplus} \| J^\vectau_\veceps f - f \|_\infty = 0$ and
\[
\limsup_{\veceps \to \veczeroplus} \| P(D J^\vectau_\veceps f) \|_\infty \leq \|P (Df) \|_\infty \leq 1.
\]
We fix a decreasing countable base $\{W_m\}_{m \in \N}$ of open neighbourhoods of $\supp f$, and then choose, for all $m \in \N$, a sequence $\veceps$ in $\Epsilon$ such that the section $g_m = J^\vectau_\veceps f$ satisfies $\| g_m - f \|_\infty \leq 2^{-m}$ and $\| P(D g_m) \|_\infty \leq 1 + 2^{-m}$.
We may also assume that $\supp g_m \subseteq W_m$ by Theorem~\ref{thm:generalmollifier} (i).
The conclusion then follows by taking $f_m = (1+2^{-m})^{-1} g_m$.
\end{proof}

\subsection{Integration and differentiation}

Approximation using mollifiers allows us to extend results such as integration by parts, Leibniz' rule, and the chain rule to the realm of weakly differentiable sections (see, for instance, \cite[Chapter 7]{gilbarg_elliptic_2001}). In what follows, $p'$ denotes the index conjugate to $p$, that is, $1/p' + 1/p = 1$.

\begin{proposition}[Integration by parts]\label{prp:parts}
Suppose that $D \in \Diff_1(\bundle{E},\bundle{F})$ and that $1 \leq p \leq \infty$.
If $f \in W^p_{D, \loc}(\bundle{E})$ and $g \in W^{p'}_{D^+,\loc}(\bundle{F})$, 
and moreover $\supp(f \otimes g)$ is compact, then
\[
\llangle Df, g \rrangle = \llangle f, D^+ g \rrangle.
\]
\end{proposition}
\begin{proof}
By exchanging $f$ with $g$ and $D$ with $D^+$ if necessary, we may suppose that $p < \infty$.

Take a bump function $\eta$ equal to $1$ on $\supp (f \otimes g)$.
By Corollary~\ref{cor:cssmoothapprox}, there exists a sequence of $C^\infty_\comp(\bundle{E})$-sections $f_m$ such that $(f_m,Df_m) \to (f,Df)$ in $L^p_\loc(\bundle{E}\oplus \bundle{F})$, and then $(\eta f_m,D(\eta f_m)) \to (\eta f,D(\eta f))$ in $L^p_\comp(\bundle{E}\oplus \bundle{F})$.
Now
\[
D(\eta f) = (\Symbol{D} \eta)f + \eta Df
\]
by Leibniz' rule \eqref{eq:dist_leibniz} for a smooth function $\eta$ and a distribution $f$, and moreover $\Symbol{D} \eta$ vanishes on $\supp (f \otimes g)$, so $\llangle (\Symbol{D} \eta) f, g\rrangle = 0$.
Hence \eqref{eqn:defn-distn} implies that
\[
\begin{aligned}
\llangle Df, g \rrangle
&= \llangle D(\eta f), g \rrangle = \lim_m \llangle D(\eta f_m), g \rrangle \\
&= \lim_m \llangle \eta f_m, D^+ g \rrangle = \llangle \eta f, D^+ g \rrangle \\
&= \llangle f, D^+ g \rrangle,
\end{aligned}
\]
as required.
\end{proof}

\begin{proposition}[Leibniz' rule]\label{prp:leibniz}
Suppose that $D \in \Diff_1(\bundle{E},\bundle{F})$ and $1 \leq p \leq \infty$.
Suppose also that $h \in W^p_{\Symbol{D},\loc}(\bundle{T})$ and $f \in W^{p'}_{D, \loc}(\bundle{E})$.
Then $hf \in W^1_{D, \loc}(\bundle{E})$ and
\[
D(h f) = (\Symbol{D} h) f + h D f.
\]
\end{proposition}
\begin{proof}
By H\"older's inequality, $hf \in L^1_\loc(\bundle{E})$ and $(\Symbol{D} h) f + h D f \in L^1_\loc(\bundle{F})$.
We must show that the $\bundle{F}$-valued distributions $D(hf)$ and $(\Symbol{D} h) f + h D f$ coincide.
In fact, for all $\phi \in C^\infty_\comp(\bundle{F})$,
\[
\llangle D(h f), \phi \rrangle = \llangle h f, D^+ \phi \rrangle = \llangle f, \overline{h} D^+ \phi \rrangle.
\]
Leibniz' rule \eqref{eq:dist_leibniz} for a smooth section $\phi$ and a distribution $\overline{h}$, together with \eqref{eq:dist_tauadj}, leads to the distributional equality
\[
D^+(\overline{h} \phi) =  -(\Symbol{D} h)^* \phi + \overline{h} D^+ \phi.
\]
By the hypotheses, each summand in the right-hand side lies in $L^p_\comp(\bundle{E})$, therefore $D^+(\overline{h} \phi) \in L^p_\comp(\bundle{E})$ too, and
\[
\llangle D(h f), \phi \rrangle = \llangle f, D^+(\overline{h} \phi) + (\Symbol{D} h)^* \phi \rrangle = \llangle h Df + (\Symbol{D} h) f, \phi \rrangle
\]
by Proposition~\ref{prp:parts}, since $\overline{h} \phi$ is compactly-supported and in $L^p_\loc(\bundle{F})$.
\end{proof}

\begin{proposition}[Chain rule]\label{prp:chain}
Suppose that $D \in \Diff_1(\bundle{T}, \bundle{E})$ is homogeneous.
If $1 \leq p < \infty$ and $f \in W^p_{D, \loc}(\bundle{T})$ is real-valued, and the function $g\colon \R \to \C$ is continuously differentiable and $g' $ is bounded,
then $g \circ f \in W^p_{D, \loc}(\bundle{T})$ and
\[
D(g \circ f) = (g' \circ f) \, D f.
\]
\end{proposition}
\begin{proof}
By Corollary~\ref{cor:cssmoothapprox}, there exists a sequence of $C^\infty_\comp(\bundle{T}_\R)$-sections $f_m$ such that $(f_m,Df_m) \to (f,Df)$ in $L^p_\loc(\bundle{T} \oplus \bundle{E})$; by extracting a subsequence, we may suppose that the convergence is also pointwise almost everywhere.
Now $g$ is Lipschitz and $f_m \to f$ in $L^p_\loc(\bundle{T})$, and so $g \circ f_m \to g \circ f$ in $L^p_\loc(\bundle{T})$.
Moreover, by the chain rule for $C^1$-functions,
\[
D(g \circ f_m) = (g' \circ f_m) \, D f_m = (g' \circ f_m) \, (Df_m - Df) + (g' \circ f_m) \, Df.
\]
Since $\|g' \circ f_m\|_\infty \leq \|g'\|_\infty < \infty$, the first summand converges to $0$ in $L^p_\loc(\bundle{E})$; moreover, since $g'$ is continuous, $g' \circ f_m \to g' \circ f$ pointwise almost everywhere, and therefore the second summand converges to $(g' \circ f) \, Df$ in $L^p_\loc(\bundle{E})$ by the dominated convergence theorem.
Thus $g \circ f_m\to g \circ f$ in $W^p_{D, \loc}(\bundle{T})$, and the conclusion follows.
\end{proof}

\section{Reversible sub-Finsler geometry}\label{section:subfinsler}

Suppose that $P$ is a continuous fibre seminorm on $T^* M$, and $P^*$ is the dual extended fibre norm on $TM$.
Thus
\[
P_x(\xi) = \sup_{\substack{v \in T_x M \\ P^*(v) \leq 1}} |\xi(v)|
\quad\text{and}\quad
P^*_x(v) = \sup_{\substack{\omega \in T^*_x M \\ P(\omega) \leq 1}} |\xi(v)|
\]
by the finite-dimensional Hahn--Banach theorem.

As a function on the tangent bundle, $P^*$ need not be continuous.
However, it may be approximated by continuous fibre norms on $TM$, as we are about to show.

\begin{lemma}\label{lem:riemannianapproximants}
There exists a countable family $\mathfrak{G}_P$ of riemannian metrics on $M$ such that
\begin{align}
P^*(v)
&= \sup_{g \in \mathfrak{G}_P} |v|_g
\qquant v \in TM  \label{eqn:p-star-g} \\
\noalign{\noindent{or, equivalently,}}
P(\xi)
&= \inf_{g \in \mathfrak{G}_P} |\xi|_g
\qquant \xi \in T^*M.   \label{eqn:p-g}
\end{align}
\end{lemma}

\begin{proof}
Recall that a riemannian metric on $M$ is given by a smooth fibre inner product on $TM$, or, by duality, by a smooth fibre inner product on $T^* M$.
From a geometric point of view, proving \eqref{eqn:p-star-g} amounts to realising the closed unit ball of $P^*$ at a point $x \in M$ (which is convex but may have no interior) as the intersection of the closed unit balls of the metrics $g$ in $\mathfrak{G}_P$, which are ellipsoids, and proving \eqref{eqn:p-g} amounts to realising the open unit ball of $P$ at a point $x \in M$ (which is convex, but may be unbounded) as the union of the open unit balls of the metrics $g$, which are also ellipsoids.
In general, this may require an infinite number of ellipsoids, as we may see by considering the problem of realising a square as an intersection or union of $g$ balls. We consider the cotangent space problem only.

It is easy to show that a riemannian metric that satisfies
\begin{equation}\label{eqn:good-riemannian}
 P(\xi) \leq |\xi|_g
\qquant \xi \in T^* M
\end{equation}
exists.
Indeed, if $g$ is a riemannian metric on $M$, then the function
\[
x \mapsto \sup_{\substack{\xi \in T^*_x M \\ |\xi|_g \leq 1}} P(\xi)
\]
is locally finite, therefore it is majorised by a strictly positive function $\psi \in C^\infty(\bundle{T})$, and one simply needs to rescale $g$ by $\psi^2$.

Take a riemannian metric $g$ on $M$ satisfying \eqref{eqn:good-riemannian} and the countable atlas $(\phi_\alpha)_{\alpha \in \Alpha}$.
Recall that each $\phi_\alpha$ maps $U_\alpha$ to $\R^n$, that $V_\alpha = \phi_\alpha^{-1}(B_{\R^n}(0,1))$, and that $M = \bigcup_{\alpha \in \Alpha} V_\alpha$.
Each subbundle $T^*U_\alpha$ of $T^* M$ is trivialisable.

Fix $\alpha \in \Alpha $, and choose a bump function $\zeta_\alpha$ with compact support in $U_\alpha$ that is equal to $1$ on $V_\alpha$ and a countable set $\mathcal{Y}_\alpha$ of smooth sections of $T^*U_\alpha$ such that
\begin{equation}\label{eq:omegadensity}
\afterline{ \{ \omega(x) \tc \omega \in \mathcal{Y}_\alpha \} } =   \{ \xi \in T^*_x M \tc |\xi|_g = 1 \}
\end{equation}
for all $x \in U_\alpha$.
To do this, it is sufficient to consider constant sections taking values in a countable dense subset of the unit sphere with respect to a trivialisation of $T^*U_\alpha$ given by a $g$-orthonormal frame.

Next, fix $\omega \in \mathcal{Y}_\alpha$.
Since $P(\omega)$ is a continuous nonnegative function on $U_\alpha$, there is a sequence of smooth functions $\psi_{\omega,k}: U_\alpha \to \R$ such that
\[
P(\omega) + 2^{-k} \leq \psi_{\omega,k} \leq P(\omega) + 2^{1-k}.
\]
We now define, for all $k \in \N$, a smooth inner product $(\cdot,\cdot)_{\alpha,\omega,k}$ and associated norm $|\cdot|_{\alpha,\omega,k}$ along the fibres of $T^* U_\alpha$ by
\begin{align*}
&\shiftleft (\xi_1,\xi_2)_{\alpha,\omega,k} \\
&= \psi_{\omega,k}(x)^2 \bigl( \langle \pi(\xi_1), \pi(\xi_2) \rangle_g + 2^{2k}  \langle \xi_1 - \pi(\xi_1), \xi_2 - \pi(\xi_2) \rangle_g  \bigr) \\
&= \psi_{\omega,k}(x)^2 \bigl( \langle \xi_1, \omega(x) \rangle_g\, \langle \xi_2, \omega(x) \rangle_g \\
&\qquad+ 2^{2k} (\langle \xi_1,\xi_2 \rangle_g - \langle \xi_1, \omega(x) \rangle_g\, \langle \xi_2, \omega(x) \rangle_g )\bigr)
\end{align*}
for all $\xi_1, \xi_2 \in T^*_x M$ and $x \in U_\alpha$, where $\pi(\xi)$ is the projection of $\xi$ in $T_xM$ onto $\R\omega(x)$, that is, $\pi(\xi) = \langle \xi,\omega(x) \rangle_g\,  \omega(x)$.
Now
\[
|\pi( \xi)|_g \leq  \frac{ |\xi |_{\alpha,\omega,k} }{\psi_{\omega,k}(x)}
\quad\text{and}\quad
|\xi - \pi(\xi)|_g  \leq  \frac{ |\xi |_{\alpha,\omega,k} }{2^k \psi_{\omega,k}(x)}
\]
for all $\xi \in T^*_x M$ and $x \in U_\alpha$,
and so, from \eqref{eqn:good-riemannian},
\[
\begin{aligned}
P(\xi)
&\leq P( \pi( \xi) ) + P(\xi - \pi(\xi)) \\
&\leq |\langle \xi,\omega(x) \rangle_g| \, P(\omega(x)) + |\xi - \pi(\xi)|_g \\
&\leq \frac{ |\xi |_{\alpha,\omega,k}} {\psi_{\omega,k}(x)} \bigl( P(\omega(x)) + 2^{-k} \bigr) \\
&\leq |\xi |_{\alpha,\omega,k}.
\end{aligned}
\]
Hence
\begin{equation}\label{eqn:P-estimate-1}
P(\xi) \leq \inf_{\substack{k\in \N \\ \omega \in \mathcal{Y}_\alpha}}  |\xi|_{\alpha,\omega,k} .
\end{equation}
Moreover, for all $x \in U_\alpha$, from the definitions of $\psi_{\omega,k}$ and $(\cdot,\cdot)_{\alpha,\omega,k}$,
\[
 P(\omega(x)) + 2^{1-k} \geq \psi_{\omega,k}(x) = |\omega(x)|_{\alpha,\omega,k}.
\]
More generally, for all $\xi \in T_x^* U_\alpha$ such that $|\xi|_g = 1$ and all $k \in \N$, we may choose $\omega \in \mathcal{Y}_\alpha$ such that $|\xi - \omega(x)|_g \leq 2^{-2k}$, by \eqref{eq:omegadensity}. Then by construction, $|\xi-\omega(x)|_{\alpha,\omega,k} \leq 3 \times 2^{-2k}$ and $P(\xi - \omega(x)) \leq 2^{-2k}$, hence
\[
P(\xi) \geq |\xi|_{\alpha,\omega,k} - 2^{1-k} - 3\times 2^{-k} - 2^{-2k}
\]
and the reverse of inequality \eqref{eqn:P-estimate-1} follows. Putting everything together, we deduce that
\[
P(\xi ) = \inf_{\substack{k\in \N \\ \omega \in \mathcal{Y}_\alpha}}  |\xi|_{\alpha,\omega,k}
\]
for all $x \in U_\alpha$ and $\xi \in T^*_xM$.

For each $\alpha \in \Alpha$, choose a bump function $\zeta_\alpha$ with compact support in $U_\alpha$ that is equal to $1$ on $V_\alpha$, and define a riemannian metric $g_{\alpha,\omega,k}$ on $M$ by setting
\[
\langle \xi, \xi \rangle_{g_{\alpha,\omega,k}}
= \zeta_\alpha(x) \, (\xi,\xi)_{\alpha,\omega,k} + (1-\zeta_\alpha(x)) \, \langle \xi, \xi \rangle_g
\]
for all $x \in M$ and $\xi \in T^*_x M$; the first summand is defined to vanish whenever $x \notin U_\alpha$.
Then clearly $P(\xi) \leq |\xi|_{g_{\alpha,\omega,k}}$ for all $\xi \in T^* M$.
Moreover, if $x \in M$, then $x \in V_\alpha$ for some $\alpha$, therefore $|\xi|_{g_{\alpha,\omega,k}} = |\xi|_{\alpha,\omega,k}$ for all $\xi \in T^*_x M$.
We now set
\[
\mathfrak{G}_P = \{ g_{\alpha,\omega,k} \tc \omega \in \mathcal{Y}_\alpha, k \in \N,  \alpha \in \Alpha \},
\]
and the desired conclusion follows.
\end{proof}

Define the finite subspace of $P^*$ in $TM$ and the zero subspace of $P$ in $T^*M$ by
\[
F(P^*_x) = \{ v \in T_x M \tc P^*(v) < \infty\}
\quad\text{and}\quad
Z(P_x) = \{ \xi \in T_x^* P \tc P(\xi) = 0 \}.
\]
Then $F(P^*_x)$ is the annihilator of $Z(P_x)$, so $\dim F(P^*_x) = \codim Z(P_x)$, and the function $x \mapsto \dim F(P^*_x)$ is lower-semicontinuous.
When this function is continuous, that is, when it is locally constant, $P^*$ has additional continuity properties.
We define $F(P^*) = \bigcup_{x \in M} F(P^*_x)$.

\begin{proposition}\label{prp:constantdimension}
Suppose that $x \mapsto \dim F(P^*_x)$ is continuous.
Then $F(P^*)$ is closed in $TM$, and $P^*$ restricted to $F(P^*)$ is continuous.
\end{proposition}
\begin{proof}
Without loss of generality, we suppose that $M$ is connected, so the function $x \mapsto \dim F(P^*_x)$ is constant, that is, $\codim F(P^*_x) = k$ for some $k$ and all $x \in M$.

Write $G$ for the $k$th grassmannian bundle over $T^* M$, so $G_x$ is the set of $k$-dimensional subspaces of $T^*_x M$, and define $X = \{S \in G \tc P|_S = 0\}$.
Then $X$ is closed in $G$, because $P$ is continuous, and $X \cap G_x = \{Z(P_x)\}$ for all $x \in M$.
Since $G$ has compact fibres, $X$ is the image of a continuous section of $G$, and this section may be lifted locally to a continuous section of the frame bundle of $T^* M$.
Thus there is a continuous local frame $\{\omega_1,\dots,\omega_n\}$ for $T^* M$ in the neighbourhood of each point of $M$ such that $Z(P_x) = \Span \{ \omega_1|_x, \dots, \omega_k|_x\}$, whence
\[
F(P^*_x) = \ker \omega_1|_x \cap \dots \cap \ker \omega_k|_x = \Span \{ \omega_{k+1}^*|_x, \dots, \omega_n^*|_x \},
 \]
where $\{\omega_1^*,\dots,\omega_n^*\}$ is the dual local frame for $TM$.
This proves that $F(P^*)$ is closed in $TM$, and determines a continuous subbundle $E$ of $TM$.

Denote by $\iota^*\colon T^* M \to E^*$ the pointwise transpose of the inclusion map $\iota\colon E \to TM$.
Then $Z(P_x) = \ker \iota^*|_{T^*_x M}$, hence $P$ induces a continuous fibre norm $Q$ on $E^*$ such that $Q \circ \iota^* = P$.
It is then easily checked that the restriction of $P^*$ to $E$ is the dual norm of $Q$ pointwise.

By the use of local trivialisations of $E^*$, we may find, for all $x \in M$, a neighbourhood $U$ of $x$ and linear isomorphisms $t_y\colon E^*_x \to E^*_y$ for all $y \in U$ such that the mapping $(y,\xi) \mapsto t_y(\xi)$ is continuous from $U \times E^*_x$ to $E^*$.
The continuity of $Q$ and the compactness of the unit sphere of $Q_x$ in $E^*_x$ then imply that, for all positive $\epsilon$, there is a neighbourhood $V$ of $x$ in $M$ such that, for all $y \in V$,
\[
(1+\epsilon)^{-1} Q_y \circ t_y \leq Q_x \leq (1+\epsilon) Q_y \circ t_y
\]
and correspondingly
\[
(1+\epsilon)^{-1} P^*_y|_{F(P^*_y)} \leq P^*_x|_{F(P^*_x)} \circ t^*_y \leq (1+\epsilon) P^*_y|_{F(P^*_y)}.
\]
This proves the continuity of $P^*|_{F(P^*)}$.
\end{proof}

\begin{definition}
A tangent vector $v \in T_x M$ is said to be $P$-subunit if $P^*(v) \leq 1$.
\end{definition}

\begin{definition}\label{def:controldistance}
We write $\Gamma^k([a,b])$ for the set of all curves $\gamma\colon [a,b] \to M$ of class $C^k$; here $k$ may be $\infty$.
A curve $\gamma\colon [a,b] \to M$ is said to be $P$-subunit if it is absolutely continuous and $\gamma'(t)$ is $P$-subunit for almost all $t \in [a,b]$.
We write $\Gamma_P^\subunit(I)$ for the set of all $P$-subunit curves defined on the interval $I$, and $\Gamma_P^\subunit$ for the set of all $P$-subunit curves when the interval of definition may vary.
We write $\dist_P$ for the distance function induced by $P$, that is,
$\dist_P(x,y)$ is the infimum of the set of all $T \in \R^+$ for which there exists $\gamma \in \Gamma_P^\subunit( [0,T] )$ such that $\gamma(0) = x$ and $\gamma(T) = y$.
\end{definition}

The infimum need not be attained: for instance, in $\R^2\setminus \{(0,0)\}$ with the euclidean metric, there is no minimising curve joining $(-1,0)$ and $(1,0)$.

Absolute continuity may be defined in various equivalent ways: here is one.

\begin{definition}\label{def:absolutecontinuous}
The curve $\gamma\colon [a,b] \to M$ is said to be \emph{absolutely continuous} if $\phi_\alpha \circ \gamma$ is locally absolutely continuous for each $\phi_\alpha$ in the atlas $\Alpha$. 
We write $AC([a,b])$ for the set of all absolutely curves on the interval $[a,b]$.
\end{definition}

Suppose that $P$ is a norm induced by a riemannian metric $g$ on $M$ and $\dist_g$ is the distance function induced by $g$.
If $\gamma \in\Gamma_P^\subunit( [a,b] )$, then
\begin{equation}\label{eq:riemannianlipschitzcurve}
\dist_g(\gamma(s),\gamma(t)) \leq |s-t|
\qquant s,t \in [a,b].
\end{equation}
Conversely, a curve $\gamma\colon [a,b] \to M$ that satisfies \eqref{eq:riemannianlipschitzcurve} is $P$-subunit: the derivative $\gamma'(t)$ may be computed in exponential coordinates centred at $\gamma(t)$, and the difference quotient is controlled by the Lipschitz constant.

Thanks to Lemma~\ref{lem:riemannianapproximants}, a similar result may be proved for an arbitrary fibre seminorm $P$.
We use the family $\mathfrak{G}_P$ of riemannian metrics defined in Lemma~\ref{lem:riemannianapproximants}.

\begin{proposition}\label{prp:subfinslermetric}
The function $\dist_P$ is an extended distance function on $M$,
\begin{equation}\label{eq:riemanniandistancecomparison}
\dist_g(x,y) \leq \dist_P(x,y)
\end{equation}
for all $x,y \in M$ and $g \in \mathfrak{G}_P$,
and the topology induced by $\dist_P$ is at least as fine as the manifold topology.
Further, for a function $\gamma\colon [a,b] \to M$, the following conditions are equivalent:
\begin{enumerate}[(i)]
\item $\gamma$ is a $P$-subunit curve;
\item $\dist_P(\gamma(s),\gamma(t)) \leq |s-t|$ for all $s,t \in [a,b]$;
\item $\dist_g(\gamma(s),\gamma(t)) \leq |s-t|$ for all $s,t \in [a,b]$ and all $g \in \mathfrak{G}_P$.
\end{enumerate}
\end{proposition}
\begin{proof}
For a riemannian metric $g \in \mathfrak{G}_P$, inequality \eqref{eq:riemanniandistancecomparison} follows easily from the fact that, for each $\gamma \in \Gamma_P^\subunit( [0,T] )$ joining $x$ to $y$, the $g$-norm of $\gamma'(t)$ is at most $1$ for almost all $t$, so
\[
\dist_g(x,y) \leq \int_0^T |\gamma'(t)|_g \,{\D}t \leq T.
\]

From \eqref{eq:riemanniandistancecomparison}, if $\dist_P(x,y) = 0$, then $\dist_g(x,y) = 0$ and hence $x = y$; it follows immediately that $\dist_P$ satisfies the other axioms for an (extended) distance function.
Moreover, again by \eqref{eq:riemanniandistancecomparison}, the topology induced by $\dist_P$ is no coarser than the topology induced by $\dist_g$, that is, the original topology of $M$.

Further, for a function $\gamma\colon [a,b] \to M$, condition (i) implies condition (ii) by the definition of $\dist_P$, while condition (ii) implies condition (iii) by  \eqref{eq:riemanniandistancecomparison}.
Finally, if condition (iii) holds, then $\gamma$ is absolutely continuous and, for all $g \in \mathfrak{G}_P$, $|\gamma'(t)|_g \leq 1$ for almost all $t \in [a,b]$.
As $\mathfrak{G}_P$ is countable, we may reverse the order of the quantifiers on $g$ and $t$, and deduce from Lemma~\ref{lem:riemannianapproximants} that $P^*(\gamma'(t)) \leq 1$ for almost all $t \in [a,b]$, which is condition (i).
\end{proof}

\subsection{Topologies on $M$}

By Proposition~\ref{prp:subfinslermetric}, the topology induced by $\dist_P$ is no coarser than the original manifold topology of $M$; recall (from Definition \ref{def:varietal}) that $\dist_P$ is varietal if the two topologies are equivalent.
In general, the topology induced by $\dist_P$ may be finer than the original manifold topology of $M$.
Unless otherwise specified, we do not assume that $\dist_P$ is varietal, and topological concepts such as compactness and convergence refer to the original topology of $M$.

\begin{lemma}\label{lem:subunitlimit}
Suppose that a sequence of curves $\gamma_m \in \Gamma_P^\subunit([a,b])$ converges pointwise to a curve $\gamma\colon [a,b] \to M$.
Then $\gamma \in \Gamma_P^\subunit([a,b])$.
\end{lemma}
\begin{proof}
The characterisation of $P$-subunit curves in Proposition~\ref{prp:subfinslermetric} (iii) is preserved by pointwise convergence.
\end{proof}

Recall that $\bar{B}_P(K,R)$ denotes  $\{ x \in M \tc \dist_P(K,x) \leq R\}$.

\begin{proposition}\label{prp:minimizingsubunit}
Suppose that $K$ is a compact subset of $M$ and $R \in \R^+$.
If ${\bar{B}_P(K,R)}$ has compact closure, then it is compact and coincides with the set of all $x \in M$ for which there exists $\gamma \in\Gamma_P^\subunit( [0,R])$ such that $\gamma(0) \in K$ and $\gamma(R) = x$.
Moreover $\bar{B}_P(K,R')$ is also compact for some $R'$ greater than $R$.
\end{proposition}
\begin{proof}
Take a riemannian metric $g \in \mathfrak{G}_P$.
If $\bar{B}_P(K,R)$ is relatively compact, then $\bar{B}_{g} (\bar{B}_P(K, R), R\epsilon)$ is compact for sufficiently small positive $\epsilon$.

Take $x \in \bar{B}_P(K,R)$.
We may then find $\gamma_m \in \Gamma_P^\subunit( [0,R(1+\epsilon_m)] )$ such that $\gamma_m(0) \in K$ and $\gamma_m((1+\epsilon_m)R) = x$, where $0 \leq \epsilon_m \leq \epsilon$ and $\epsilon_m \downarrow 0$.
Hence the images of the $\gamma_m$ are all contained in $\bar{B}_{g} (\bar{B}_P(K, R), R\epsilon)$.
If we rescale these curves so that they are all defined on $[0,R]$, we obtain a sequence of curves $\tilde\gamma_m\colon [0,R] \to M$ which are $(1+\epsilon)$-Lipschitz with respect to $\dist_g$, and whose images are all contained in $\bar{B}_{g} (\bar{B}_P(K, R), R\epsilon)$.
By the Arzel\`a--Ascoli theorem, after taking a subsequence, the $\tilde\gamma_m$ converge uniformly to a curve $\gamma\colon [0,R] \to M$ for which $\gamma(0) \in K$ and $\gamma(R) = x$.

Now, for all positive $\delta$, the rescalings of the curves $\tilde\gamma_m$ on $[0,(1+\delta) R]$ are eventually $P$-subunit (since $\epsilon_m \to 0$), therefore their limit, that is, the rescaling of $\gamma$ on $[0,(1+\delta) R]$, is also $P$-subunit, by Lemma~\ref{lem:subunitlimit}.
In other words, $P^*(\gamma'(t)) \leq 1+\delta$ for all positive $\delta$ and almost all $t \in [0,R]$.
It follows by exchanging quantifiers that $\gamma$ is $P$-subunit.

Finally, take $x$ in the closure of $\bar{B}_P(K,R)$.
Then there is a sequence of $P$-subunit curves $\gamma_m\colon [0,R] \to M$ such that $\gamma_m(0) \in K$ and $\gamma_m(R) \to x$.
As before, we may extract a subsequence that converges uniformly to a $P$-subunit curve $\gamma\colon [0,R] \to M$ such that $\gamma(0) \in K$ and $\gamma(R) = x$, and therefore $x \in \bar{B}_P(K,R)$.
This shows that $\bar{B}_P(K,R)$ is closed, hence compact.
The same argument proves that $\bar{B}_P(K,R(1+\epsilon))$ is compact too, since it is contained in the compact set $\bar{B}_{g} (\bar{B}_P(K, R), R\epsilon)$.
\end{proof}

If $\dist_P$ is varietal, then the proof of Proposition~\ref{prp:minimizingsubunit} may be simplified.

\begin{definition}\label{eqn:def-R_P(K)}
For a compact subset $K$ of $M$, we define
\begin{equation}
R_P(K) = \sup \{ R \in \R^+ \tc \bar{B}_P(K,R) \in \mathfrak{K}(M) \}.
\end{equation}
For a point $x$ in $M$, we write $R_P(x)$ instead of  $R_P(\{x\})$.
\end{definition}
By Proposition~\ref{prp:minimizingsubunit}, the supremum is never a maximum and is always strictly positive.

\subsection{Distance, rectifiability and length}

The previous characterisation of $P$-subunit curves shows that $(M,\dist_P)$ is an (extended) length space, in the sense of Gromov (see, for instance, \cite{roe_lectures_2003}).

\begin{definition}\label{def:length}
Suppose that $\gamma\colon \leftclosedint a,b\rightclosedint \to M$ is a continuous curve. 
The $P$-length of $\gamma$, written $\ell_P(\gamma)$, is defined to be
\begin{equation}
 \sup \left\{ \sum_{j=1}^m \dist_P(\gamma(t_{j-1}),\gamma(t_j)) \tc m \in \N, \, a = t_0 \leq \dots \leq t_m = b \right\}.
\end{equation}
\end{definition}

To help state the next results, we define $\Gamma_P([a,b])$ and $\Gamma([a,b])$ to be the sets of all $\dist_P$-continuous and all continuous curves $\gamma\colon [a,b] \to M$.

\begin{proposition}\label{prp:lengthspace}
For all $x,y \in M$, the distance $\dist_P(x,y)$ is equal to 
\begin{equation}\label{eq:lengthspace}
 \inf \left\{ \ell_P(\gamma) \tc  \gamma \in \Gamma_P( [a,b] ), \ \gamma(a) = x, \ \gamma(b) = y \right\}.
\end{equation}
\end{proposition}
\begin{proof}
Write $\tilde \dist_P(x,y)$ for the expression \eqref{eq:lengthspace}.
On the one hand, $\tilde \dist_P$ is an extended distance function and $\tilde \dist_P \geq \dist_P$ since $\ell_P(\gamma) \geq \dist_P(\gamma(a),\gamma(b))$ for all $\gamma \in \Gamma_P( [a,b] )$.
On the other hand, if $\gamma \in \Gamma_P^\subunit( [0,T] )$, then $\gamma \in \Gamma_P( [a,b] )$ and $\ell_P(\gamma) \leq T$, by Proposition~\ref{prp:subfinslermetric}, and the reverse inequality $\tilde \dist_P \leq \dist_P$ follows.
\end{proof}

Next we show that the expression \eqref{eq:lengthspace} does not change if we require only that the curves $\gamma$ are continuous with respect to the manifold topology.

\begin{proposition}\label{prp:finitelengthimpliesPcts}
If $\gamma \in \Gamma( \leftclosedint a,b\rightclosedint )$ and $\ell_P(\gamma) < \infty$, then $\gamma \in \Gamma_P( \leftclosedint a,b\rightclosedint )$, and the topology on $\gamma(\leftclosedint a,b\rightclosedint)$ induced by $\dist_P$ and the relative topology coincide.
\end{proposition}
\begin{proof}
Since $\gamma$ is continuous, $\gamma(\leftclosedint a,b \rightclosedint)$ is compact, hence $R_P(\gamma(\leftclosedint 0,L \rightclosedint)) > 0$ by Proposition~\ref{prp:minimizingsubunit}.

Fix now $\bar{t} \in \leftclosedint a,b\rightopenint$.
First, since $\ell_P(\gamma) < \infty$,
\[
\inf_{\epsilon > 0} \sup_{t,t' \in \leftopenint \bar{t}, \bar{t}+\epsilon \rightopenint} \dist_P(\gamma(t),\gamma(t')) = 0;
\]
in fact, if the infimum $\eta$ were positive, then we could find a decreasing sequence $(t_m)_{m\in\N}$ tending to $\bar{t}$ such that $\dist_P(\gamma(t_{2k+1}),\gamma(t_{2k})) \geq \eta/2$, and deduce that
\[
\ell_P(\gamma) \geq \sum_{k=0}^{j-1} \dist_P(\gamma(t_{2k+1}),\gamma(t_{2k})) \geq j \frac{\eta}{2}
\]
for all $j \in \N$, which is a contradiction.

Therefore, for every positive $\delta$, there is a positive $\epsilon$ such that
\[
\gamma(\leftopenint \bar{t},\bar{t}+\epsilon\rightopenint) \subseteq \bar{B}_P(\gamma(t),\delta)
\]
for all $t \in \leftopenint \bar{t},\bar{t}+\epsilon\rightopenint$.
If $\delta < R_P(\gamma(\leftclosedint 0,L \rightclosedint))$, then $\bar{B}_P(\gamma(t),\delta)$ is closed, hence
\[
\gamma(\bar{t}) \in \afterline{\gamma(\leftopenint \bar{t},\bar{t}+\epsilon\rightopenint)} \subseteq \bar{B}_P(\gamma(t),\delta)
\]
by the continuity of $\gamma$, which means that $\dist_P(\gamma(\bar{t}),\gamma(t)) \leq \delta$.

This proves that $\lim_{t \to \bar{t}+} \dist_P(\gamma(t),\gamma(\bar{t})) = 0$.
The proof when $\bar{t} \in \leftopenint a,b\rightclosedint$ and $t \to \bar{t}-$ is similar.
To conclude, recall that a continuous map from a compact space to a Hausdorff space is closed; hence every topology on $M$ that makes $\gamma$ continuous induces the quotient topology induced by $\gamma$ on $\gamma(\leftclosedint a,b \rightclosedint)$.
\end{proof}

\begin{corollary}\label{corol:lengthspacesmooth}
For all $x,y \in M$, the distance $\dist_P(x,y)$ is equal to 
\begin{equation*}
 \inf \left\{ \ell_P(\gamma) \tc  \gamma \in \Gamma( [a,b] ), \ \gamma(a) = x, \ \gamma(b) = y \right\}.
\end{equation*}
\end{corollary}
\begin{proof}
This follows immediately from Propositions \ref{prp:lengthspace} and \ref{prp:finitelengthimpliesPcts}.
\end{proof}

We may express the length of an absolutely continuous curve as an integral.

\begin{proposition}\label{prp:lengthintegral}
Suppose that $\gamma \in AC([a,b])$.
Then
\begin{equation}\label{eq:lengthintegral}
\ell_P(\gamma) = \int_a^b P^*(\gamma'(t)) \, {\D}t.
\end{equation}
If $\ell_P(\gamma) < \infty$, then $\gamma$ is also $\dist_P$-absolutely continuous and
\begin{equation}\label{eq:metricderivative}
P^*(\gamma'(t)) = \lim_{s \to t} \frac{\dist_P(\gamma(s),\gamma(t))}{|s-t|}
\end{equation}
for almost all $t \in [a,b]$.
\end{proposition}

We remark that, in the general theory of absolutely continuous curves in metric spaces (see, for example, \cite[Section 4.1]{ambrosio_topics_2004} or \cite[Section 1.1]{ambrosio_gradient_2008}), the right-hand side of \eqref{eq:metricderivative} is known as the \emph{metric derivative} of $\gamma$.

\begin{proof}
Note first that the corresponding statement for a riemannian metric $g$ on $M$ is easily proved.
To do so, define $\ell_g$ like $\ell_P$ in Definition \ref{def:length}, but with $\dist_P$ replaced by $\dist_g$.
By using exponential coordinates centered at $\gamma(t)$, one sees that
\begin{equation}\label{eq:riemannianmetricderivative}
|\gamma'(t)|_g = \lim_{s \to t} \frac{\dist_g(\gamma(s),\gamma(t))}{|s-t|}
\end{equation}
for all points $t$ in $[a,b]$ at which $\gamma$ is differentiable, and it follows from the theory of absolutely continuous curves in metric spaces that
\begin{equation}\label{eq:riemannianlengthintegral}
\ell_{g}(\gamma|_{\leftclosedint t_1,t_2\rightclosedint}) = \int_{t_1}^{t_2} |\gamma'(\tau)|_g \,{\D}\tau
\end{equation}
whenever $a \leq t_1 \leq t_2 \leq b$.
From \eqref{eqn:p-star-g}, \eqref{eq:riemanniandistancecomparison} and \eqref{eq:riemannianmetricderivative}, we deduce that
\begin{equation}\label{eq:metricderivativeinequality}
P^*(\gamma'(t)) \leq \liminf_{s \to t} \frac{\dist_P(\gamma(s),\gamma(t))}{|s-t|}
\end{equation}
for all $t \in [a,b]$ where $\gamma$ is differentiable.

Suppose now that $\ell_P(\gamma) < \infty$.
Then the function $r\colon \leftclosedint a,b\rightclosedint \to \R$, defined by $r(t) = \ell_P(\gamma|_{\leftclosedint a,t\rightclosedint})$, is nondecreasing, so differentiable almost everywhere, and
\[
\int_{a}^{b} r'(\tau) \,{\D}\tau \leq r(b) - r(a) = \ell_P(\gamma).
\]
Now
\[
\dist_P(\gamma(t_1),\gamma(t_2)) \leq \ell_P(\gamma|_{\leftclosedint t_1,t_2\rightclosedint}) = r(t_2) - r(t_1)
\]
whenever $a \leq t_1 \leq t_2 \leq b$, so $P^*(\gamma'(t)) \leq r'(t)$ for almost all $t \in [a,b]$ from \eqref{eq:metricderivativeinequality}, and \emph{a fortiori}
\[
\int_a^b P^*(\gamma'(\tau)) \,{\D}\tau \leq \ell_P(\gamma).
\]
The same inequality holds trivially when the right-hand side is infinite.

Conversely, if $T = \int_a^b P^*(\gamma'(\tau)) \,{\D}\tau < \infty$, then define $\tilde r\colon \leftclosedint a,b\rightclosedint \to [0,T]$ by
\[
\tilde r(t) = \int_a^t P^*(\gamma'(\tau)) \,{\D}\tau .
\]
The function $\tilde r$ is nondecreasing and surjective.
Further, if $a \leq t_1 \leq t_2 \leq b$ and $g \in \mathfrak{G}_P$, then
\[
\dist_g(\gamma(t_1),\gamma(t_2)) 
\leq \int_{t_1}^{t_2} |\gamma'(\tau)|_g \,{\D}\tau 
\leq \int_{t_1}^{t_2} P^*(\gamma'(\tau)) \,{\D}\tau 
= \tilde r(t_2) - \tilde r(t_1),
\]
by \eqref{eqn:p-star-g} and \eqref{eq:riemannianlengthintegral}.
In particular, if $\tilde r(t_1) = \tilde r(t_2)$ then $\gamma(t_1) = \gamma(t_2)$, hence we may define a function $\tilde\gamma\colon \leftclosedint 0,T\rightclosedint \to M$ by $\gamma = \tilde\gamma \circ \tilde r$, and $\tilde\gamma$ is $1$-Lipschitz with respect to $\dist_g$ for every $g \in \mathfrak{G}_P$.
By Proposition~\ref{prp:subfinslermetric}, this implies that $\tilde\gamma\colon [0,T] \to M$ is $1$-Lipschitz with respect to $\dist_P$, hence $\tilde\gamma \in \Gamma_P^\subunit( [0,T] )$ and
\[
\ell_P(\gamma) = \ell_P(\tilde\gamma) \leq T = \int_a^b P^*(\gamma'(\tau)) \,{\D}\tau.
\]
Again, this inequality holds trivially when the right-hand side is infinite, and we have proved \eqref{eq:lengthintegral}.

If $\ell_P(\gamma) < \infty$, then $P^*(\gamma'(\cdot))$ is integrable on $[a,b]$ and
\[
\ell_P(\gamma|_{\leftclosedint t_1,t_2 \rightclosedint}) = \int_{t_1}^{t_2} P^*(\gamma'(\tau)) \,{\D}\tau
\]
whenever $a \leq t_1 \leq t_2 \leq b$; now \eqref{eq:metricderivative} follows from the theory of absolutely continuous curves in metric spaces.
\end{proof}

\begin{corollary}
For all $x,y \in M$,
\[
\dist_P(x,y) = \inf \left\{ \int_a^b P^*(\gamma'(t)) \,{\D}t \tc \gamma \in AC([a,b]),\ \gamma(a) = x,\  \gamma(b) = y \right\}.
\]
\end{corollary}

We conclude our discussion of curves and lengths by pointing out that any curve of finite $P$-length may be reparametrised using arc-length, and then becomes a subunit curve, from part (iii) of Proposition \ref{prp:subfinslermetric}.

\subsection{Completeness}

We say that the fibre seminorm $P$ is \emph{complete} if the set $\bar{B}_P(K,R)$ is relatively compact for all compact subsets $K$ of $M$ and all positive $R$.
By Proposition~\ref{prp:minimizingsubunit}, $P$ is complete if and only if $R_P(K) = \infty$ for all compact subsets $K$ of $M$.

\begin{proposition}\label{prop:P-complete-M-complete}
If $P$ is complete, then the metric space $(M,\dist_P)$ is complete.
The converse holds if $\dist_P$ is varietal.
\end{proposition}
\begin{proof}
Suppose that $P$ is complete, and take a $\dist_P$-Cauchy sequence $(x_m)_{m\in\N}$ in $M$.
The set $\{x_m\}_{m\in\N}$ is $\dist_P$-bounded, hence it is relatively compact, thus we may find a subsequence $(x_{n_k})_{k\in\N}$ that converges to a point $x \in M$ in the manifold topology.
By completeness and Proposition~\ref{prp:minimizingsubunit}, the function $\dist_P(x_m, \cdot)$ is lower-semicontinuous, whence
\[
\dist_P(x_m,x) \leq \liminf_{k \to \infty} \dist_P(x_m,x_{n_k}),
\]
and the right-hand side tends to $0$ as $m$ tends to $\infty$ since $(x_m)_{m\in\N}$ is $\dist_P$-Cauchy, hence $\dist_P(x_m,x)$ tends to $0$.

Conversely, suppose that $(M,\dist_P)$ is complete and $\dist_P$ is varietal.
Then each compact subset $K$ of $M$ is $\dist_P$-bounded, so $\bar{B}_P(K,R)$ is also bounded for all positive $R$, and it is closed because $\dist_P$ is continuous.
Since $(M,\dist_P)$ is a complete locally compact length space, closed $\dist_P$-bounded subsets of $M$ are compact \cite[Theorem~1.5]{roe_lectures_2003}, and we are done.
\end{proof}

\begin{definition}\label{defn:support_of_seminorm}
The closed set $\afterline{\{x \in M \tc P_x \neq 0\}}$ is said to be the \emph{support} of the fibre seminorm $P$.
\end{definition}

\begin{proposition}\label{prp:compactsupportcomplete}
If $P$ is compactly-supported, then it is complete.
\end{proposition}
\begin{proof}
If $x \in M \setminus \supp(P)$, then the only $P$-subunit vector in $T_x M$ is the null vector.
Hence all $P$-subunit curves passing through $M \setminus \supp(P)$ are constant, and every point of $M \setminus \supp(P)$ has infinite $\dist_P$-distance to every other point of $M$.
Hence, for all compact subsets $K$ of $M$ and all positive $R$,
\[
\bar{B}_P(K,R) = K \cup \bar{B}_P(K \cap \supp(P), R),
\]
and necessarily $\bar{B}_P(K \cap \supp(P), R) \subseteq \supp(P)$.
Thus $\bar{B}_P(K \cap \supp(P),R)$ is compact, by Proposition~\ref{prp:minimizingsubunit}, and consequently $\bar{B}_P(K,R)$ is compact.
\end{proof}

\subsection{Subunit vector fields and H\"ormander's condition}\label{subsection:hoermander}

We begin with a definition.
\begin{definition}
A smooth section of $TM$ that is $P$-subunit everywhere in $M$ is said to be a $P$-subunit vector field.\footnote{We consider all vector fields to be smooth.}
We write $\mathfrak{X}_P$ for the set of all $P$-subunit vector fields and $L(\mathfrak{X}_P)$ for the Lie algebra of vector fields generated by $\mathfrak{X}_P$.
\end{definition}

Various extended distance functions may be defined as in Definition \ref{def:controldistance}, by restricting $\gamma$ to a subclass of $\Gamma_P^\subunit$.
For example, we might restrict out attention to smooth $P$-subunit curves, or piecewise smooth $P$-subunit curves, or flow curves along $P$-subunit vector fields.
More precisely, the flow curves of a $P$-subunit vector field are smooth $P$-subunit curves, and any concatenation of flow curves of $P$-subunit fields is a $P$-subunit curve; such a concatenation will be called a $P$-subunit piecewise flow curve.

\begin{definition}
We write $\Gamma_P^\infty$ for the set of smooth $P$-subunit curves, $\Gamma_P^\flow$ for the set of $P$-subunit piecewise flow curves, $\dist_P^\infty$ for the extended distance corresponding to the class $\Gamma_P^\infty$ and $\dist_P^\flow$ for the extended distance function corresponding to the class $\Gamma_P^{\flow}$.
\end{definition}

It is not obvious that these three distance functions are the same, however
\begin{equation}\label{eq:distancesinequalities}
\dist_P \leq \dist_P^\infty \leq \dist_P^\flow.
\end{equation}
The second inequality is justified because we obtain the same distance function $\dist_P^\infty$ by taking the class of smooth $P$-subunit curves as by taking the class of piecewise smooth $P$-subunit curves, that is, the $P$-subunit curves $\gamma\colon \leftclosedint a,b\rightclosedint \to M$ for which a finite subdivision $\{ t_0 ,\dots ,t_k \}$ of $[a,b]$ exists such that $\gamma|_{\leftclosedint t_{j-1},t_j\rightclosedint}$ is smooth when $j=1,\dots,k$.
Indeed, for every positive $\epsilon$, there is a smooth increasing bijection $\eta\colon \leftclosedint a,b+\epsilon\rightclosedint \to \leftclosedint a,b\rightclosedint$ such that $\eta' \leq 1$ and $\eta^{(h)}(\eta^{-1}(t_j)) = 0$ when $j=0,\dots,k$ and $h \geq 1$, so the reparametrisation $\gamma \circ \eta\colon \leftclosedint a,b+\epsilon\rightclosedint \to M$ is $P$-subunit and smooth.

\begin{example}
Suppose that $\phi\colon \R \to \R$ is continuous but not differentiable anywhere, and that $\phi(0) = 0$.
Given $(p,q) \in \R^2$, define the seminorm $P_{(p,q)}\colon \R^2 \to \leftclosedint 0,\infty\rightopenint$ by 
\[
P_{(p,q)}(\xi,\eta) = | \xi + \phi(q)\eta|/(1+\phi(q)^2)^{1/2}.
\]
Then
\[
P_{(p,q)}^* (u,v) = 
\begin{cases}
|(u,v)|   &\text{if $(u,v) \in \R (1, \phi(q))$} \\
\infty     &\text{otherwise}.
\end{cases}
\]
It is easy to check that the vector field $\partial /\partial x$ along the $x$ axis does not extend to a $P$-subunit vector field, because we require vector fields to be smooth.
In fact, there are no nonnull $P$-subunit vector fields. It follows that $\dist_P^\flow((0,0), (1,0)) = \infty$, while $\dist_P^\infty((0,0), (1,0)) = 1$.
\end{example}

\begin{definition}
The fibre seminorm $P$ is said to satisfy H\"ormander's condition if $\{ X|_x \tc X \in L(\mathfrak{X}_P) \} = T_x M$ for every $x \in M$.
\end{definition}

\begin{proposition}\label{prp:hoermander}
If $P$ satisfies H\"ormander's condition, then $\dist_P^\flow$ is varietal and \emph{a fortiori} $\dist_P$ and $\dist_P^\infty$ are varietal too.
\end{proposition}
\begin{proof}
Recall that $L(\mathfrak{X}_P)$ is the linear span of the iterated Lie brackets of elements of $\mathfrak{X}_P$.
Hence, for every fixed $x \in M$, there is a finite subset $\mathfrak{X}$ of $\mathfrak{X}_P$ such that the iterated commutators of elements of $\mathfrak{X}$ up to some order, $m$ say, evaluated at $x$, span $T_x M$.
Denote by $\dist_\mathfrak{X}$ the extended distance function corresponding to the class of $P$-subunit curves that are concatenations of flow curves of vector fields in $\mathfrak{X}$.
Then clearly $\dist_P^\flow \leq \dist_\mathfrak{X}$, so, by Chow's theorem (see, for example, \cite[Chapter 2]{montgomery_tour_2002}), for some $g \in \mathfrak{G}_P$ and constant $\kappa$,
\[
B_g(x,r) \subseteq B_{\dist_\mathfrak{X}}(x,\kappa r^{1/m}) \subseteq B_{\dist_P^\flow}(x,\kappa r^{1/m})
\]
for all sufficiently small positive $r$.
Hence the $\dist_P^\flow$ balls centered in $x$ are neighbourhoods of $x$ and, by the arbitrariness of $x$, the $\dist_P^\flow$-open sets are open.
The conclusion follows from the inequalities \eqref{eq:riemanniandistancecomparison} and \eqref{eq:distancesinequalities}.
\end{proof}

We remark that when when the dimension of the spaces of finite vectors varies from point to point, the H\"ormander condition depends on the seminorm $P$ as well as on the vector space of finite vectors. 
For example, take a smooth function $\phi\colon \R \to \R$, and define the seminorm $P_{(p,q)}$ on $\R^2$ by
\[
P_{(p,q)}(u,v) = \bigl( u^2 + \phi^{2}(p) v^2\bigr)^{1/2}.
\]
Then $P$ satisfies H\"ormander's condition if $\phi(p) = p^k$ where $k \in \N$, but not if $\phi$ is the smooth extension of $p \mapsto \E ^{-1/p^2} $ to $\R$ (see the discussion in \S \ref{subsection:nonsmootharclength}).
However, the spaces of finite vectors coincide everywhere for these two examples.

\begin{definition}\label{defn:lipschitz}
The fibre seminorm $P$ is said to satisfy the Lipschitz seminorm condition if, for every $\alpha \in \Alpha$, there is a countable family $\mathfrak{X}$ of $P$-subunit vector fields on $U_\alpha$ and a constant $L$, which may depend on $\alpha$, such that
\begin{enumerate}[(i)]
\item $|\tau_\alpha X(x) - \tau_\alpha X(y)| \leq L |x-y|$ for all $x,y \in B_{\R^n}(0,1)$ and $X \in \mathfrak{X}$, and
\item $\afterline{\{ X|_x \tc X \in \mathfrak{X} \} } = \{ v \in T_x M \tc P^*(v) \leq 1\}$ for all $x \in V_\alpha$.
\end{enumerate}
\end{definition}

Since the $V_\alpha$ are relatively compact in $M$ and form a locally finite cover of $M$, the Lipschitz seminorm condition for $P$ does not depend on the choice of the atlas $\{\phi_\alpha\}_{\alpha \in \Alpha}$.

\begin{theorem}\label{thm:randomapproximation}
Suppose that $P$ satisfies the Lipschitz seminorm condition, and that $\gamma \in \Gamma_P^\subunit( [0,T] )$.
For all neighbourhoods $W$ of $\gamma(T)$ there exists a neighbourhood $U$ of $\gamma(0)$ such that, for all $x \in U$, there exists $\delta \in \Gamma_P^\flow( [0,T] )$ for which $\delta(0) = x$ and $\delta(T) \in W$.
\end{theorem}
\begin{proof}
Fix a riemannian metric $g \in \mathfrak{G}_P$.

With a view to a contradiction, suppose that $\gamma \in \Gamma_P^\subunit( \leftclosedint 0,T\rightclosedint )$ is ``bad'', that is, the conclusion does not hold for $\gamma$.
Clearly $\gamma|_{\leftclosedint 0,T/2\rightclosedint}$ or $\gamma(\cdot + T/2)|_{\leftclosedint 0,T/2\rightclosedint}$ is bad too.
Iteration of this bisection procedure, together with a compactness argument, shows that we may suppose that $\gamma(\leftclosedint 0,T\rightclosedint) \subseteq \bar{B}_g(z,r)$ for some $z \in M$ and $r\in\R^+$ such that $\bar{B}_g(z,3r) \subseteq V_\alpha$ for some $\alpha \in \Alpha$; further iteration allows us to suppose that $T < r$, so
\begin{equation}\label{eq:curve_in_chart}
\gamma(\leftclosedint 0,T\rightclosedint) \subseteq \bar{B}_P(\bar{B}_g(\gamma(0),T),T) \subseteq V_\alpha.
\end{equation}
Now take the countable family $\mathfrak{X}$ of $P$-subunit vector fields $X_k$ and the Lipschitz constant $L$ corresponding to $\alpha$ as in Definition~\ref{defn:lipschitz}.
There is a constant $\kappa$ such that
\begin{equation}\label{eq:eusub}
|\tau_\alpha(v)| \leq \kappa P^*(v) 
\qquant v \in TV_\alpha,
\end{equation}
where $|\cdot|$ denotes the euclidean norm on $\R^n$.

Take $x \in \bar{B}_g(\gamma(0),T)$.
We aim to construct $\delta\colon [0,T] \to M$ that is a piecewise flow curve of fields in $\mathfrak{X}$, such that $\delta(0) = x$, and $\delta(T)$ is arbitrarily near $\gamma(T)$ whenever $x$ is sufficiently near $\gamma(0)$.
The image of any such $\delta$ is contained in $V_\alpha$ by \eqref{eq:curve_in_chart}, therefore from now on we work in the coordinates $\phi_\alpha$.
For simplicity, we continue to write $\gamma$ rather than $\phi_\alpha \circ \gamma$.
Hence 
\begin{equation}\label{eq:gammaintegral}
\gamma(t) = \gamma(0) + \int_0^t \gamma'(\tau) \,{\D}\tau.
\end{equation}
By altering $\gamma'$ on a negligible subset of $\leftclosedint 0,T\rightclosedint$, we may suppose that $\gamma'$ is a Borel function, $\gamma'(t)$ is $P$-subunit for all $t \in \leftclosedint 0,T\rightclosedint$, and \eqref{eq:gammaintegral} still holds.

Fix $\epsilon \in \R^+$.
By the density and smoothness properties of the family $\mathfrak{X}$, the function $\nu_0\colon \leftclosedint 0,T\rightclosedint \to \N$, given by
\[
\nu_0(t) = \min \{ k \in \N \tc | X_k(\gamma(t)) - \gamma'(t) | \leq \epsilon \},
\]
is well-defined and Borel.
Fix $N \in \Z^+$ and set $d = T/N$ and $b(t) = \lfloor t/d \rfloor d$.
Then the function $\nu_1\colon \leftclosedint 0,T\rightclosedint \to \N$, given by
\[
\nu_1(t) = \min \{ k \in \N \tc | X_k(\gamma(b(t))) - X_{\nu_0(t)}(\gamma(b(t))) | \leq \epsilon \},
\]
is also well-defined and Borel; furthermore, since $b$ takes its values in the finite set $\{0,d,2d,\dots,Nd\}$ and the unit $P^*$-ball at $\gamma(jd)$ is compact when $j=0,\dots,N$, the function $\nu_1$ takes a finite number of values too.

Set $I_j = \leftclosedint j d, (j+1) d \rightopenint$, where $j=0,\dots,N-1$.
We define $\nu_2\colon \leftclosedint 0,T\rightclosedint \to \N$ to be the increasing rearrangement of $\nu_1$ on each of the intervals $I_j$, that is, $\nu_2(t) = n$ when $|I_{\lfloor t/d \rfloor} \cap \{ \nu_1 \leq n-1 \}| \leq t - b(t) < |I_{\lfloor t/d \rfloor} \cap \{ \nu_1 \leq n \}|$, and set $\nu_2(T) = \nu_1(T)$.
Hence $\nu_2$ takes a finite number of values and is piecewise constant.
Concatenating flow curves along fields in $\mathfrak{X}$, we define $\delta\colon \leftclosedint 0,T\rightclosedint \to M$ by
\begin{equation}\label{eq:deltaintegral}
\delta(t) = x + \int_0^t X_{\nu_2(\tau)}(\delta(\tau)) \,{\D}\tau.
\end{equation}
We want now to estimate $|\gamma(T) - \delta(T)|$.

Set $D_j = |\delta(jd) - \gamma(jd)|$ when $j=0,\dots,N$.
Clearly
\[
D_{j+1} \leq D_j + \biggl| \int_{I_{j}} (\gamma'(\tau) - X_{\nu_2(\tau)}(\delta(\tau))) \,{\D}\tau \biggr|.
\]
Decompose the integrand as
\[
\begin{aligned}
\gamma'(\tau) - X_{\nu_2(\tau)}(\delta(\tau)) &= \gamma'(\tau) - X_{\nu_0(\tau)}(\gamma(\tau)) \\
&\qquad+ X_{\nu_0(\tau)}(\gamma(\tau)) - X_{\nu_0(\tau)}(\gamma(jd)) \\
&\qquad+ X_{\nu_0(\tau)}(\gamma(jd)) - X_{\nu_1(\tau)}(\gamma(jd)) \\
&\qquad+ X_{\nu_1(\tau)}(\gamma(jd)) - X_{\nu_2(\tau)}(\gamma(jd)) \\
&\qquad+ X_{\nu_2(\tau)}(\gamma(jd)) - X_{\nu_2(\tau)}(\delta(\tau)).
\end{aligned}
\]
The norms of the first and third pieces are at most $\epsilon$, by definition of $\nu_0$ and $\nu_1$.
The second and the fifth pieces are controlled by the Lipschitz seminorm condition, together with inequalities
\[
|\gamma(\tau)-\gamma(jd)| \leq \kappa d
\quad\text{and}\quad
|\delta(\tau)-\gamma(jd)| \leq D_j + \kappa d
\]
for all $\tau \in I_{j}$, by \eqref{eq:eusub}, \eqref{eq:gammaintegral} and \eqref{eq:deltaintegral}.
The fourth piece vanishes after integration over $I_{j}$, because it is the difference of two simple functions, one of which is a rearrangement of the other.
Putting everything together,
\[
\begin{aligned}
D_{j+1} 
&\leq D_j + \epsilon d + L \kappa d^2 + \epsilon d + L (D_j + \kappa d) d \\
& = (1+ L d) D_j + 2 \epsilon d + 2 L \kappa d^2,
\end{aligned}
\]
and by induction,
\[
D_j \leq (1+Ld)^j D_0 + 2(\kappa d+\epsilon/L) \left((1+Ld)^j -1\right).
\]
Since $d = T/N$ and $(1+LT/N)^N \leq \E^{LT}$,
\[
|\delta(T) - \gamma(T)| \leq \E^{LT} |x-\gamma(0)| + 2(\kappa T/N+\epsilon/L) (\E^{LT}-1).
\]
Note now that $T$, $\kappa$ and $L$ do not depend on the parameters $x$, $\epsilon$ and $N$ of the construction.
Hence by taking $N$ sufficiently large, $\epsilon$ sufficiently small, and $x$ sufficiently near $\gamma(0)$, we may construct a subunit piecewise flow curve $\delta$ for which $|\delta(T) - \gamma(T)|$ is arbitrarily small.
This contradicts the badness of $\gamma$ and proves the desired result.
\end{proof}

\begin{corollary}\label{cor:smoothdistanceequality}
Suppose that $P$ satisfies the Lipschitz seminorm condition.
For all $x,y \in M$ such that $x \neq y$,
\begin{equation}\label{eq:liminfdistances}
\dist_P(x,y) \geq \liminf_{z \to y} \dist_P^\flow(x,z).
\end{equation}
If $\dist_P^\flow$ is varietal, then $\dist_P = \dist_P^\flow$.
More generally, if $\dist_P^\infty$ is varietal, then $\dist_P(x,y) = \dist_P^\infty(x,y)$, and both are equal to
\begin{equation}\label{eq:distancelengthsmooth}
\inf \left\{ \ell_P(\gamma) \tc  \gamma \in \Gamma^\infty( \leftclosedint a,b\rightclosedint ),\ \gamma(a) = x, \ \gamma(b) = y\right\}
\end{equation}
for all $x,y \in M$.
\end{corollary}
\begin{proof}
The inequality \eqref{eq:liminfdistances} is trivially satisfied when $\dist_P(x,y) = \infty$.
Otherwise, by Theorem~\ref{thm:randomapproximation}, for all open neighbourhoods $W$ of $y$ that do not contain $x$ and all $T$ greater than $\dist_P(x,y)$, we may find $\delta \in \Gamma_P^\flow( \leftclosedint 0,T\rightclosedint )$ such that $\delta(0) = x$ and $\delta(T) \in W$.
Since $\delta(\leftclosedint 0,T\rightclosedint)$ is connected and $\delta$ is not constant, $\delta(\leftclosedint 0,T \rightclosedint) \cap W \neq \{y\}$; hence we may find $z \in \delta(\leftclosedint0,T\rightclosedint) \cap W \setminus \{y\}$ such that $\dist_P^\flow(x,z) \leq T$.

In particular, if $\dist_P^\flow$ is varietal, then it is continuous, hence
\[
\dist_P^\flow(x,y) \geq \dist_P(x,y) \geq \liminf_{z \to y} \dist_P^\flow(x,z) = \dist_P^\flow(x,y)
\]
by \eqref{eq:distancesinequalities} and \eqref{eq:liminfdistances}, and the equality $\dist_P = \dist_P^\flow$ follows.
In fact, \eqref{eq:distancesinequalities} and \eqref{eq:liminfdistances} also imply that $\dist_P(x,y) \geq \liminf_{z \to y} \dist_P^\infty(x,z)$ when $x \neq y$, therefore the same argument proves that $\dist_P = \dist_P^\infty$ whenever $\dist_P^\infty$ is varietal.
It remains to note that the infimum \eqref{eq:distancelengthsmooth} is at least $\dist_P(x,y)$ by Corollary~\ref{corol:lengthspacesmooth}, and at most $\dist_P^\infty(x,y)$ because $\ell_P(\gamma) \leq T$ for all $\gamma \in \Gamma_P^\infty( \leftclosedint 0,T\rightclosedint )$.
\end{proof}

It is interesting to compare the expressions for the distance as the infimum of the lengths of curves in the preceding corollary, Proposition \ref{prp:lengthspace}, and Corollary \ref{corol:lengthspacesmooth}.
When the function $x \mapsto \dim Z(P_x)$ is continuous, the equality of $\dist_P^\infty(x,y)$ and \eqref{eq:distancelengthsmooth} may be obtained without the hypotheses of Corollary~\ref{cor:smoothdistanceequality}, thanks to the following result.

\begin{proposition}\label{prop:smoothreparametrisation}
Suppose that $x \mapsto \dim Z(P_x)$ is continuous and that $\gamma \in \Gamma^1(\leftclosedint a,b\rightclosedint)$.
Then $\ell_P(\gamma)$ is the infimum of the set of all $T \in \R^+$ for which there is a smooth diffeomorphism $r\colon \leftclosedint 0,T \rightclosedint \to \leftclosedint a,b\rightclosedint $ such that $\gamma \circ r$ is $P$-subunit.
\end{proposition}
\begin{proof}
By Proposition~\ref{prp:subfinslermetric}, $\ell_P(\gamma) = \ell_P(\gamma \circ r) \leq T$ when $\gamma \circ r \in \Gamma_P^\subunit( \leftclosedint 0,T\rightclosedint )$.
Hence we are done if $\ell_P(\gamma) = \infty$.
Otherwise, by Propositions~\ref{prp:constantdimension} and~\ref{prp:lengthintegral}, the set $\{ t \in \leftclosedint a,b \rightclosedint \tc P^*(\gamma'(t)) < \infty\}$ is closed in $\leftclosedint a,b\rightclosedint$ and has full measure, so is all of $\leftclosedint a,b\rightclosedint$.
This means that $P^*(\gamma'(t)) < \infty$ for all $t \in \leftclosedint a,b\rightclosedint$, and hence $t \mapsto P^*(\gamma'(t))$ is continuous by Proposition~\ref{prp:constantdimension} again.

Take now $\epsilon \in \R^+$.
We may find a smooth function $h_\epsilon\colon [a,b] \to \R$ such that
\[
P^*(\gamma'(t)) + \epsilon/2 \leq h_\epsilon(t) \leq P^*(\gamma'(t)) + \epsilon
\qquant t \in \leftclosedint a,b\rightclosedint,
\]
and then define the smooth diffeomorphism $s_\epsilon\colon [a,b] \to [0,s_\epsilon(b)]$ by
\[
s_\epsilon(t) = \int_a^t h_\epsilon(\tau) \,{\D}\tau.
\]
Denote by $r_\epsilon\colon [0,s_\epsilon(b)] \to [a,b]$ the inverse of $s_\epsilon$; then $\gamma \circ r_\epsilon$ is $P$-subunit since $(\gamma \circ r_\epsilon)' = (\gamma' \circ r_\epsilon) / (h_\epsilon \circ r_\epsilon)$.
The conclusion now follows because
\[
s_\epsilon(b) = \int_a^b h_\epsilon(\tau) \,{\D}\tau \to \ell_P(\gamma)
\]
as $\epsilon \to 0$ by Proposition~\ref{prp:lengthintegral}.
\end{proof}

In \S~\ref{subsection:nonsmootharclength}, we show that Proposition \ref{prop:smoothreparametrisation} need not hold if $x \mapsto \dim Z(P_x)$ is not continuous.

\section{The control distance for a differential operator}

Take $D \in \Diff_1(\bundle{E},\bundle{F})$.
We define a continuous fibre seminorm $P_D$ on $T^* M$ by
\[
P_D(\xi) = |\sigma_1(D)(\xi)|_\op.
\]
All the notions introduced in Section \ref{section:subfinsler} in connection with the seminorm $P_D$ may be applied to $D$: we will speak, for example, of $D$-subunit vectors and $D$-subunit curves, and the distance function $ \dist_{P_D}$ will be called the \emph{control distance function} associated to $D$ and written $\dist_D$.
These notions depend only on the seminorm $P_D$, so by \eqref{eq:symbadj}, they do not change if we replace $D$ with $D^+$, or with the operator $\DD \in \Diff_1(\bundle{E} \oplus \bundle{F}, \bundle{E} \oplus \bundle{F})$ given by
\begin{equation}\label{eq:selfadjointreduction}
\DD (f,g) = (D^+g, Df),
\end{equation}
which satisfies $\DD = \DD^+$ and
\[
\sigma_1(\DD)(\xi) = \begin{pmatrix} 0 & \sigma_1(D^+)(\xi) \\ \sigma_1(D)(\xi) & 0 \end{pmatrix},
\]
so $P_{\DD} = P_D = P_{D^+}$.

We will also say that $D$ is complete if $P_D$ is complete.
In particular, by Proposition~\ref{prp:compactsupportcomplete}, $D$ is complete if its symbol $\sigma_1(D)$ is a compactly-supported section of $\Hom(\C T^* M, \Hom(\bundle{E},\bundle{F}))$.

\subsection{The weak differentiability of Lipschitz functions}

Recall from \S~\ref{subsection:hoermander} that a $D$-subunit vector field is a smooth section $X$ of $TM$ that is $D$-subunit at each point of $M$, and that $\dist_D^\flow \geq \dist_D$ because any concatenation of flow curves of $D$-subunit fields is a $D$-subunit curve.

\begin{lemma}\label{lem:lipschitzsubunit}
Suppose that a representative of $f \in L^1_\loc(\bundle{T})$ satisfies
\[
|f(x) - f(y)| \leq L \dist_D^\flow(x,y)
\qquant x,y \in M,
\]
where $L \in \R^+$.
Then, for all $D$-subunit vector fields $X$, the distributional derivative $X f$ is in $L^\infty(\bundle{T})$ and $\| Xf \|_\infty \leq L$.
\end{lemma}
\begin{proof}
Compare with \cite[Theorem~1.3]{garofalo_lipschitz_1998}.
Without loss of generality, we suppose that $L = 1$.

Take $x \in M$, and a smooth bump function $\eta$ that is equal to $1$ in a neighbourhood of $x$.
It is enough to show that $\eta Xf \in L^1_\loc(\bundle{T})$ and $\| \eta Xf \|_\infty \leq 1$.
We may therefore suppose that $\supp X$ is compact and contained in a coordinate chart.
Moreover, since weak derivatives are independent of the measure on $M$, we may suppose that the measure coincides with Lebesgue measure in coordinates.
It will then be sufficient to show that
\[
|\llangle Xf, \overline\phi \rrangle| \leq \|\phi\|_1
\]
for all $\phi \in C^\infty_\comp(\bundle{T})$ with support contained in the coordinate chart.

Now
\[
\llangle Xf, \overline\phi \rrangle
= \llangle f, X^+ \overline\phi \rrangle
= -\int f (x)\, X \phi(x) \,{\D}x - \int f(x) \, (\Div X)(x) \, \phi(x) \,{\D}x.
\]
Denote by $(t,x) \mapsto F_t(x)$ the flow of $X$, so
\[
\begin{aligned}
-\int f(x) \, X \phi(x) \,{\D}x
&= \lim_{t \to 0} \int f(x) \frac{\phi(F_{-t}(x)) - \phi(x)}{t} \,{\D}x \\
&= \lim_{t \to 0} \int \frac{f(F_{t}(x)) - f(x)}{t} \phi(x) \,\det {\D}F_{t}(x) \,{\D}x \\
&\qquad + \lim_{t\to 0} \int f(x) \phi(x) \frac{\det {\D}F_{t}(x) - 1}{t} \,{\D}x.
\end{aligned}
\]
Note that the last limit exists, since $({\D} \det {\D}F_t(x)/ {\D}x) \bigr|_{t=0}  = \Div X(x)$, hence
\[
\llangle Xf, \overline\phi \rrangle
= \lim_{t \to 0} \int \frac{f(F_{t}(x)) - f(x)}{t} \phi(x) \,\det {\D}F_{t}(x) \,{\D}x.
\]
Further, $t \mapsto F_t(x)$ is a flow curve of the $D$-subunit field $X$ for all $x$.
Therefore $\dist_D^\flow(F_t(x),x) \leq |t|$ and so $|f(F_{t}(x))-f(x)|\leq |t|$.
Moreover, $\det {\D}F_0$ is identically equal to $1$, and the desired conclusion follows.
\end{proof}

\begin{proposition}\label{prp:lipschitzsymbol}
Suppose that a representative of $f \in L^1_\loc(\bundle{T}_\R)$ satisfies
\[
|f(x) - f(y)| \leq L \dist_D^\flow(x,y)
\qquant x,y \in M,
\]
where $L \in \R^+$.
Then $f$ is weakly $\Symbol{D}$-differentiable and $\|\,|\Symbol{D} f|_\op\|_\infty \leq L$.
\end{proposition}
\begin{proof}
Again, we suppose that $L=1$.

Consider first the case where $\bundle{E} = \bundle{F}$ and $D = D^+$.
Given any section $V \in C^\infty(\bundle{E})$, define the differential operator $X_V \in \Diff_1(\bundle{T}, \bundle{T})$ by
\[
X_V h = \I V^* (\Symbol{D} h) V = \langle \I \sigma_1(D)({\D}h) V, V \rangle ;
\]
this is the differential operator $\Symbol{D}$ composed with the multiplication operator $g \mapsto \I V^* g V$.
This operator is homogeneous (that is, it annihilates constants) and preserves real-valued functions by \eqref{eq:dist_tauadj}, because $D = D^+$; therefore $X_V$ corresponds to a smooth vector field on $M$.
Moreover, if $\|V\|_\infty \leq 1$, then
\[
|{\D}h(X_V)| = |\langle \I \sigma_1(D)({\D}h) V, V \rangle| \leq |\sigma_1(D)({\D}h)|_\op,
\]
from which it follows that $X_V$ is a $D$-subunit vector field; in this case, therefore, $X_V f \in L^\infty(\bundle{T})$ and $\| X_V f \|_\infty \leq 1$ by Lemma~\ref{lem:lipschitzsubunit}.

More generally, we may define the operators $X_{V,W} h = \I W^* (\Symbol{D} h) V$, and \eqref{eq:dist_tauadj} implies that $\afterline{(X_{V,W} h)} = X_{W,V} \overline{h}$.
Since $f$ is real-valued,
\[
X_{V,W} f = \frac{1}{2} (X_{V+W} f - X_V f - X_W f) + \frac{\I }{2} (X_{V+\I W} f - X_V f - X_{\I W} f) ,
\]
so $\| X_{V,W} f \|_\infty \leq 3$  when $\max\{ \|V\|_\infty, \|W\|_\infty\} \leq 1$.

To prove that $\Symbol{D} f \in L^\infty$, it will be sufficient to show that there is a constant $\kappa$ such that
\[
|\llangle \Symbol{D} f, \phi \rrangle| \leq \kappa \|\phi\|_1
\]
for each $\phi \in C^\infty_\comp(\Hom(\bundle{E},\bundle{E}))$ supported in an open subset $U$ of $M$ in which $\bundle{E}$ is trivialisable.
We may write $\phi$ as
\[
\phi = \sum_{j,k} \phi_{j,k} V_k^* \otimes V_j
\]
for a suitable choice of orthonormal frame $\{V_1,\dots,V_r\}$ of $\bundle{E}|_U$ and sections $\phi_{j,k} \in C^\infty_\comp(\bundle{T}|_U)$, so
\[
\llangle \Symbol{D} f, \phi \rrangle = \sum_{j,k}  \llangle V_j^* (\Symbol{D} f) V_k, \phi_{j,k} \rrangle = - \I \sum_{j,k} \llangle X_{V_k,V_j} f, \phi_{j,k} \rrangle,
\]
hence
\[
|\llangle \Symbol{D} f, \phi \rrangle| \leq 3 \sum_{j,k} \|\phi_{j,k}\|_1 \leq 3 r^2 \|\phi\|_1.
\]

Thus $\Symbol{D} f$ is an $L^\infty$-section of $\Hom(\bundle{E},\bundle{E})$ that satisfies $(\Symbol{D} f)^* = -\Symbol{D} f$ pointwise almost everywhere.
By using local trivialisations, it is easy to construct a countable family of sections $V_m \in C^\infty_\comp(\bundle{E})$ such that $\|V_m \|_\infty \leq 1$ and the set of the $V_m (x)$ of unit norm is dense in the unit sphere of $\bundle{E}_x$ for all $x \in M$.
Thus
\[
|\Symbol{D} f|_\op = \sup_{m\in\N} |\langle (\Symbol{D} f) V_m, V_m\rangle| = \sup_{m\in\N} |X_{V_m} f| \leq 1
\]
pointwise almost everywhere.

In the general case, if $\DD$ is defined as in \eqref{eq:selfadjointreduction}, then $\DD^+ = \DD$ and $\dist_D = \dist_{\DD}$, therefore our last result implies that $\Symbol{\DD} f \in L^\infty$ and $\|\,|\Symbol{\DD} f|_\op\,\|_\infty \leq 1$.
However,
\[
\Symbol{\DD} h = \begin{pmatrix} 0 & \Symbol{(D^+)} h \\ \Symbol{D} h & 0 \end{pmatrix},
\]
therefore weak $\Symbol{\DD}$-differentiability implies weak $\Symbol{D}$-differentiability, and $|\Symbol{\DD} f|_\op = |\Symbol{D} f|_\op$ pointwise almost everywhere.
The conclusion follows.
\end{proof}

Proposition~\ref{prp:lipschitzsymbol} extends to complex-valued functions $f$, by decomposing $f$ in its real and imaginary parts; however in this way one obtains the weaker estimate $|\Symbol{D} f|_\op \leq 2$.
The example where $M = \R^2$, $\bundle{E} = \bundle{T}$, $D = \partial_1 - \I \partial_2$, and $f(x_1,x_2) = x_1 + \I x_2$ shows that this estimate cannot be improved; since $|\Symbol{\DD} f|_\op \geq |\Symbol{D} f|_\op$, the assumption that $D = D^+$ does not help.

A partial converse of Proposition~\ref{prp:lipschitzsymbol} is easily established under additional regularity assumptions on $f$.

\begin{proposition}\label{prp:regularsymbollipschitz}
Suppose that $f \in C^1(\bundle{T}_\R)$ and $\|\,|\Symbol{D} f|_\op\|_\infty \leq L$, where $L \in \R^+$.
Then
\[
|f(x) - f(y)| \leq L\dist_D(x,y)
\qquant x,y \in M.
\]
\end{proposition}
\begin{proof}
Again, we suppose that $L=1$.

Take a $D$-subunit curve  $\gamma\colon [0,T] \to M$ from $x$ to $y$.
Then $f \circ \gamma\colon [0,T] \to \R$ is absolutely continuous and
\[
|(f \circ \gamma)'(t)| = |{\D}f|_{\gamma(t)}(\gamma'(t))| \leq |\sigma_1(D)({\D}f|_{\gamma(t)})|_\op = |\Symbol{D} f(\gamma(t))|_\op \leq 1
\]
for almost all $t \in [0,T]$, since $\gamma'(t)$ is $D$-subunit.
Hence
\[
|f(x)-f(y)| \leq \int_0^T |(f \circ \gamma)'(t)| \,{\D}t \leq T.
\]
The conclusion follows from the arbitrariness of $\gamma$.
\end{proof}

To remove the regularity assumptions on $f$ from the previous statement, we need an extra hypothesis on the $\dist_D$-topology.

\begin{proposition}\label{prp:symbollipschitz}
Suppose that $f \in W^\infty_{\Symbol{D},\loc}(\bundle{T}_\R)$ and $\dist_D$ is varietal.
Then $f$ has a continuous representative.
If also $\|\,|\Symbol{D} f|_\op\|_\infty \leq L$, where $L \in \R^+$, then this continuous representative satisfies
\[
|f(x) - f(y)| \leq L \dist_D(x,y)
\qquant x,y \in M.
\]
\end{proposition}
\begin{proof}
Since $f$ is real-valued, the smooth approximants $J_\veceps f$ given by Theorem~\ref{thm:generalmollifier} are real-valued too, and form a bounded set in $W^\infty_{\Symbol{D},\loc}(\bundle{T})$.

Given any $x \in M$, take $R_x$ in $\leftopenint 0, R_D(\{x\})\rightopenint$ and write $K_x$ for the closed ball $\bar{B}_D(x,R_x)$.
By Proposition~\ref{prp:minimizingsubunit}, for all $y \in K_x$ there exists  a $D$-subunit curve $\gamma\colon [0,T] \to M$ joining $x$ to $y$ such that $T = \dist_D(x,y) \leq R_x$, hence the points of $\gamma$ lie in $K_x$.
Now, arguing as in the proof of Proposition~\ref{prp:regularsymbollipschitz},
\[
|J_\veceps f(x) - J_\veceps f(y)| \leq \dist_D(x,y) \sup_{z\in K_x} \, |\Symbol{D} J_\veceps f(z)|_\op.
\]
Since $\dist_D$ is varietal, the boundedness of the set $\{J_\veceps f\}_{\veceps \in \Epsilon}$  in $W^\infty_{\Symbol{D},\loc}(\bundle{T})$ implies the local equiboundedness and equicontinuity of the $J_\veceps f$, so, by the Arzel\`a--Ascoli theorem, there exists a subsequence of the net $(J_\veceps f)_{\veceps\in\Epsilon}$ that converges uniformly on compacta to a continuous function $g\colon M \to \R$.
Further, $f = g$ pointwise almost everywhere since $J_\veceps f \to f$ in $L^1_\loc(\bundle{E})$, and by replacing $f$ with $g$ we may suppose that $f$ is continuous on $M$.

If moreover $\|\,|\Symbol{D} f|_\op\|_\infty \leq L$ almost everywhere, then Corollary~\ref{cor:bdsmoothapprox} yields a sequence of smooth real-valued functions $f_m$ that converges locally uniformly to $f$, for which $\|\,|\Symbol{D} f_m |_\op\|_\infty \leq L$.
Since these $f_m$ are $L$-Lipschitz with respect to $\dist_D$ by Proposition~\ref{prp:regularsymbollipschitz}, their limit $f$ is $L$-Lipschitz too.
\end{proof}

In general, Propositions~\ref{prp:regularsymbollipschitz} and \ref{prp:symbollipschitz} do not extend to complex-valued functions $f$.
Indeed, suppose that $M = \C = \R^2$, $\bundle{E} = \bundle{T}$, and $D = \partial_1 + \I\partial_2$; then a holomorphic function $f\colon \C \to \C$ satisfies $D f = 0$ but may not be globally Lipschitz.
However, when $D = D^+$, the propositions do extend, since, by \eqref{eq:dist_tauadj},
\[
|\Symbol{D} \operatorname{Re} f|_\op \leq (|\Symbol{D} f|_\op + |\Symbol{D} \overline{f}|_\op)/2 = |\Symbol{D} f|_\op .
\]

We present now a consequence of Proposition~\ref{prp:lipschitzsymbol} that does not require $\dist_D$ to be varietal.

\begin{proposition}\label{prp:sobolev0}
Suppose that $1 \leq p < \infty$.
Then $L^p_c \cap W^p_D(\bundle{E}) \subseteq W^p_{D,0}(\bundle{E})$.
If $D$ is complete, then $W^p_D(\bundle{E}) = W^p_{D,0}(\bundle{E})$.
\end{proposition}
\begin{proof}
If $f \in W^p_D(\bundle{E})$ is compactly-supported, then it may be approximated in $W^p_D$ by compactly-supported smooth sections of $\bundle{E}$, by parts (i), (ii), (iv), and (ix) of Theorem~\ref{thm:generalmollifier}, hence $f \in W^p_{D,0}(\bundle{E})$.

Suppose now that $D$ is complete.
Take a sequence of compact sets $K_m$ in $M$ such that $K_m$ is contained in the interior of $K_{m+1}$ and $\bigcup_m K_m = M$, and define
\[
g_m(x) = (1-n^{-1} \dist_D(K_m,x))_+
\]
for all positive $n$.
By Proposition~\ref{prp:minimizingsubunit}, the functions $g_m$ are upper-semicontinuous and compactly-supported; moreover $0 \leq g_m \leq 1$, $g_m \uparrow 1$ pointwise and
\[
|g_m(x) - g_m(y)| \leq n^{-1} \dist_D(x,y),
\]
so $g_m$ is weakly $\Symbol{D}$-differentiable and $\|\,|\Symbol{D} g_m|_\op\|_\infty \leq n^{-1}$, by Proposition~\ref{prp:lipschitzsymbol}.
If $u \in W^p_D(\bundle{E})$, then, by Proposition~\ref{prp:leibniz},
\[
D (g_m u) = (\Symbol{D} g_m) u + g_m D u,
\]
therefore $g_m u, D(g_m u) \in L^p$ and $(g_m u, D(g_m u)) \to (u,Du)$ in $L^p$ by the dominated convergence theorem.
Since the $g_m u$ are compactly-supported, they belong to $W^p_{D,0}(\bundle{E})$, and we conclude that $u \in W^p_{D,0}(\bundle{E})$ too.
\end{proof}

\subsection{Equivalent definitions of the control distance}

Since the control distance function $\dist_D$ is the distance function associated with the fibre seminorm $P_D$, various equivalent characterisations of $\dist_D$ are contained in Section~\ref{section:subfinsler}.
In particular, the results of \S~\ref{subsection:hoermander} apply because of the following property.

\begin{proposition}
The fibre seminorm $P_D$ satisfies the Lipschitz seminorm condition.
\end{proposition}
\begin{proof}
Since $P_D = P_{\DD}$, it is not restrictive to suppose that $D = D^+$.
We may suppose moreover that the local trivialisations of $\bundle{E}$ are isometric.
Take $\alpha \in \Alpha$, and choose a bump function $\eta$ on $\R^n$ which is equal to $1$ on $B_{\R^n}(0,1)$, and a countable dense set $\mathcal{W}$ of the unit sphere in $\C^r$.
For every $w \in \mathcal{W}$, define $V_w \in C^\infty_\comp(\bundle{E})$ by requiring that $V_w$ is supported in $U_\alpha$ and 
\[
\tau_\alpha(V_w)(x) = \eta(x) w 
\qquant x \in \R^n,
\]
and then define the $D$-subunit field $X_w$ by
\[
X_w h = \I V_w^* (\Symbol{D} h) V_w = \langle \I \sigma_1(D)(dh) V_w, V_w \rangle .
\]
If $D$ is expressed in coordinates, as in \eqref{eq:diffopcoord}, then \eqref{eq:symbolcoord} implies that
\[
\tau_\alpha X_w(x) = ( \I \eta(x)^2 \langle \mat{a}_j(x) w, w \rangle)_{j}
\qquant x \in \R^n,
\]
from which it is clear that the family $\{\tau_\alpha X_w\}_{w \in \mathcal{W}}$ is equi-Lipschitz, with a Lipschitz constant depending on the derivatives of the smooth coefficients $\mat{a}_j$ of $D$.
Moreover, for all $x \in V_\alpha$, the set $\{V_w|_x\}_{w \in \mathcal{W}}$ is dense in the unit sphere of $\bundle{E}_x$, so
\[
P_D(\xi) = |\sigma_1(D)(\xi)|_\op = \sup_{w \in W} |\xi(X_w|_x)|
\]
for all $x \in V_\alpha$ and $\xi \in T^*_x M$.
The bipolar theorem \cite[Section 20.8]{koethe_topological_1969} implies that, for all $x \in V_\alpha$, the set $\{v \in T_x M \tc P^*_D(v) \leq 1\}$ is the closed convex envelope of $\{\pm X_w|_x \tc w \in \mathcal{W}\}$.
Hence the set $\mathfrak{X}$ of convex combinations with rational coefficients of elements of $\{\pm X_w \tc w \in \mathcal{W}\}$ is a countable family of compactly-supported $D$-subunit fields, such that $\{\tau_\alpha X\}_{X \in \mathfrak{X}}$ is equi-Lipschitz, and $\{X|_x\}_{X \in \mathfrak{X}}$ is dense in $\{v \in T_x M \tc P^*_D(v) \leq 1\}$ for all $x \in V_\alpha$.
\end{proof}

The following result is an immediate consequence of Proposition~\ref{prp:hoermander} and Corollary~\ref{cor:smoothdistanceequality}.

\begin{corollary}
If $P_D$ satisfies H\"ormander's condition, then $\dist_D^\flow$ is varietal, and, $\dist_D(x,y)  = \dist_D^\infty(x,y)  = \dist_D^\flow(x,y) $ for all $x,y \in M$; further, all are equal to 
\[
\inf \left\{ \ell_D(\gamma) \tc  \text{$\gamma \in \Gamma^\infty (\leftclosedint a,b\rightclosedint) $},\  \gamma(a) = x, \ \gamma(b) = y \right\} .
\]

\end{corollary}

Another characterisation of the control distance may be given in terms of smooth functions with ``bounded gradient''.

\begin{proposition}\label{prp:distancesmoothfunctions}
Suppose that $\dist_D$ is varietal.
Then
\begin{equation}\label{eq:distancesmoothfunctions}
\dist_D(x,y) = \sup \left\{|\xi(x) - \xi(y)| \tc  \xi \in C^\infty(\bundle{T}_\R) ,\ \| |\Symbol{D} \xi|_\op \|_\infty \leq 1\right\}.
\end{equation}
If $D$ is complete, then the supremum may be restricted to $\xi$ in $C^\infty_\comp(\bundle{T}_\R)$.
\end{proposition}
\begin{proof}
The left-hand side of \eqref{eq:distancesmoothfunctions} is greater than or equal to the right-hand side, without any assumptions on $D$, from  Proposition~\ref{prp:regularsymbollipschitz}.
For the reverse inequality, take $x,y \in M$ and $\lambda \in\leftopenint0, \dist_D(x,y)\rightopenint$, and define
\[
f(z) = (\lambda - \dist_D(x,z))_+
\qquant z \in M.
\]
Then $f$ is finite and continuous; moreover, by Proposition~\ref{prp:lipschitzsymbol}, $f$ is weakly $\Symbol{D}$-differentiable and $\|\,|\Symbol{D} f|_\op\|_\infty \leq 1$.
Therefore, by Corollary~\ref{cor:bdsmoothapprox}, there is a sequence of real-valued smooth functions $f_m$ such that $|\Symbol{D} f_m|_\op \leq 1$ and $f_m$ converges locally uniformly to $f$ ; thus
\[
|f_m(x) - f_m(y)| \to |f(x) - f(y)| = \lambda,
\]
and the first part of the conclusion follows by the arbitrariness of $\lambda$.
If $D$ is complete, then the function $f$ is compactly-supported, and by Corollary~\ref{cor:bdsmoothapprox} the smooth approximants $f_m$ may also be chosen compactly-supported.
\end{proof}

Now we are going to show that the characterisation of $\dist_D$ given by Proposition~\ref{prp:distancesmoothfunctions} may hold even when $\dist_D$ is not varietal.

Fix a riemannian metric $g$ on $M$.
This induces a fibre inner product on $\C T^* M$.
For all $m \in \N$, define $D_m \in \Diff_1(\bundle{E} \oplus \bundle{T}, \bundle{F} \oplus \C T^* M)$ by
\begin{equation}\label{eqn:def-of-Dn}
D_m(f,g) = (Df, 2^{-m} dg).
\end{equation}
Then
\[
\Symbol{(D_m)} h = \begin{pmatrix} \Symbol{D} h & 0 \\ 0 & 2^{-m} dh\end{pmatrix},
\]
so
\[
P_{D_m}(\xi) = \max \{P_D(\xi), 2^{-m} |\xi|_g\}
\qquant \xi \in T^* M .
\]

In particular, a vector $v \in T M$ is $D$-subunit if and only if it is $D_m$-subunit for all $m \in \N$.
Moreover, $\dist_{D_m} \leq 2^m \dist_g$, so $\dist_{D_m}$ is varietal for all $m \in \N$.

\begin{proposition}\label{prp:distancesupremumdistances}
Suppose that $D_0$, given by \eqref{eqn:def-of-Dn}, is complete.
Then
\[
\dist_D(x,y) = \sup_{m \in \N} \dist_{D_m}(x,y)
\qquant  x,y \in M.
\]
\end{proposition}
\begin{proof}
Recall from Proposition~\ref{prp:subfinslermetric} that a curve is $P$-subunit if and only if it is $1$-Lipschitz with respect to $\dist_P$.
From the definition of $D_0$, it is clear that $\dist_D \geq \dist_{D_m}$ for all $m \in \N$.

Fix $x,y \in M$ such that $\sup_{m \in \N} \dist_{D_m}(x,y)$ is finite, and take a finite $T$ greater than the supremum.
For all $m \in \N$, we choose a $D_m$-subunit curve $\gamma_m\colon \leftclosedint0,T\rightclosedint \to M$ such that $\gamma_m(0) = x$ and $\gamma_m(T) = y$.
If $m \leq m'$, then $\dist_{D_m} \leq \dist_{D_{m'}}$ by definition, hence $\gamma_{m'}$ is also $D_m$-subunit.

Consequently, all the curves $\gamma_m$ are $1$-Lipschitz with respect to $\dist_{D_0}$, and all take their values in $\bar{B}_{D_0}(x,T)$, which is compact because $D_0$ is complete.
By the Arzel\`a--Ascoli theorem, there is a subsequence of $(\gamma_m)_{m \in \N}$ that converges uniformly to a continuous curve $\gamma\colon [0,T] \to M$.
Then $\gamma(0) = x$ and $\gamma(T) = y$; further, $\gamma$ is $1$-Lipschitz with respect to all the $\dist_{D_m}$, that is, $\gamma$ is $D_m$-subunit for all $m \in \N$.
Thus $\gamma'(t)$ is $D_m$-subunit for almost every $t \in \leftclosedint0,T\rightclosedint$ and all $m\in\N$, which implies that $\gamma'(t)$ is $D$-subunit.
Hence $\gamma$ is $D$-subunit, and $\dist_D(x,y) \leq T$.
\end{proof}

\begin{corollary}
Suppose that $D_0$, given by \eqref{eqn:def-of-Dn}, is complete.
Then, for all $x,y \in M$,
\[
\dist_D(x,y) 
= \sup \left\{|\xi(x) - \xi(y)| \tc  \xi \in C^\infty_\comp(\bundle{T}_\R) ,\ \| |\Symbol{D} \xi|_\op\|_\infty \leq 1 \right\}.
\]
\end{corollary}
\begin{proof}
Fix $x,y \in M$, and take $\lambda$ less than $\dist_D(x,y)$.
By Proposition~\ref{prp:distancesupremumdistances}, there exists $m \in \N$ such that $\lambda < \dist_{D_m}(x,y)$.
Since $\dist_{D_m}$ is varietal and $D_m$ is complete, there exists $\xi \in C^\infty_\comp(\bundle{T}_\R)$ such that $|\xi(x) - \xi(y)| > \lambda$ and $\|\,|\Symbol{D_m} \xi|_\op\|_\infty \leq 1$ by Proposition~\ref{prp:distancesmoothfunctions}.
However,
\[
|\Symbol{D_m} \xi|_\op = \max\{|\Symbol{D} \xi|_\op, 2^{-m} |{\D}\xi|_g\} \geq |\Symbol{D} \xi|_\op,
\]
so $\|\,|\Symbol{D}\xi|_\op\|_\infty \leq 1$.
\end{proof}

In general, the completeness of $D_0$ depends on the choice of the riemannian metric $g$ in \eqref{eqn:def-of-Dn}.
However, if $M$ is compact, then any riemannian metric $g$ gives a complete $D_0$.
Moreover, if $D_0$ is complete for some $g$, then $D$ is complete too.
The following examples show that the completeness hypothesis is stronger than necessary to ensure \eqref{eq:distancesmoothfunctions}, but that it does not always hold.

\begin{example}
Suppose that $M = \leftopenint 0, 1 \rightopenint^2$ and $D = \sfrac{\partial}{\partial x_1}$.
Then
\[
\dist_D(x,y) = \begin{cases}
|x_1 - y_1| &\text{if $x_2 = y_2$,}\\
\infty &\text{otherwise,}
\end{cases}
\]
and \eqref{eq:distancesmoothfunctions} holds without $\dist_D$ being varietal or $D$ being complete.
\end{example}

\begin{example}
Suppose that $M = \R^2 \setminus \{(0,0)\}$ and $D = \sfrac{\partial}{\partial x_1}$.
Then 
\[
\dist_D((-1,0),(1,0)) = \infty.
\]
However, for all smooth $\xi\colon M \to \R$, the condition that $\|D\xi\|_\infty \leq 1$ implies that $|\xi(-1,t) - \xi(1,t)| \leq 2$ when $t \neq 0$ by the mean value theorem, and hence $|\xi(-1,0) - \xi(1,0)| \leq 2$ by continuity.
\end{example}

\section{The $L^2$ theory: formal and essential self-adjointness}

Given a differential operator $D \in \Diff_k(\bundle{E},\bundle{F})$, we denote for the moment by $D_s$ its restriction to compactly-supported smooth sections and by $D_d$ its extension to distributions.
Then $D_s$ may be thought of as a densely defined operator $L^2(\bundle{E}) \parto L^2(\bundle{F})$, and we may consider its Hilbert space adjoint $(D_s)^*\colon L^2(\bundle{F}) \parto L^2(\bundle{E})$. 
It is easily checked that the domain of $(D_s)^*$ is the space $W_{D^+}^2(\bundle{F})$, that is,
\[
\{ f \in L^2(\bundle{F}) \tc D^+_d f \in L^2(\bundle{E})\},
\]
and $(D_s)^*$ is the restriction of $D^+_d$ to this domain.
In particular, $(D_s)^* \supseteq D^+_s$, so $(D_s)^*$ is densely defined, $D_s$ is closable and
\[
\afterline{(D_s)} = (D_s)^{**} \subseteq (D^+_s)^*.
\]
The domains of $\afterline{(D_s)}$ and $(D^+_s)^*$ will be called the minimal and maximal domains of $D$ respectively; note that the maximal domain of $D$ is $W_{D}^2(\bundle{E})$, whereas the minimal domain of $D$ is $W^2_{D,0}(\bundle{E})$.
If $D$ is formally self-adjoint, that is, $\bundle{E} = \bundle{F}$ and $D = D^+$, then
\[
\afterline{(D_s)} = (D_s)^{**} \subseteq (D_s)^*,
\]
and clearly $\afterline{(D_s)}$ and $(D_s)^*$ are the minimal and maximal closed symmetric extensions of $D_s$ respectively; the essential self-adjointness of $D_s$ is thus equivalent to the equality of the minimal and maximal domains of $D$.

Henceforth, we will write $D$, $D^*$ and $\afterline{D}$ instead of $D_s$, $(D_s)^*$ and $\afterline{(D_s)}$.

We now rephrase the content of Proposition~\ref{prp:sobolev0} when $p=2$.

\begin{proposition}\label{prp:compactdomains}
Suppose that $D \in \Diff_1(\bundle{E},\bundle{F})$.
A compactly-supported section $f \in L^2(\bundle{E})$ belongs to the maximal domain of $D$ if and only if it belongs to its minimal domain.
If moreover $D$ is complete, then the minimal and maximal domains of $D$ coincide.
\end{proposition}

\begin{corollary}\label{cor:selfadjoint}
Suppose that $D \in \Diff_1(\bundle{E},\bundle{E})$ is complete and formally self-adjoint.
Then $D$ is essentially self-adjoint.
\end{corollary}

\section{Finite propagation speed}

Take a formally self-adjoint element $D \in \Diff_1(\bundle{E},\bundle{E})$.
We say that $u_t$ is a \emph{solution} of
\begin{equation}\label{eq:onde}
\frac{\D}{{\D}t} u_t = \I D^* u_t
\end{equation}
if $t \mapsto u_t$ is an $L^2(\bundle{E})$-valued map defined on a subinterval $I$ of $\R$, which is continuously differentiable on $I$ (as an $L^2(\bundle{E})$-valued map), takes its values in the domain of $D^*$ and satisfies \eqref{eq:onde} for all $t \in I$; we say moreover that $u_t$ is \emph{energy-preserving} if $t \mapsto \| u_t \|_2$ is constant.

If $D$ admits a self-adjoint extension $\tilde D$, so
\[
\afterline{D} \subseteq \tilde D = (\tilde D)^* \subseteq D^*,
\]
then $u_t = \E^{\I t\tilde D} f$ is an energy-preserving solution of \eqref{eq:onde} for all $f$ in the domain of $\tilde D$, since $\E^{\I t\tilde D}$ is unitary.
In fact, $u_t = \E^{\I t\tilde D} f$ is defined for an arbitrary $f \in L^2(\bundle{E})$, but need not be differentiable in $t$, and satisfies an integral version of the equation \eqref{eq:onde}, that is,
\[
u_t = u_0 + \I D^* \int_0^t u_s \,{\D}s
\]
\cite[Lemma~II.1.3]{engel_oneparameter_2000}; however such a ``mild solution'' of \eqref{eq:onde} may be approximated by ``classical solutions'' because the domain of $\tilde D$ is dense in $L^2(\bundle{E})$.

In any case, for an arbitrary $D$, compactly-supported solutions automatically preserve energy.

\begin{proposition}\label{prp:energy}
If $u_t$ is a solution of \eqref{eq:onde}, defined on an interval $I$, and $\supp u_t$ is compact for all $t \in I$, then $u_t$ is energy-preserving.
\end{proposition}
\begin{proof}
The function $t \mapsto \| u_t \|_2^2$ is differentiable, with derivative
\[
\I \llangle D^* u_t, u_t \rrangle - \I \llangle u_t, D^* u_t \rrangle.
\]
Since $\supp u_t$ is compact, $u_t$ is in the domain of $\afterline{D}$ by Proposition~\ref{prp:compactdomains}, therefore
\[
\llangle u_t, D^* u_t \rrangle = \llangle \afterline{D} u_t, u_t \rrangle = \llangle D^* u_t, u_t \rrangle,
\]
and hence the derivative of $t \mapsto \| u_t \|_2^2$ is identically null.
\end{proof}

The relationship between preservation of energy and compactness of support may be partially reversed.

\begin{theorem}\label{thm:propagation}
Suppose that $K \subseteq W$, where $K\in \mathfrak{K}(M)$ and $W \in \mathfrak{O}(M)$.
There exists $\epsilon$, depending on $K$ and $W$,  such that, for all energy-preserving solutions $u_t$ of \eqref{eq:onde} defined on an interval $I$ containing $0$, if
\[
\supp u_0 \subseteq K
\]
then
\[
\supp u_t \subseteq W
\]
for all $t \in I \cap \leftopenint -\epsilon,\epsilon\rightopenint $.
\end{theorem}
\begin{proof}
See \cite[Proposition 10.3.1]{higson_analytic_2000}.

Choose a bump function $g$ that is $1$ on $K$ and supported in $W$, and take $\epsilon$ less than $\|\,|\Symbol{D} g|_\op\|_\infty^{-1}$.
Choose also a smooth nondecreasing function $\phi\colon \R \to [0,1]$ such that $\phi(t) < 1$ when $t < 1$ and $\phi(t) = 1$ when $t \geq 1$.

Define functions $h_t\colon M \to [0,1]$ for all $t \in \leftclosedint 0,\infty\rightopenint$ by
\[
h_t(x) = \phi(g(x) + \epsilon^{-1} t).
\]
Now $h_t$ depends smoothly on $t$, with derivative (relative to $t$) given by
\[
\dot h_t(x) = \epsilon^{-1} \phi'(g(x) + \epsilon^{-1} t) \geq 0,
\]
since $\phi$ is smooth and nondecreasing; further, by the chain rule,
\[
\Symbol{D} h_t(x) = \phi'(g(x) + \epsilon^{-1} t) \, \Symbol{D} g(x) = \epsilon \, \dot h_t(x) \, \Symbol{D} g(x).
\]
Consequently,
\begin{equation}\label{eq:pointwisepositive}
\dot h_t -\I \Symbol{D} h_t = \dot h_t (1 - \I \epsilon \Symbol{D} g) \geq 0
\end{equation}
pointwise as a section of $\Hom(\bundle{E},\bundle{E})$, since $\I \epsilon \Symbol{D} g$ is pointwise self-adjoint by \eqref{eq:dist_tauadj}, and $\|\,|\I \epsilon \Symbol{D} g|_\op\|_\infty \leq 1$.

Take $u_t$ as in the hypotheses.
Then, by Leibniz' rule,
\[
D(h_t u_t) = (\Symbol{D} h_t) u_t + h_t D u_t.
\]
Now $h_t u_t \in L^2$ and $D(h_t u_t) \in L^2$ since $u_t \in L^2$, $Du_t \in L^2$, $h_t \in L^\infty$, and $\Symbol{D} h_t \in L^\infty$.
Moreover $h_t u_t$ is compactly-supported, so belongs to the domain of $\afterline{D}$ by Proposition~\ref{prp:compactdomains}, and
\[
\llangle h_t u_t, D^* u_t \rrangle = \llangle \afterline{D} (h_t u_t), u_t \rrangle = \llangle D^* (h_t u_t), u_t \rrangle,
\]
hence
\[
\begin{aligned}
\frac{\D}{{\D}t} \llangle h_t u_t, u_t \rrangle
&= \llangle \dot h_t u_t, u_t \rrangle +\I \llangle h_t D^* u_t, u_t \rrangle -\I \llangle h_t u_t, D^* u_t \rrangle \\
&= \llangle \dot h_t u_t, u_t \rrangle -\I \llangle (\Symbol{D} h_t) u_t, u_t \rrangle \geq 0
\end{aligned}
\]
by Leibniz' rule and \eqref{eq:pointwisepositive}.
Therefore, for all positive $t$,
\[
\llangle h_t u_t, u_t \rrangle \geq \llangle h_0 u_0, u_0 \rrangle = \llangle u_0, u_0 \rrangle = \llangle u_t, u_t \rrangle,
\]
since $u_t$ is energy-preserving, $\supp u_0 \subseteq K$ and $h_0$ is equal to $1$ on $K$.
But then $h_t u_t = u_t$ almost everywhere, since  $0 \leq h_t \leq 1$.
Note that $h_t = \phi(\epsilon^{-1} t) < 1$  on the open set $M \setminus \supp g$ when $t < \epsilon$; hence $\supp u_t \subseteq \supp g \subseteq W$.

The case where $t < 0$ may be treated by replacing $D$ with $-D$ and $u_t$ with $u_{-t}$.
\end{proof}

As a consequence, we establish the uniqueness of energy-preserving solutions of \eqref{eq:onde} for small times and compactly-supported initial datum.

\begin{corollary}\label{cor:uniquenessenergypreserving}
Suppose that $K$ is a compact subset of $M$.
There exists $\epsilon \in \R^+$, depending on $K$, such that, for all $f \in L^2(\bundle{E})$ for which $\supp f \subseteq K$, two energy-preserving solutions $u^{}_t$ and $v^{}_t$ of \eqref{eq:onde} that satisfy $u^{}_0 = v^{}_0 = f$ coincide when $|t| < \epsilon$.
In particular, when $|t|  < \epsilon$, the value of $\E^{\I t\tilde D} f$ does not depend on the self-adjoint extension $\tilde D$ of $D$.
\end{corollary}
\begin{proof}
Take a relatively compact open neighbourhood $W$ of $K$ in $M$, and take $\epsilon$, depending on $K$ and $W$, as in Theorem~\ref{thm:propagation}.

Write $w_t$ for $u^{}_t - v^{}_t$.
The $w_t$  is a solution of \eqref{eq:onde}, and $\supp w_t \subseteq W$ when $|t| < \epsilon$ by Theorem~\ref{thm:propagation}, therefore $w_t$ is energy-preserving when $|t| < \epsilon$ by Proposition~\ref{prp:energy}, and the conclusion follows since $\|w_0\|_2 = 0$.
\end{proof}

\subsection{Propagation and the control distance}

By using the control distance function $\dist_D$ associated to $D = D^+ \in \Diff_1(\bundle{E},\bundle{E})$, we establish a quantitative version of Theorem~\ref{thm:propagation}.
Recall Definition \ref{eqn:def-R_P(K)} of $R_D(K)$. 

\begin{theorem}\label{thm:propagationdistance}
Suppose that $K$ is a compact  subset of $M$.
If $U$ is an energy-preserving solution of \eqref{eq:onde} defined on an interval $I$ containing $0$ and
\[
\supp u_0 \subseteq K,
\]
then
\[
\supp u_t \subseteq \bar{B}_D(K,|t|)
\]
for all $t \in I \cap \leftopenint -R_D(K), R_D(K)\rightopenint$.
\end{theorem}
\begin{proof}
It suffices to prove that $\supp u_t \subseteq \bar{B}_D(K,\epsilon)$ when $|t| < \epsilon$, since
\[
\bar{B}_D(K,|t|) = \bigcap_{\delta > |t|} \bar{B}_D(K,\delta).
\]

Take any $\epsilon$ such that $|t| < \epsilon < R_D(K)$.
We now follow the proof of Theorem~\ref{thm:propagation}, with one modification: we define $g$, which is no longer smooth, by
\[
g(x) = (1-\epsilon^{-1} \dist_D(K,x))_+
\qquant x \in M.
\]
Again by Proposition~\ref{prp:minimizingsubunit}, $\bar{B}_D(K,r)$ is compact when $r \leq \epsilon$, and hence $g$ is upper-semicontinuous; moreover it is clear that $0 \leq g \leq 1$, that $g = 1$ on $K$, that $\supp g \subseteq \bar{B}_D(K,\epsilon)$ and that
\[
|g(x) - g(y)| \leq \epsilon^{-1} \dist_D(x,y),
\]
so $g$ is weakly $\Symbol{D}$-differentiable and $\|\,|\Symbol{D} g|_\op\|_\infty \leq \epsilon^{-1}$ by Proposition~\ref{prp:lipschitzsymbol}.
The steps of the proof of Theorem~\ref{thm:propagation} may now be repeated, interpreting $\Symbol{D}$-derivatives in the weak sense, and using Propositions~\ref{prp:leibniz} and \ref{prp:chain} whenever Leibniz' rule and the chain rule are invoked.
\end{proof}

A quantitative version of Corollary~\ref{cor:uniquenessenergypreserving} on uniqueness of energy-preserving solutions may be derived as before.
In fact, with a little more effort, we also establish an existence result.
To avoid boundary value problems, we restrict attention to the interval $\leftopenint -R_D(K),R_D(K)\rightopenint $. 

\begin{theorem}
Suppose that $K \in \mathfrak{K}(M)$ and that $f \in W^2_D(\bundle{E})$ is supported in $K$.
Then there exists an energy-preserving solution $u_t$ of \eqref{eq:onde} on the interval $\leftopenint -R_D(K),R_D(K)\rightopenint $ such that $u_0 = f$; moreover, any other energy-preserving solution of \eqref{eq:onde} with initial datum $f$ coincides with $u_t$ on the intersection of their domains.
\end{theorem}
\begin{proof}
All energy-preserving solutions $u_t$ such that $\supp u_0 \subseteq K$ remain compactly-supported when $|t| < R_D(K)$ by Theorem~\ref{thm:propagationdistance}, so uniqueness of the solution on $\leftopenint -R_D(K),R_D(K)\rightopenint $ is proved as in Corollary~\ref{cor:uniquenessenergypreserving}.
It remains to show the existence of an energy-preserving solution on all intervals $[-R,R]$ where $R < R_D(K)$.

Note that, if $D$ is complete, then $D^*$ is self-adjoint by Corollary~\ref{cor:selfadjoint}, and $u_t = \E^{\I t D^*} f$ is the required solution.
In the general case, take a bump function $\eta$ that is equal to $1$ on a neighbourhood of $\bar{B}_D(K,R)$, and define $D_0 \in \Diff_1(\bundle{E},\bundle{E})$ by
\[
D_0 f = \frac{1}{2} (\eta Df + D(\eta f)) = \eta Df + \frac{1}{2} (\Symbol{D} \eta) f.
\]
Then it is easily checked that $D_0$ is formally self-adjoint and $\sigma_1(D_0) = \eta \sigma_1(D)$ is compactly-supported, therefore $D_0$ is complete by Proposition \ref{prop:P-complete-M-complete}, and we may take $u_t = \E^{\I tD_0^*}f$, which is an energy-preserving solution of
\[
\frac{\D}{{\D}t} u_t = \I D_0^* u_t
\]
for all $t \in \R$.
Since $\sigma_1(D_0) = \eta \sigma_1(D)$ and $|\eta| \leq 1$, all $D_0$-subunit vectors are $D$-subunit, hence $\dist_D \leq \dist_{D_0}$ and consequently, by Theorem~\ref{thm:propagationdistance},
\[
\supp u_t \subseteq \bar{B}_{D_0}(K,R) \subseteq \bar{B}_D(K,R)
\]
for all $|t| \leq R$.
Moreover, $D$ and $D_0$ coincide on a neighbourhood of $\bar{B}_D(K,R)$ by construction, therefore $u_t$ is a solution of \eqref{eq:onde} when $|t| \leq R$.
\end{proof}

\subsection{Second-order operators}\label{subsection:secondorder}

Consider now the second-order equation
\begin{equation}\label{eq:onde_2}
\Bigl(\frac{\D}{{\D}t}\Bigr)^2 u_t = -L^* u_t,
\end{equation}
for some positive $L \in \Diff_2(\bundle{E},\bundle{E})$.
Suppose that $\tilde L$ is a positive self-adjoint extension of $L$, denote the continuous extension of $\lambda \mapsto \lambda^{-1} \sin \lambda$ to $\R$ by $\sinc$, and define $u_t$ by 
\begin{equation}\label{eq:sol_onde_2}
u_t = \cos(t \tilde L^{1/2}) f + t \sinc(t \tilde L^{1/2}) g.
\end{equation}
It is well-known that $u_t$ satisfies \eqref{eq:onde_2} together with the initial conditions $u_0 = f$ and $\dot u_0 = g$, at least when $f$ is in the domain of $\tilde L$ and $g$ is in the domain of $\tilde L^{1/2}$.

Suppose that $L$ factorises as $L = D^+ D$ for some $D \in \Diff_1(\bundle{E},\bundle{F})$, and recall that $\DD = D \oplus D^+$.
If $\tilde{\DD}$ is any self-adjoint extension of $\DD$, then $\tilde{\DD}^2$ preserves the decomposition of $L^2(\bundle{E} \oplus \bundle{F})$ as $L^2(\bundle{E}) \oplus L^2(\bundle{F})$, and $\tilde{\DD}^2 (f,0) = (\tilde L f, 0)$ for some positive self-adjoint extension $\tilde L$ of $L$.
In particular,
\begin{align*}
(\cos(t\tilde L^{1/2}) f, 0) 
&= \cos(t (\tilde{\DD}^2)^{1/2}) (f,0) = \cos(t \tilde{\DD}) (f,0) \\
&= \frac{\E^{\I t\tilde{\DD}} + \E^{- \I t\tilde{\DD}}}{2} (f,0),
\end{align*}
because the cosine function is even, and moreover
\[
\frac{\D}{{\D}t} (t \sinc(t \tilde L^{1/2}) g) = \cos(t \tilde L^{1/2}) g.
\]
Therefore if $\supp f \,\cup\, \supp g$ is compact and $u_t$ is defined by \eqref{eq:sol_onde_2}, then from Theorem~\ref{thm:propagationdistance} we deduce that
\[
\supp u_t \subseteq \bar{B}_D(\supp f \cup \supp g,|t|)
\]
whenever $|t| < R_D(\supp f \cup \supp g)$.

Note that
\[
P_D(\xi) = |\sigma_1(D)(\xi)|_\op = ( |\sigma_2(L)(\xi^{\odot 2})|_\op )^{1/2}
\]
by \eqref{eq:symbcomp} and \eqref{eq:symbadj}, so the fibre seminorm $P_D$ and the associated distance function may be expressed directly in terms of the second-order symbol of $L$.

When $P_D$ is complete, $\DD$ is essentially self-adjoint by Corollary~\ref{cor:selfadjoint}.
In fact the smoothness of solutions of symmetric hyperbolic systems with smooth coefficients (see \cite[Section 7.6]{alinhac_hyperbolic_2009} for an elementary proof), together with the finite propagation speed, implies that the operators $\E^{\I t\DD^*}$ preserve $C^\infty_c(\bundle{E} \oplus \bundle{F})$, and an argument of Chernoff \cite[Lemma 2.1]{chernoff_essential_1973} proves that $\DD^2$ is essentially self-adjoint too.
In particular, if  $L = D^+ D$, then $L$ is essentially self-adjoint, and $W^2_L(\bundle{E}) \subseteq W^2_D(\bundle{E})$ with continuous inclusion, because
\[
\llangle Df, Df \rrangle = \llangle f, Lf \rrangle \leq \|f\|_2 \|Lf\|_2
\]
for every $f$ in the maximal domain of $L$.
It is then not difficult to deduce that, for all maps $t \mapsto u_t$ in $C^2(I; L^2(\bundle{E}))$ that satisfy \eqref{eq:onde_2}, the equality
\[
v_t = (\dot u_t, \I D u_t)
\]
defines a mild solution of
\[
\frac{\D}{{\D}t} v_t = \I \DD^* v_t,
\]
and consequently $v_t = \E^{\I  t\DD^*} v_0$ (see \cite[Propositions VI.3.2 and II.6.4]{engel_oneparameter_2000}).
This implies that \eqref{eq:onde_2} has a unique solution for given initial data $u_0 = f$ and $\dot u_0 = g$, that is,
\[
u_t = \cos(t (L^*)^{1/2}) f + t \sinc(t (L^*)^{1/2}) g.
\]

\section{Examples}
This section contains examples that illustrate our theory.
We begin with a discussion of multilinear algebra, then pass to the examples.
Most of these are concerned with applications to differential operators, but the final example shows that smooth subunit parametrisations of smooth curves may not enable us to compute length.
\subsection{Preliminaries on multilinear algebra}

Suppose that $V$ is an $n$-dimensional vector space over $\C$.
As usual, if $\{v_1,\dots,v_n\}$ is a basis of $V$ and $J = \{j_1,\dots,j_k\}$, where $1 \leq j_1 < \dots < j_k \leq n$, then we define the element $v_J$ of the exterior algebra $\bigoplus_{k=0}^n \Ext^k V$ by
\[
v_J = v_{j_1} \wedge \dots \wedge v_{j_k} .
\]
When $V$ is endowed with a hermitean inner product $\langle\cdot,\cdot\rangle$, there exists a unique hermitean inner product $\langle\cdot,\cdot\rangle$ on the exterior algebra $\Ext V$ such that $\Ext^k V \perp \Ext^{k'} V$ when $k \neq k'$ and, for every orthonormal basis $\{v_1,\dots,v_n\}$ of $V$, the multivectors $v_J$, where $J$ varies over the $k$-element subsets of $\{1 \dots, n\}$, form an orthonormal basis of $\Ext^k V$ when $k=0,\dots,n$.

Given any $\alpha,\beta \in \Ext V$, we define $\alpha \vee \beta \in \Ext V$ by requiring that
\[
\langle \alpha \vee \beta , \gamma \rangle = \langle \beta, \alpha \wedge \gamma \rangle
\qquant \gamma \in \Ext V.
\]
The map $(\alpha,\beta) \mapsto \alpha \vee \beta$ is sesquilinear (conjugate-linear in the first variable), 
and moreover
\[
\langle \alpha \wedge \beta, \alpha \vee \gamma \rangle = \langle \alpha \wedge \alpha \wedge \beta, \gamma \rangle = 0,
\]
so
\[
|\alpha \wedge \beta + \alpha \vee \gamma|^2 = |\alpha \wedge \beta|^2 + |\alpha \vee \gamma|^2.
\]
Suppose now $\alpha \in \Ext^1 V = V$; then we set $\alpha = |\alpha| v_1$ and extend $v_1$ to an orthonormal basis $\{v_1,\dots,v_n\}$ of $V$.
If $\beta = \sum_J b_J v_J$ for some $b_J \in \C$, where $J$ ranges over the subsets of $\{1,\dots,n\}$, then
\[
\alpha \wedge \beta = |\alpha| \sum_{J \not\ni 1} b_J v_{J \cup \{1\}}
\quad\text{and}\quad
\alpha \vee \beta = |\alpha| \sum_{J \not\ni 1} b_{J \cup \{1\}} v_J,
\]
hence
\begin{equation}\label{eq:prodottoscalarevettore}
|\alpha \wedge \beta - \alpha \vee \beta|^2 = |\alpha \wedge \beta|^2 + |\alpha \vee \beta|^2 = |\alpha|^2 |\beta|^2.
\end{equation}

\subsection{Riemannian manifolds}

Let $M$ be an $n$-dimensional manifold.
The exterior algebra $\Ext M$ over the complexified cotangent bundle $\C T^* M$ is the bundle $\bigoplus_{k=0}^m \Ext^k M$; its sections are known as differential forms.
In particular, $\Ext^0 M = \bundle{T}$, $\Ext^1 M = \C T^* M$, and the differential ${\D} \in \Diff_1(\Ext^0 M, \Ext^1 M)$ extends to the exterior derivative ${\D} \in \Diff_1(\Ext M, \Ext M)$, which satisfies ${\D}^2 = 0$ and
\begin{gather*}
{\D}\alpha \in C^\infty(\Lambda^{k+1} M)
\quad\text{and}\quad
{\D}(\alpha \wedge \beta) = {\D}\alpha \wedge \beta + (-1)^k \alpha \wedge {\D}\beta
\end{gather*}
for all $\alpha \in C^\infty(\Ext^k M)$ and all $\beta \in C^\infty(\Ext M)$.
Hence
\[
[{\D},m(h)] \alpha = {\D}(h\alpha) - h{\D}\alpha = {\D}h \wedge \alpha,
\]
for all $h \in C^\infty(\bundle{T})$ and $\alpha \in C^\infty(\Ext M)$, that is,
\[
\sigma_1({\D})(\xi) \beta = \xi \wedge \beta ,
\]
when $x \in M$, $\xi \in \C T^*_x M$, and $\beta \in \Ext_x M$.

Suppose now that $M$ is endowed with a riemannian metric $g$.
This defines a hermitean fibre inner product on $\C T^* M$, which in turn extends to a hermitean fibre inner product $\langle\cdot,\cdot\rangle$ on $\Ext M$.
The formal adjoint ${\D}^+$ of the exterior derivative ${\D}$ is then defined, and satisfies
\[
\sigma_1({\D}^+)(\xi) \beta = - \overline{\xi} \vee \beta
\]
where $x \in M$, $\xi \in \C T^*_x M$ and $\beta \in \Ext_x M$, by \eqref{eq:symbadj}.
We set $D = {\D} + {\D}^+$; then $D$ is formally self-adjoint and
\[
\sigma_1(D)(\xi) \beta = \xi \wedge \beta - \overline{\xi} \vee \beta,
\]
so, when $\xi = \overline{\xi} \in T^*_x M$ is real,
\[
|\sigma_1(D)(\xi) \beta| = |\xi| |\beta|,
\]
by \eqref{eq:prodottoscalarevettore}, and
\[
P_D(\xi) = |\sigma_1(D)(\xi)|_\op = |\xi|.
\]
Thus the control distance function $\dist_D$ associated to $D$ is just the riemannian distance function $\dist_g$ on $M$.

Define $\varDelta = D^2 = {\D}{\D}^+ + {\D}^+ {\D}$.
This is the Laplace operator on forms induced by the riemannian structure.
Hence, according to \S~\ref{subsection:secondorder}, when $(M,g)$ is complete, the riemannian distance also describes the propagation of the solution $u_t$ of the second-order equation $\ddot u_t = -\varDelta u_t$ given by
\[
u_t = \cos(t \varDelta^{1/2}) u_0 + t \sinc(t \varDelta^{1/2}) \dot u_0.
\]
Since $\varDelta$ preserves the degree of forms, such a solution $u_t$ is a $k$-form for all $t \in \R$ whenever the initial data $u_0$ and $\dot u_0$ are both $k$-forms.

\subsection{Hermitean complex manifolds}

Suppose now that $M$ is a complex manifold of real dimension $2n$.
The decomposition
\[
\C T^* M = \Ext^{1,0} M \oplus \Ext^{0,1} M
\]
given by the complex structure in turn induces a decomposition of $\Ext^k M$, namely,
\[
\Ext^k M = \bigoplus_{p+q = k} \Ext^{p,q} M;
\]
then $\bigoplus \Ext^{p,q} M$ is an algebra bigrading of $\Ext M$.
Let $\pi_{p,q} \in \Hom(\Ext M, \Ext M)$ denote the projection onto $\Ext^{p,q} M$.
The exterior derivative ${\D}$ decomposes as $\partial + \dbar$, where
\[
\partial \alpha = \pi_{p+1,q} {\D}\alpha
\quad\text{and}\quad
\dbar\alpha = \pi_{p,q+1} {\D}\alpha
\qquant \alpha \in C^\infty(\Ext^{p,q} M);
\]
then $\partial^2 = \partial \dbar + \dbar \partial = \dbar^2 = 0$ and
\[
\partial(\alpha \wedge \beta) = \partial \alpha \wedge \beta + (-1)^k \alpha \wedge \partial \beta
\quad\text{and}\quad
\dbar(\alpha \wedge \beta) = \dbar \alpha \wedge \beta + (-1)^k \alpha \wedge \dbar \beta
\]
for all $\alpha \in C^\infty(\Ext^k M)$, $\beta \in C^\infty(\Ext M)$.
As before,
\[
[\partial, m(h)] \alpha = \partial h \wedge \alpha
\quad\text{and}\quad
[\dbar, m(h)] \alpha = \dbar h \wedge \alpha,
\]
so
\[
\sigma_1(\partial)(\xi)\beta = \pi_{1,0} \xi \wedge \beta
\quad\text{and}\quad
\sigma_1(\dbar)(\xi)\beta = \pi_{0,1} \xi \wedge \beta.
\]
For any choice of riemannian metric $g$ on $M$,
\[
\sigma_1(\partial^+)(\xi)\beta = -\pi_{1,0} \overline{\xi} \vee \beta
\quad\text{and}\quad
\sigma_1(\dbar^+)(\xi)\beta = -\pi_{0,1} \overline{\xi} \vee \beta,
\]
by \eqref{eq:symbadj}, hence also, when $\xi = \overline{\xi}$, that is, $\xi$ is real,
\[
|\sigma_1(\partial + \partial^+)(\xi)\beta| = |\pi_{1,0} \xi| |\beta|
\quad\text{and}\quad
|\sigma_1(\dbar + \dbar^+)(\xi)\beta| = |\pi_{0,1} \xi| |\beta|
\]
by \eqref{eq:prodottoscalarevettore}.
In particular, if $g$ is compatible with the complex structure (that is, the complex structure $J\colon T_x M \to T_x M$ is an isometry for every $x \in M$), then for real $\xi$,
\[
|\pi_{1,0} \xi|^2 = |\pi_{0,1} \xi|^2 = |\xi|^2/2 ,
\]
so the distance functions associated to $\partial + \partial^+$ and $\dbar + \dbar^+$ coincide with the riemannian distance function on $M$ multiplied by $\sqrt{2}$ (that is, the propagation speed with respect to the riemannian distance is at most $1/\sqrt{2}$).
The complex Laplacian $\square$ on forms is given by
\[
\square = (\dbar + \dbar^+)^2 = \dbar \dbar^+ + \dbar^+ \dbar;
\]
when $M$ is a K\"ahler manifold, $\varDelta = 2 \square$, which is consistent with the result already obtained for $\varDelta$.

See \cite{wells_differential_2008,folland_neumann_1972} for more on the material in this subsection.

\subsection{CR manifolds}

Let $M$ be an $n$-dimensional manifold endowed with a CR structure of codimension $n-2k$, that is, an involutive complex subbundle $\bundle{L}$ of $\C TM$ of rank $k$ such that $\bundle{L}_x \cap \overline{\bundle{L}}_x = \{0\}$ for all $x$ in $M$.
The exterior algebra $\Ext^{0,\bullet} M = \Ext (\overline{\bundle{L}}^*)$ over the dual of $\overline{\bundle{L}}$ may be identified with the quotient of $\Ext M$ by a suitable graded fibre ideal $\bundle{I}$.
Correspondingly $C^\infty(\Ext^{0,\bullet} M)$ may be identified with $C^\infty(\Ext M) / C^\infty(\bundle{I})$.
The exterior derivative ${\D}$ passes to the quotient bundle, giving a differential operator $\dbar_b \in \Diff_1(\Ext^{0,\bullet} M,\Ext^{0,\bullet} M)$ that satisfies
\begin{gather*}
\dbar_b ^2 = 0 \\
\dbar_b \alpha \in C^\infty(\Ext^{0,q+1} M) \\
\dbar_b (\alpha \wedge \beta) = \dbar_b \alpha \wedge \beta + (-1)^q \alpha \wedge \dbar_b \beta
\end{gather*}
for all $\alpha \in C^\infty(\Ext^{0,q} M)$ and all $\beta \in C^\infty(\Ext^{0,\bullet} M)$.

Note that $\Ext^{0,0} M = \Ext^0 M = \bundle{T}$, so $\dbar_b f = \pi {\D}f$, where $\pi\colon \C T^* M \to \overline{\bundle{L}}^*$ is the restriction morphism.
Thus
\[
[\dbar_b, m(h)] \alpha = \dbar_b h \wedge \alpha 
\quad\text{and}\quad
\sigma_1(\dbar_b)(\xi) \beta = \pi\xi \wedge \beta.
\]
Any choice of hermitean fibre inner product on $\overline{\bundle{L}}$ induces a hermitean inner product along the fibres of $\Ext^{0,\bullet} M$, and
\[
\sigma_1(\dbar_b^+)(\xi) \beta = - \pi{\overline\xi} \vee \beta,
\]
so again, for all real $\xi$,
\[
|\sigma_1(\dbar_b + \dbar_b^+)(\xi) \beta| = |\pi\xi| |\beta|,
\]
that is, if $D = \dbar_b + \dbar_b^+$, then
\[
P_D(\xi) = |\pi \xi|.
\]

If $\xi \in T^* M$, then $\pi \xi = 0$ if and only if $\xi$ vanishes on $TM \cap (\bundle{L} \oplus \overline{\bundle{L}})$; in other words, the Levi distribution $TM \cap (\bundle{L} \oplus \overline{\bundle{L}})$ is the subbundle spanned by the $D$-subunit vectors.
In particular, if $M$ is a nondegenerate CR manifold and $n=2k+1$, then $P_D$ satisfies H\"ormander's condition; for a discussion of the higher-codimensional case, see, for example, \cite[Section 12.1]{boggess_cr_1991}.

Note moreover that the Kohn Laplacian $\square_b$ on the tangential Cauchy--Riemann complex is given by
\[
\square_b = D^2 = \dbar_b \dbar_b^+ + \dbar_b^+ \dbar_b .
\]

For more information on CR manifolds, see, for example, \cite{boggess_cr_1991,dragomir_differential_2006}.

\subsection{Subriemannian structures}

Let $E$ be a real vector bundle on $M$, endowed with a fibre inner product and a smooth bundle homomorphism $\mu\colon E \to TM$.
Consider the adjoint morphism $\mu^*\colon T^* M \to E^*$, and its complexification $\mu^*\colon \C T^* M \to \C E^*$.
Define the differential operator $D \in \Diff_1(\bundle{T},\C E^*)$ by $D f = \mu^*({\D}f)$.
Then $\Symbol{D} = D$, modulo the identification $\Hom(\bundle{T},\C E^*)  = \C E^*$; further $P_D(\xi) = |\mu^*(\xi)|$, $P_D^*(v) = \inf \{ |w| \tc v = \mu(w) \}$, and the $D$-subunit vectors are the images under $\mu$ of the $w \in E$ such that $|w| \leq 1$.

A commonly considered case is when $E$ is a subbundle of $TM$ and $\mu$ is the inclusion map.
Then $E$ is called the horizontal distribution \cite[Section 1.4]{montgomery_tour_2002}, and is the set of the tangent vectors $v$ for which $P_D^*(v) < \infty$.

Another commonly considered case \cite{jerison_subelliptic_1987,garofalo_lipschitz_1998} is when $E$ is the trivial bundle $\bundle{T}^r$ with the standard inner product.
In this case, there are (subunit) vector fields $X_j = \mu(Y_j)$, where the $Y_j$ are the constant sections of $E$ corresponding to the standard basis of $\R^r$.
Hence
\[
P_D(\xi)^2 = \sum_j |\mu^*(\xi)(Y_j)|^2 = \sum_j |\xi(X_j)|^2,
\]
so
\[
|\Symbol{D} f|_\op^2 = |D f|^2 = \sum_j |X_j f|^2,
\]
and
\[
P_D^*(v)^2 = \inf \Bigl\{ \sum_j c_j^2 \tc v = \sum_j c_j X_j|_x \Bigr\}
\qquant v \in T_x M.
\]

\subsection{Nonriemannian propagation}

The fibre seminorm $P_D$ on $T^*M$ associated to $D \in \Diff_1(\bundle{E},\bundle{F})$ is defined to be the pullback of an operator norm along the fibres of $\Hom(\bundle{E},\bundle{F})$.
In the previous examples, however, $P_D$ is actually induced by some (possibly degenerate) inner product on $T^* M$.
We present now a simple example showing that this is not always the case.

Let $M$ be $\R^n$, take $\bundle{E} = \bundle{F} = \bundle{T}^n$, and define $D$ by
\[
D (f_1,\dots,f_n) = (\I \partial_1 f_1,\dots,\I \partial_n f_n),
\]
where $\partial_1,\dots,\partial_n$ are the partial derivatives on $\R^n$.
Then
\[
\sigma_1(D)(\xi) = \begin{pmatrix}\I \xi_1 & & \\ & \ddots & \\ & &\I \xi_n \end{pmatrix},
\]
so $P_D(\xi) = |\xi|_\infty$ and $P_D^*(v) = |v|_1$; here, as usual, $|\xi|_\infty = \max_j |\xi_j|$ and $|v|_1 = \sum_j |v_j|$.
Consequently, $\dist_D(x,y) = |x-y|_1$, hence $\dist_D$ is varietal and $D$ is complete, therefore $D$ is essentially self-adjoint, and
\[
\E^{\I tD} (f_1,\dots,f_n)(x) = (f_1(x_1-t,x_2,\dots,x_n),\dots,f_n(x_1,\dots,x_{n-1},x_n-t)).
\]
Hence the condition $\supp (\E^{\I tD} f) \subseteq \bar{B}_D(\supp f,|t|)$ given by Theorem~\ref{thm:propagationdistance} is optimal.
This shows that the natural distance describing the propagation of solutions of \eqref{eq:onde} need not be riemannian or even subriemannian.

\subsection{Nonsmooth arc-length reparametrisation}\label{subsection:nonsmootharclength}

\runinhead{Construction of the subriemannian structure.}

Take $M = \R^2$ with Lebesgue measure.

Fix a smooth ${{u}}\colon \R^2 \to \R$.
Define the smooth vector fields $X,Y$ on $\R^2$ by
\begin{align*}
X|_p &= \frac{2}{\sqrt{4+3{{u}}(p)^2}} \left( \frac{\partial}{\partial x} + {{u}}(p) \frac{\sqrt{3}}{2} \frac{\partial}{\partial y} \right), \\
Y|_p &= \frac{{{u}}(p)}{2 \sqrt{1+{{u}}(p)^2}} \frac{2}{\sqrt{4+3{{u}}(p)^2}} \left( - {{u}}(p) \frac{\sqrt{3}}{2} \frac{\partial}{\partial x} + \frac{\partial}{\partial y} \right),
\end{align*}
where $\{\sfrac{\partial}{\partial x}, \sfrac{\partial}{\partial y}\}$ denotes the standard basis of $\R^2$.
With respect to the standard riemannian (that is, euclidean) structure of $\R^2$,
\[
\langle X, Y \rangle = 0, \qquad |X| = 1,
\quad\text{and}\quad
|Y| = \frac{|{{u}}|}{2 \sqrt{1+{{u}}^2}}
\]
at every point of $\R^2$.
Indeed, if we define the ``matrix field'' $M$ by
\[
M|_p = \frac{1}{\sqrt{4+3{{u}}(p)^2}}
\begin{pmatrix}
2                            & - {{u}}(p) \sqrt{3}\\
{{u}}(p) \sqrt{3} & 2
\end{pmatrix},
\]
then $M$ is pointwise orthogonal and
\[
X = M \frac{\partial}{\partial x}, \qquad Y = \frac{{{u}}}{2 \sqrt{1+{{u}}^2}} M \frac{\partial}{\partial y}
\]
pointwise.

Define the differential operator $D \in \Diff_1(\bundle{T},\bundle{T}^2)$ by $Df = (X f , Y f)$.
Then
\[
\sigma_1(D)|_p(\xi) = (\langle X|_p, \xi \rangle , \langle Y|_p, \xi \rangle ),
\]
hence the associated fibre seminorm on the cotangent bundle $T^* \R^2$ is given by
\[
\begin{aligned}
P_D|_p(\xi)^2 
&= |\sigma_1(D)|_p(\xi)|_\op^2 = \langle X|_p, \xi \rangle^2 + \langle Y|_p, \xi \rangle^2 \\
&= \left\langle \xi , H|_p \, \xi \right\rangle,
\end{aligned}
\]
where
\[
H
= \frac{1}{4(1+{{u}}^2)} \begin{pmatrix} 4 + {{u}}^2 & 2\sqrt{3}{{u}} \\ 2\sqrt{3}{{u}} & 4{{u}}^2 \end{pmatrix}.
\]

On the one hand, at points $p$ where ${{u}}(p) \neq 0$, the matrix $H|_p$ is nondegenerate; in this case, the norm $P_D^*$ on the tangent bundle is given by
\[
P_D|_p(v)^2 = \langle v, H|_p^{-1} v \rangle,
\]
where
\[
H^{-1}
= \frac{1}{{{u}}^2} \begin{pmatrix} 4{{u}}^2 & - 2 \sqrt{3}{{u}} \\ - 2 \sqrt{3}{{u}} & 4+{{u}}^2 \end{pmatrix},
\]
and $\{X|_p, Y|_p\}$ is an orthonormal basis for the corresponding inner product on $T_p \R^2$.

On the other hand, at points $p$ where ${{u}}(p) = 0$,
\[
P_D|_p(\xi) = \left| \left\langle \frac{\partial}{\partial x}, \xi \right\rangle\right|,
\]
hence $P^*_D$ is the extended norm
\[
P_D^*|_p(v) = \begin{cases} \left| \left\langle \sfrac{\partial}{\partial x}, v \right\rangle\right| &\text{if $\left\langle \sfrac{\partial}{\partial y}, v \right\rangle = 0$,} \\ \infty & \text{otherwise,} \end{cases}
\]
and $X|_p = \sfrac{\partial}{\partial x}$ and $Y|_p = 0$.
In particular,
\[
P_D^*|_p\left(\frac{\partial}{\partial x} \right)
= \begin{cases}
2 &\text{if ${{u}}(p) \neq 0$,}\\
1 &\text{if ${{u}}(p) = 0$.}
\end{cases}
\]

\runinhead{A choice of ${{u}}$.}

Let $\Q = \{q_m\}_{m \in \N}$ be an enumeration of the rational numbers, and set
\[
 A = \bigcup_{m \in \N} \leftopenint q_m - 2^{-m-3}, q_m + 2^{-m-3} \rightopenint.
\]
Then $A$ is a dense open subset of $\R$ whose measure $|A|$ is at most $\sum_{m = 0}^\infty 2^{-m-2}$, that is, $1/2$.

Since $\R \setminus A$ is closed in $\R$, there exists a smooth function ${{v}}\colon \R \to \leftclosedint 0, 1 \rightclosedint$ such that ${{v}}^{-1}(0) = \R \setminus A$ \cite[Theorem~1.5]{kahn_global_1980}.
In fact, after composing ${{v}}$ with a smooth function from $\R$ to $[0,1]$ that vanishes exactly on $\leftopenint-\infty,0\rightclosedint$, we may suppose that ${{v}}$ vanishes to infinite order at all points of $\R \setminus A$.
Set then ${{u}} ( x , y) = {{v}}(x)$.

\runinhead{H\"ormander's condition.}

Let $Z$ be a $D$-subunit field.
Then $Z = \phi X + \psi Y$ for some real-valued functions $\phi,\psi$ with $\phi^2 + \psi^2 = 1$.
Since $\langle Z, X \rangle = \phi$, we see that $\phi$ is smooth, so $\phi X$ and $\psi Y$ are smooth too.
Moreover, since $|\psi| \leq 1$, the smooth field $\psi Y$ vanishes at least to the same order as $Y$, at every point of $\R^2$, and hence $\psi Y$ vanishes to infinite order at every point of $(\R \setminus A) \times \R$.

Take now a system $Z_1,\dots,Z_r$ of $D$-subunit vector fields, and decompose $Z_j$ as $\phi_j X + \psi_j Y$.
Then any iterated Lie bracket of $Z_1,\dots,Z_r$ is the sum of an iterated Lie bracket of $\phi_1 X,\dots,\phi_r X$ and of iterated Lie brackets where some of the $\psi_j Y$ occur.
The first summand is then a smooth multiple of $X$, whereas the other summands vanish to infinite order at every point of $(\R \setminus A) \times A$ (indeed, the set of smooth vector fields vanishing to infinite order at some $p \in M$ is an ideal of the Lie algebra of smooth vector fields).
We conclude that the iterated Lie bracket of $Z_1,\dots,Z_r$, evaluated at any point of $(\R \setminus A) \times \R$, is a multiple of $X$.

Hence H\"ormander's condition for $P_D$ fails at all points of $(\R \setminus A) \times \R$.

\runinhead{Topologies.}

Define $Z = 2^{-1} \sfrac{\partial}{\partial x}$ and $W = {{u}}(\sqrt{4+{{u}}^2})^{-1} \sfrac{\partial}{\partial y}$, and then set $\mathfrak{X} = \{Z,W\}$.
Then $\mathfrak{X}$ is a system of smooth $D$-subunit vector fields on $\R^2$.
Write $\dist_\mathfrak{X}$ for the distance function corresponding to the class of $D$-subunit curves that are piecewise flow curves of $Z$ or $W$.
Clearly
\[
\dist_D \leq \dist_D^\infty \leq \dist_D^\flow \leq \dist_\mathfrak{X}.
\]
We now show that $\dist_\mathfrak{X}$ is varietal, so all the other distance functions above are.

Take $(x,y) \in \R^2$ and $r \in \R^+$.
We want to prove that $\bar{B}_{\mathfrak{X}}((x,y),r)$ is a neighbourhood of $(x,y)$.
Since $A$ is dense in $\R$, there is $x' \in A$ such that $|x-x'| < r/8$ and ${{v}}(x') \neq 0$.
We claim that every point $( \tilde x , \tilde y ) \in \R^2$ such that
\[
|(\tilde x,\tilde y) - (x,y)|_\infty < \min\left\{ \frac{r}{8}, \frac{r}{4} \frac{|{{v}}(x')|}{\sqrt{4+{{v}}(x')^2}} \right\}
\]
belongs to $\bar{B}_\mathfrak{X}((x,y),r)$.
The idea is to go from $(x , y)$ to $(x' , y)$ along the flow of $Z$, then from $(x' , y)$ to $(x' , \tilde y)$ along the flow of $W$, and finally from $(x' , \tilde y)$ to $(\tilde x , \tilde y)$ along the flow of $Z$.
Such a curve is defined on an interval of length
\[
2 |x-x'| + \frac{\sqrt{4+{{v}}(x')^2}}{|{{v}}(x')|} |y-\tilde y| + 2|x' - \tilde x| < \frac{r}{4} + \frac{r}{4} + \frac{r}{2} = r,
\]
hence its final point $(\tilde x,\tilde y)$ belongs to the ball $\bar{B}_\mathfrak{X}((x,y),r)$.

\runinhead{A smooth curve with nonsmooth arc-length.}

Let $\phi\colon \leftclosedint 0,T\rightclosedint \to \R$ be absolutely continuous, and define $\gamma(t) = ( \phi(t) , 0 )$.
Then
\[
P_D^*|_{\gamma(t)}(\gamma'(t))
= \begin{cases}
2 |\phi'(t)| &\text{if $\phi(t) \in A$,}\\
|\phi'(t)| &\text{if $\phi(t) \notin A$,}\\
\end{cases}
\]
at every point $t$ where $\phi$ is differentiable.

The set $\tilde A = \phi^{-1}(A)$ is open in $\leftclosedint 0, T \rightclosedint$, but need not be dense.
However, any connected subset of $\leftclosedint 0,T \rightclosedint \setminus \tilde A$ is mapped by $\phi$ onto a connected subset of $\R \setminus A$, which has at most one element because $A$ is dense.
Hence $\phi$ is locally constant on the interior of $\leftclosedint 0, T \rightclosedint \setminus \tilde A$, so $\phi'(t) = 0$ for every interior point $t$ of $\leftclosedint 0, T \rightclosedint \setminus \tilde A$.
The remaining points of $\leftclosedint 0, T \rightclosedint \setminus \tilde A$, that is, the boundary points, also belong to the closure of $\tilde A$.

Suppose now that $\gamma$ is $D$-subunit and $C^1$.
Then $|\phi'(t)| \leq 1/2$ for all $t \in \tilde A$.
Since $\phi'$ is continuous, $|\phi'| \leq 1/2$ on the closure of $\tilde A$, hence on all $\leftclosedint 0, T \rightclosedint$.
This means that
\[
|\phi(T) - \phi(0)| \leq \int_0^T |\phi'(t)| \,{\D}t \leq T/2,
\]
that is, $T \geq 2 |\phi(T) - \phi(0)|$.

Suppose further that $\phi$ is nondecreasing, $\phi(0) = 0$ and $\phi(T) = 1$, so $T \geq 2$.
The length $\ell_D(\gamma)$ does not depend on the parametrisation.
Therefore, if we define $\tilde\gamma(t) = ( t , 0 )$ for all $t \in \leftclosedint 0, 1 \rightclosedint$, then
\[
\begin{aligned}
\ell_D(\gamma) 
 = \ell_D(\tilde\gamma)
&= \int_0^1 P_D^*|_{\tilde\gamma(t)}(\tilde\gamma'(t)) \,{\D}t \\
&= 2 \left|\leftclosedint 0,1 \rightclosedint \cap A\right| + \left|\leftclosedint 0,1 \rightclosedint \setminus A\right| \\
& = 1 + \left|\leftclosedint 0,1 \rightclosedint \cap A\right| \leq 3/2 < 2.
\end{aligned}
\]

In summary, every $D$-subunit, $C^1$ reparametrisation of $\gamma$ is defined on an interval of width at least $2$.
By contrast, the arc-length reparametrisation of $\gamma$ is $D$-subunit and defined on an interval of width at most $3/2$.

\bibliographystyle{amsabbrv}
\bibliography{propagation}

\end{document}